%% file: main_OptimalFlow_arXiv.tex
\documentclass[letterpaper, 11pt, journal, onecolumn]{ieeeconf}
\IEEEoverridecommandlockouts
\usepackage{amsmath,amssymb,amsfonts,theorem}
\usepackage{graphicx,color}

\usepackage{tikz}
\usetikzlibrary{matrix}
\usetikzlibrary{positioning}
\usetikzlibrary{arrows,shapes, shadows, calc}
\usetikzlibrary{spy}
\usepackage{pgfplots}
\pgfplotsset{compat=1.3}

\usepackage{caption}
\usepackage{subcaption}

{ \theorembodyfont{\normalfont} 

\newtheorem{remark}{Remark}
}
\newtheorem{assumption}{Assumption}
\newtheorem{definition}{Definition}
\newtheorem{theorem}{Theorem}

\newtheorem{proposition}{Proposition}

\def\QEDclosed{\mbox{\rule[0pt]{1.3ex}{1.3ex}}} 
\def\QEDopen{{\setlength{\fboxsep}{0pt}\setlength{\fboxrule}{0.2pt}\fbox{\rule[0pt]{0pt}{1.3ex}\rule[0pt]{1.3ex}{0pt}}}}
\def\QED{\QEDopen}

\newcommand{\D}{\mathrm{D}}

\newcommand{\1}{\mathbf{1}} 

\def\be{\begin{equation}}
\def\ee{\end{equation}}
\def\ba{\begin{array}}
\def\ea{\end{array}}
\def\eqa{\begin{eqnarray}}
\def\eqe{\end{eqnarray}}

\newcommand{\mc}{\mathcal}

\graphicspath{{figs/}}

\begin{document}

%
%
%
%

%

\title{Optimal pricing control in distribution networks with time-varying supply and demand }

\author{Mathias B{\"u}rger, Claudio De Persis, and Frank Allg{\"o}wer
\thanks{M. B{\"u}rger and F. Allg{\"o}wer  thank the German Research Foundation (DFG) for financial support within the Cluster of Excellence in Simulation Technology (EXC 310/2) at the University of Stuttgart.}
\thanks{The work of C. De Persis is supported by the research grants Efficient Distribution
of Green Energy (Danish Research Council of Strategic Research),
Flexiheat (Ministerie van Economische Zaken, Landbouw en Innovatie), and
by a starting grant of the Faculty of Mathematics and Natural Sciences,
University of Groningen.}
\thanks{
M. B{\"u}rger and Frank Allg{\"o}wer are with the Institute for Systems Theory and Automatic Control, University of Stuttgart, Pfaffenwaldring 9, 70550 Stuttgart, Germany, \texttt{ \{mathias.buerger, frank.allgower\}@ist.uni-stuttgart.de}. 
}
\thanks{C. De Persis is with the Faculty of Mathematics and Natural Sciences, University of Groningen, Nijenborgh 4,  9747 AG Groningen, the Netherlands, \texttt{ \{c.de.persis@rug.nl\}}}}

\maketitle

\begin{abstract}
This paper studies the problem of optimal flow control in dynamic inventory systems.
A dynamic optimal distribution problem, including time-varying supply and demand, capacity constraints on the transportation lines, and convex flow cost functions of Legendre-type,  is formalized and solved.
The time-varying optimal flow is characterized in terms of the time-varying dual variables of a corresponding network optimization problem.
A dynamic feedback controller is proposed that regulates the flows asymptotically to the optimal flows and achieves in addition a balancing of all storage levels.

\end{abstract}

\section{Introduction}
Distribution networks are systems composed of several inventories where goods/materials/products are stored, and transportation lines along which the goods are transported. 
The control problem in distribution networks is to assign the flows to the transportation lines such that an external supply of goods, arriving at some of the inventories, is distributed to inventories in such a way that a requested demand can always be satisfied. 
The static problem of distributing a constant flow to transportation lines according to some cost functions has been at the heart of the mathematical theory called \emph{network optimization} \cite{Ford1956}, \cite{Bertsekas1998}, \cite{Rockafellar1998}.
However, real networks must react dynamically on changes in the external conditions, such as time-varying supply and demand.
If the networks operate on a fast time scale 
continuous feedback controllers are required to adapt the flow instantaneously to external changes.
Beyond that, feedback controllers also promise to provide desirable robustness properties to the network.
In addition to achieving an optimal routing, those controllers should at best also distribute the stored goods equally to all inventories.

Feedback control of inventory systems has recently found significant attention.
The control of distribution networks with unknown supply and demand has been investigated in \cite{Bauso2006}, using a combination of robust optimization and control theory. A control strategy is derived that ensures robust feasibility as well as a min-max type optimality.
Robust control strategies for distribution networks have been considered, e.g., in \cite{Blanchini2001}, \cite{Baric2012}.
The aspect of balancing in distribution networks has been studied in \cite{Danielson2013} using a one-step model predictive control scheme.
In  \cite{vanderSchaft2012},  \cite{Wei2013}, 
it was shown that the balancing problem can be formulated in a Hamiltonian framework. 

Network flow and balancing problems are intimately related to consensus or output agreement problems. In \cite{Burger2013a}, \cite{Burger2013} we showed that networks of passive systems with constant external disturbances solve asymptotically certain network flow problems. Furthermore, the phenomena of output agreement and clustering are mathematically related to optimal network flow problems  \cite{Burger2011}.
The connection between output agreement and network flow control has been further studied from an internal model control perspective in \cite{DePersis2013}, \cite{Burger2013b}, \cite{Burger2013c}.
Considering demands and supplies that are generated by an external dynamical system (exo-system), distributed dynamic controllers were proposed that ensure an exact routing as well as a balancing of the inventory levels. In \cite{Burger2013c}, also controllers were proposed that route the flow asymptotically optimal through the network, provided the underlying cost functions are quadratic. 

We address here the problem of balancing and optimal routing  in time-varying distribution networks with convex routing costs and capacity constraints on the transportation lines. 
The time-varying optimal flow problem is formulated, assuming that the supply/demand is generated by an exo-system. 
It is shown that for a certain class of convex cost functions, namely \emph{Legendre-type} functions, the time-varying optimal flow can be characterized in terms of the time-varying dual solutions. 
In particular, an explicit differential equation for the dual solutions is derived. The construction of this differential equation requires to find the analytic solution to a \emph{generalized Dirichlet problem} \cite{Biggs1997}.
Based on this characterization of the time-varying dual variables, a feedback controller is proposed that achieves a balancing of the storage levels and routes the flow asymptotically optimal through the network. 

The remainder of the paper is organized as follows. 
The inventory system model and the relevant assumptions on the supply are introduced in Section \ref{sec.OptimalDistribution}.
The time-varying optimal flow problem is presented in Section \ref{sec.OptRouting}, where also a differential equation for the optimal dual variables is derived.
A feedback controller, ensuring balancing and asymptotic optimal routing is proposed in Section \ref{sec.Feedback}. 
The design procedure is illustrated on an example system in Section \ref{sec.Example}. \\

\textbf{Preliminaries:}
The notation we employ is standard. The set of (non negative) real numbers is denoted by $\mathbb{R}$ ($\mathbb{R}_{\geq}$). The distance of a point $q$ from a set $\mathcal{A}$ is defined as $\mbox{dist}_{\mc{A}} q = \mbox{inf}_{p \in \mc{A}} \|p - q\|$.
Given a matrix $B$ (vector $x$), the notation $[B]_{ik}$ ($[x]_i$) denotes its $ik$-th ($i$-th) entry.
The range-space and null-space of a matrix $B$ are denoted by $\mc{R}(B)$ and $\mc{N}(B)$, respectively.
A graph $\mc{G}=(V,E)$ is an object consisting of a finite set of nodes, $|V| = n$, and edges, $|E| = m$.
Throughout the paper we assume that a graph $\mc{G}$ is assigned an (arbitrary) orientation.
The incidence matrix $B \in \mathbb{R}^{n \times m}$ of the
graph $\mc{G}$, is a $\{0, \pm 1\}$ matrix with $[B]_{ik}$ having value
`+1' if node $i$ is the initial node of edge $k$, `-1' if it is the terminal
node, and `0' otherwise. 
%

For a convex function $\mc{P}: \mathbb{R}^{p} \rightarrow \mathbb{R}$, the \emph{effective domain} is $\mathrm{dom} \mc{P}  = \{ s: \mc{P}(s) < \infty\}$. 
The convex conjugate of $\mc{P}(s)$ is $\mc{P}^{\star}(r) = \sup_{s} \{r^{T}s - \mc{P}(s) \; : \; s \in \mathrm{dom} \mc{P} \; \}$.

%
%
%
%

\section{Time-varying distribution problems} \label{sec.OptimalDistribution}

The objective of this paper is to design distribution control laws for transportation networks with storage that achieve two objectives: (i) balancing of the storage levels, and (ii) an optimal routing of the flows trough the network.
Consider an inventory network with $n$ inventories and $m$ transportation lines, described by a graph $\mc{G}=(V,E)$, and let $B$ be the incidence matrix (with arbitrary orientation). The dynamics of the inventory system is given as
\begin{align}
\begin{split} \label{sys.Inventory1}
\dot{x}(t) &= B\lambda(t) + Pw(t),
\end{split}
\end{align}
where $x \in \mathbb{R}^{n}$ represents the storage level, $\lambda \in \mathbb{R}^{m}$ the flow along the lines, and $Pw$ an external in-/outflow of the inventories, i.e., the supply or demand. 
This basic model (or its discrete time version) is studied  in \cite{Bauso2006}, \cite{vanderSchaft2012}, \cite{Baric2012}, \cite{Danielson2013}, \cite{DePersis2013} and is a relevant model for supply chains (\cite{Alessandri2011}) or data networks (\cite{Moss1983}).

We assume throughout this paper that the supply and demand are varying over time and in particular that they as well as their derivatives are known, that is, all information about the evolution of the flow is known. In compliance with the control theoretic approach of this paper, we assume that the supply/demand is generated by a dynamical system
\begin{align}\label{sys.Inventory_w} 
\dot{w}(t) = s(w(t)), \quad w \in \mathbb{R}^{r},\end{align}
with initial condition $w(0) = w_{0}$.
\begin{assumption} \label{ass.Boundedness}
There exists a compact set $\mc{W}$ that is forward invariant under \eqref{sys.Inventory_w} and the initial conditions $w(0)$ are contained in this set $\mc{W}$.
\end{assumption}
Note that this assumption ensures the boundedness of the supply and demand vector.
Throughout this paper we restrict our attention to networks with balanced supply and demand.
\begin{assumption}\label{ass.BalancedFlow}
The supply/demand of the network is balanced, that is 
\begin{align*}
\1_{n}^{T}Pw(t) = 0 \;\mbox{{\rm for all }} t \geq 0.
\end{align*}
\end{assumption}
We assume that the flow on each line can flow in both directions, but is restricted to satisfy certain capacity bounds. That is, we assume that for each line $k \in \{1,\ldots,m\}$ there exists $\bar{\lambda}_{k} \in \mathbb{R}_{>0}$ and the flow must satisfy
\begin{align} \label{eqn.CapacityConstraints}
- \bar{\lambda}_{k} \leq \lambda_{k}(t) \leq \bar{\lambda}_{k} \quad \mbox{for\;all\;} t \geq 0.
\end{align}
If the flow on one line is unconstrained, we let $\bar{\lambda}_{k}$ be infinity. In the following, we will also write $- \bar{\lambda} \leq \lambda \leq \bar{\lambda}$, meaning that the inequalities hold for each component.

A \emph{cut} of the network is defined to be a signed edge set\footnote{We assume that the direction of the edges is chosen in accordance with the definition of the incidence matrix $B$.} $Q \subset E$ (with positive and negative parts denoted by $Q^{+}$ and $Q^{-}$) such that, for some node set $S \subset V$ 
\begin{align}
\begin{split} \label{eqn.InducedCut}
Q^{+} = \{ k \in E \; :\; k = (i,j)\; \mbox{with}\; i \in S, j \in V \setminus S\} \\
Q^{-} = \{ k \in E \; :\; k = (j,i)\; \mbox{with}\; i \in S, j \in V \setminus S\}.
\end{split}
\end{align}
We will denote the cut induced by a set of nodes $S \subset V$ with $Q(S)$.
%
The \emph{capacity} of a cut is
\begin{align*}
\bar{\lambda}_{Q}= \sum_{k \in Q} \bar{\lambda}_{k}.
\end{align*} 

To ensure feasibility of the problem, we impose the following condition.
\begin{assumption} \label{ass.Cut}
There exists $\epsilon > 0$ such that for any set of nodes $S \subseteq V$ it holds that
\begin{align*}
-\bar{\lambda}_{Q(S)} + \epsilon \leq \sum_{i \in S} [Pw(t)]_{i}  \leq \bar{\lambda}_{Q(S)} - \epsilon.
\end{align*}
\end{assumption}
The assumption ensures that the ``in/out-flow'' to any set of nodes is strictly smaller than the capacity of the corresponding induced cut.
The following result summarizes the conditions that ensure feasibility of the distribution problem.

\begin{proposition}[Feasibility] \label{prop.Feasibility}
Consider the inventory system \eqref{sys.Inventory1} with a time-varying supply/demand generated by the exo-system \eqref{sys.Inventory_w} and capacity constraints \eqref{eqn.CapacityConstraints}. Suppose Assumptions \ref{ass.Boundedness}, \ref{ass.BalancedFlow} and \ref{ass.Cut} hold. Then,
(i) for any flow vector $\lambda \in \mathbb{R}^{m}$ the average inventory level  $x^{\mathrm{avg}}(t) = \frac{1}{n} \1_{n}^{T}x(t)$ remains constant, that is $x^{\mathrm{avg}}(t) = x^{\mathrm{avg}}(0)$ for all $t \geq 0$; and 
(ii) there exists a continuous flow trajectory $\lambda : \mathbb{R}_{\geq 0} \rightarrow \mathbb{R}^{m}$ that is strictly feasible, that means there is a vector valued function $\lambda : \mathbb{R}_{\geq 0} \rightarrow \mathbb{R}^{m}$ such that for each edge $k$ the bounds \eqref{eqn.CapacityConstraints} hold with strict inequality and $B \lambda(t) + Pw(t) = 0$. 
\end{proposition}

\begin{proof}
To see the first claim note that the average $x^{\mathrm{avg}}(t)$ evolves as 
$\dot{x}^{\mathrm{avg}} = \frac{1}{n}\1_{n}^{T}B\lambda+\frac{1}{n}\1_{n}^{T}Pw(t) = 0.$
To prove the second claim, we show first that at each fixed time $t$, there exists a feasible flow vector. To indicate that we consider a single time instant, we write in the following $w = w(t)$.
Assume without loss of generality that all transportation lines have a finite capacity, i.e., $\bar{\lambda}_{k} < \infty$.
Let $\Lambda = \mathrm{diag} \{ \bar{\lambda}_{1}, \ldots, \bar{\lambda}_{m}\}$.
Consider the linear optimization problem
\begin{align}
\begin{split}  \label{prob.LPPrimal}
\min_{\lambda} & \; \| \Lambda^{-1}\lambda\|_{\infty}\quad \mathrm{s.t.} \; B\lambda = - Pw.
\end{split}
\end{align}
The linear programming dual can be derived  as
\begin{align}
\begin{split} \label{prob.LPDual}
\max_{u} & \; (-Pw)^{T}u \quad \mathrm{s.t.} \;   \| \Lambda B^{T}u\|_{1} \leq 1.
\end{split}
\end{align}
 Note now that the extreme points of the polytope defined by $\| \Lambda B^{T}u\|_{1} \leq 1$ are multiples of the vectors whose elements are zeros and ones (see e.g., \cite{Strang1984}).
  Thus, each extreme point must be of the form $u = c_{Q} e_{Q}$, where $Q$ is a cut induced by some node set $S \subseteq V$ according to \eqref{eqn.InducedCut}, $c_{Q} \in \mathbb{R}$ is a constant such that $|c_{Q}|= \bar{\lambda}_{Q}^{-1}$, and $e_{Q} \in \mathbb{R}^{n}$ is a vector with entries being $[e_{Q}]_{i} =1$ if node $i \in S$ and $[e_{Q}]_{i} =0$ if $i \in V \setminus S$.
It follows now from Assumption \ref{ass.Cut} that \eqref{prob.LPDual} must always have an optimal value strictly smaller than one. From strong duality of linear programs follows that also \eqref{prob.LPPrimal} must have an optimal value strictly smaller than one. 
This shows that at each fixed time instant $t$, there exists a feasible flow vector $\lambda(t)$ that satisfies all constraints, i.e., $- \bar{\lambda} < \lambda(t) < \bar{\lambda}$. 

It remains to prove that there exists a continuous function $\lambda : \mathbb{R}_{\geq 0} \rightarrow \mathbb{R}^{m}$ that represents a feasible flow trajectory.
Consider now the set of feasible flows for a given supply/demand vector $w$, i.e.,
\[ \Delta(w) = \{ \lambda \in \mathbb{R}^{m} \; : \; B \lambda = - Pw, \, - \bar{\lambda} \leq \lambda \leq \bar{\lambda} \}.\]
The set $\Delta(w)$ is compact, convex and has a non-empty relative interior.
One can understand $\Delta$ as a \emph{set-valued map} from $\mathbb{R}^{r}$ to $\mathbb{R}^{m}$. 
As a set-valued map, $\Delta$ is \emph{lower semi-continuous}.\footnote{ The set valued function $\Delta$ is lower semi-continuous at a point $w \in \mathrm{dom}\Delta$ if and only if for any $\lambda \in \Delta(w)$ and any sequence $w_{k}$ converging to $w$, there exists a sequence $\lambda_{k} \in \Delta(w_{k})$ converging to $\lambda$. It is said to be lower semi-continuous if it is lower semi-continuous at every point $w \in \mathrm{dom} \Delta$. }
The existence of a continuous selection $\lambda(w)$ of $\Delta$ follows now from Micheal's Theorem (\cite[Theorem 9.1.2]{Aubin1990}).
Since additionally $w(t)$ is a continuous function of time (i.e., it respects the dynamics \eqref{sys.Inventory_w}), there exists a continuous function $\lambda(t)$  such that $\lambda(t) \in \Delta(w(t))$, i.e., there is a continuous feasible flow trajectory.
\end{proof}

 %

\section{Optimal Time-Varying Routing} \label{sec.OptRouting}

In the following, we assume that the routing of flow along transportation lines of the network incorporates a cost. We seek a control strategy to route the time-varying flow in such a way through the network that the routing cost is minimal at each time instant. We characterize next a class of feed-forward controllers that achieve such an optimal routing.

The flow cost on a transportation line is described by a convex $C^{\infty}$ cost function
 \[ \mc{P}_{k}(\lambda_{k}), \, k = 1,2,\ldots,m \]
  for the flow $\lambda_{k}$. 
  %
%
We aim to characterize the steady state flows that are such that at each time instant the  following \emph{static optimal distribution problem} is solved
\begin{align}
\begin{split} \label{prob.OFP}
\lambda^{w} = \mbox{arg}\min \;& \sum_{k=1}^{m}\mc{P}_{k}(\lambda_{k}) \\
& 0 = B \lambda + Pw,
\end{split}
\end{align}
where $ w = w(t)$ is the supply/demand at time $t$. We will use in the following the short hand notation $\mc{P}(\lambda) = \sum_{k=1}^{m} \mc{P}_{k}(\lambda_{k})$. 

The convex cost functions are such that they incorporate the capacity constraints of the transportation lines, in the sense that $P_{k}(\lambda_{k}) = \infty$ for $|\lambda_{k}| > \bar{\lambda}_{k}$.
More precisely, we assume that the cost functions are of \emph{Legendre type} \cite[p. 258]{Rockafellar1997}.

\begin{assumption}\label{ass.Legendre}
The functions $\mc{P}_{k}$, $k \in \{1,\ldots,m\}$ are strictly convex and differentiable on the open set $\{ \lambda_{k} \; : \; - \bar{\lambda}_{k} < \lambda_{k}(t) < \bar{\lambda}_{k} \}$, and $\lim_{\ell \rightarrow \infty} |\nabla \mc{P}(\lambda^{[\ell]}_{k})| = + \infty$ for any sequence $\lambda_{k}^{[1]},\lambda_{k}^{[2]},\ldots$ such that $|\lambda_{k}^{[\ell]} | < \bar{\lambda}_{k}$ and $\lim_{\ell \rightarrow \infty} |\lambda_{k}^{[\ell]}| = \bar{\lambda}_{k}$.
\end{assumption}
With cost functions of this form, any finite solution to \eqref{prob.OFP}  represents a feasible flow. 
Note that if the assumptions of Proposition \ref{prop.Feasibility} hold, then the optimization problem \eqref{prob.OFP} has always a finite optimal solution.


We can now characterize the optimality conditions of \eqref{prob.OFP}. The Lagrangian function associated to \eqref{prob.OFP} with Lagrange multiplier $\zeta$ is 
\[ \mc{L}(\lambda,\zeta) = \mc{P}(\lambda) + \zeta^{T}(-B \lambda - P w).\]
Now,  $(\lambda^{w},\zeta^{w})$ is an optimal primal/dual solution pair to \eqref{prob.OFP} if the following nonlinear equations hold
\begin{align}
\begin{split} \label{eqn.OptimalityConditions}
\nabla \mc{P}(\lambda^{w}) - B^{T}\zeta^{w} &= 0 \\
B \lambda^{w} + Pw &= 0.
\end{split}
\end{align}
Note that the first condition simply implies $\nabla \mc{P}(\lambda^{w}) \in \mc{R}(B^{T})$ since $\zeta^{w}$ has no further constraints. 
We define the set of optimal routing/supply pairs as
\begin{align*}
\begin{split}
\Gamma = \{ (\lambda,w) \in \mathbb{R}^{m} \times \mc{W} \; : &\; \nabla \mc{P}(\lambda) \in \mc{R}(B^{T}), \\
&\; B \lambda + Pw = 0  \}.
\end{split}
\end{align*}
In particular, $(\lambda^{w},w) \in \Gamma$ if and only if $\lambda^{w}$ is an optimal solution to the static optimal distribution problem \eqref{prob.OFP} with the supply vector $w$.
%

\begin{proposition}[Feed-forward Controller] \label{eqn.FFfeasibility}
Suppose Assumptions  \ref{ass.Boundedness}, \ref{ass.BalancedFlow}, \ref{ass.Cut} and \ref{ass.Legendre} hold. Then, there exists functions $\Phi : \mathbb{R}^{n} \times \mathbb{R}^{r} \rightarrow \mathbb{R}^{n}$ and $\Psi : \mathbb{R}^{n} \rightarrow \mathbb{R}^{m}$, such that 
for any $w(t)$ being a solution to \eqref{sys.Inventory_w} there exists $\zeta^{w}(0)$ such that the solution $\zeta^{w}(t)$ to the dynamical system
\begin{align}
\begin{split} \label{eqn.FF}
\dot{\zeta} &= \Phi(\zeta,w),  \\
\lambda &= \Psi(\zeta)
\end{split}
\end{align}
originating at $\zeta^{w}(0)$ is such that $\lambda^{w}(t) = \Psi(\zeta^{w}(t))$ satisfies $(\lambda^{w}(t),w(t)) \in \Gamma$, where $w(t)$ is the solution to \eqref{sys.Inventory_w} originating at $w(0) = w_{0}$.
\end{proposition}

The proof of this result provides a method to construct the dynamical system \eqref{eqn.FF}. However, before we present the proof,  some additional discussion is required. \\

The main difficulty associated to the proposed characterization of the set $\Gamma$ relates to the constraint $\nabla \mc{P}(\lambda) \in \mc{R}(B^{T})$. To avoid this constraint, we express the optimality conditions in terms of the \emph{dual solutions} $\zeta$. 
Under the given assumption that the cost functions $\mc{P}_{k}$ are of Legendre type,  the gradients $\nabla \mc{P}_{k}$ are invertible on the open sets $\{\lambda_{k} \; : \; -\bar{\lambda}_{k} < \lambda_{k} < \bar{\lambda}_{k}\}$ \cite{Rockafellar1998}. Since the optimal solution must always be strictly feasible, the two optimality conditions \eqref{eqn.OptimalityConditions} can  be combined to
\begin{align} \label{eqn.DualCond} B \nabla \mc{P}^{-1}(B^{T}\zeta^{w}) + Pw = 0.\end{align}
Bearing this in mind, we define the set of optimal dual solutions as
\[
\Gamma_{D} = \{ (\zeta,w) \in \mathbb{R}^{n} \times \mc{W} : B \nabla \mc{P}^{-1}(B^{T}\zeta) + Pw = 0 \}.
\] 
We want to emphasize two properties of $\Gamma_{D}$. First, if $(\zeta^{w},w) \in \Gamma_{D}$ then $(\zeta^{w}+ c\1,w) \in \Gamma_{D}$ for any $c \in \mathbb{R}$. Second, $(\zeta^{w},w) \in \Gamma_{D}$ if any only if the routing strategy $\lambda^{w} = \nabla \mc{P}^{-1}(B^{T}\zeta^{w})$ satisfies $(\lambda^{w},w) \in \Gamma$. 
\begin{remark}
In the considered framework, the inverse of the gradient of $\mc{P}_{k}$ is the gradient of the convex conjugate $\mc{P}_{k}^{\star}$, i.e., $\nabla \mc{P}^{-1}(\cdot) = \nabla \mc{P}^{\star}(\cdot)$ \cite{Rockafellar1997}. 
For completeness, we want to point out that the dual problem to \eqref{prob.OFP} is
\begin{align}
\begin{split} \label{prob.OPP}
\min_{z, \zeta}  \; &\sum_{k=1}^{m} \mc{P}^{\star}_{k}(z_{k}) + w(t)^{T}P^{T}\zeta \\
& z = B^{T}\zeta,
\end{split}
\end{align}
where $\mc{P}^{\star}_{k}$ are the convex conjugates of $\mc{P}_{k}$. Any optimal solution $\zeta^{w}$ to \eqref{prob.OPP} is such that $(\zeta^{w},w(t)) \in \Gamma_{D}$.
\end{remark}
The main idea pursued here is to design the feed-forward controller for the \emph{dual variables}.
We need to design a dynamics for $\zeta(t)$ such that the constraint \eqref{eqn.DualCond} is satisfied for all times.
Suppose now \eqref{eqn.DualCond} is feasible at the initial time, that is
\begin{align} \label{eqn.vwDynamics0}
B \nabla \mc{P}^{-1}(B^{T}\zeta^{w}(0)) = -Pw(0). \;    
\end{align}
Then, $\dot{\zeta}^{w}(t)$ must be such that
\begin{align} \label{eqn.vwDynamics}
\left(B \D \mc{P}^{-1}(B^{T}\zeta^{w}(t)) B^{T}\right) \dot{\zeta}^{w} = - P\dot{w}(t),
\end{align}
where $\D \mc{P}^{-1}(\cdot) = \mathrm{diag} [\nabla^{2}\mc{P}_{1}^{-1}(\cdot), \ldots, \nabla^{2} \mc{P}_{m}^{-1}(\cdot)] $. Note that $\D \mc{P}^{-1}(\cdot)$ is the Hessian matrix of $\mc{P}^{\star}$. 

The equation \eqref{eqn.vwDynamics} defines an implicit differential equation for $\zeta^{w}$. For the design of the feed-forward controller \eqref{eqn.FF}, we need to transform it into an explicit differential equation.

Under the stated assumptions $\D \mc{P}^{-1}(\cdot)$ is a diagonal matrix, with all entries on the diagonal being positive. Thus, the matrix 
\begin{align}\label{eqn.WeightedLaplacian1}
 L(\zeta^{w}) := B \D \mc{P}^{-1}(B^{T}\zeta^{w}) B^{T}
 \end{align}
 is a weighted, state dependent \emph{Laplacian matrix} of the network, with the weights on the edges being the second derivatives of the cost function's convex conjugates. 
The Laplacian matrix has always one eigenvalue at zero and is therefore not invertible. To find an explicit differential equation for $\zeta^{w}$, we make the following observation. 
 Since $\1_{n}^{T}Pw(t) = 0$ there exists always $\omega(t) \in \mathbb{R}^{m}$ such that
\[ Pw(t) = B \omega(t). \]
One can always define $\omega(t) = B^{T}(BB^{T})^{\dagger}Pw(t)$, where $(BB^{T})^{\dagger}$ is the \emph{generalized inverse} of the matrix $BB^{T}$.\footnote{By the properties of the generalized inverse, it is readily verified that $B\omega(t) = BB^{T}(BB^{T})^{\dagger}Pw(t) = (I- \frac{1}{n}\1_{n}\1_{n}^{T})Pw(t) =Pw(t)$. }
Now, the conditions \eqref{eqn.vwDynamics} can be written as
\begin{align} \label{eqn.GenDirichlet}
L(\zeta^{w}(t)) \dot{\zeta}^{w}(t) = - B\dot{\omega}(t).
\end{align}
Following \cite[Chap. 16]{Biggs1997}, we call the problem to find $\dot{\zeta}^{w}$ a \emph{generalized Dirichlet problem}. 

It is described in \cite[p. 659]{Biggs1997}, how an analytic solution for $\dot{\zeta}^{w}$ in \eqref{eqn.GenDirichlet}  can be constructed.
In particular, one can find analytically a matrix $X(\zeta^{w}) \in \mathbb{R}^{m \times n}$ such that
\begin{align} \label{eqn.TreeSol}
 \dot{\zeta}^{w}(t)  = -X(\zeta^{w}(t))^{T} \D \mc{P}(B^{T}\zeta^{w}(t))\dot{\omega}
\end{align}
solves \eqref{eqn.vwDynamics}. The algorithmic procedure for the construction of $X(\zeta^{w})$ is provided in \cite[p. 659]{Biggs1997} and is reviewed in the appendix. The construction of $X(\zeta^{w})$ has an explicit graph theoretic interpretation and the matrix has a characteristic structure.
The entries of $X(\zeta^{w})$ are fractions of polynomials, with each polynomial being the sum of products of different functions $\nabla^{2}\mc{P}_{k}^{-1}(\cdot)$.
Since, by assumption, all $\nabla^{2}\mc{P}_{k}^{-1}(\cdot)$ are strictly positive and continuously differentiable,  $X(\zeta^{w})$ is continuously differentiable.
Note that if the assumptions of Proposition \ref{prop.Feasibility} hold, the optimal flow problem \eqref{prob.OFP} has at each time instant a feasible solution, and thus admits a finite dual solution.
We are now ready to prove Proposition \ref{eqn.FFfeasibility}. \\

\begin{proof} \textit{(of Proposition \ref{eqn.FFfeasibility})}
We prove the proposition by constructing the corresponding feed-forward controller.  Consider the dynamical system
 \begin{align}
\begin{split} \label{sys.FFcontrollerRouting}
 \dot{\zeta} &=  -X(\zeta)^{T} \D \mc{P}(B^{T}\zeta)B^{T}(BB^{T})^{\dagger} Ps(w) \\
 \lambda &= \nabla \mc{P}^{-1}(B^{T}\zeta),
\end{split}
\end{align}
where $X(\zeta)$ is as constructed in \eqref{eqn.TreeSol} and $(BB^{T})^{\dagger}$ is the generalized inverse of $BB^{T}$.
If the initial condition $\zeta(0)$ is chosen as an optimal dual solution at time $t=0$, that is $(\zeta(0),w(0)) \in \Gamma_{D}$, then the solution $\zeta(t)$ to \eqref{sys.FFcontrollerRouting} originating at $\zeta(0)$ will be a solution to \eqref{eqn.vwDynamics}, and will therefore be an optimal dual solution for all times, i.e., $(\zeta(t),w(t)) \in \Gamma_{D}$ for all $t \geq 0$. 
Furthermore, if $\zeta(t)$ is an optimal dual solution at time $t$, then  $\lambda(t)= \nabla \mc{P}^{-1}(B^{T}\zeta(t))$ is an optimal primal solution, i.e., $(\lambda(t),w(t) )\in \Gamma$.
Thus, Proposition \ref{eqn.FFfeasibility} follows with $\Phi(\zeta,w) =  -X(\zeta)^{T} \D \mc{P}(B^{T}\zeta)B^{T}(BB^{T})^{\dagger} Ps(w)$, $\Psi(\zeta) = \nabla \mc{P}^{-1}(B^{T}\zeta)$ and $\zeta(0)$ such that $(\zeta(0),w(0)) \in \Gamma_{D}$.
\end{proof}

\section{Feedback Controller Design} \label{sec.Feedback}

We can use the proposed feed-forward controller now to design a dynamic feedback controller for the inventory system \eqref{sys.Inventory1},  that ensures convergence to the optimal routing without being initialized with an optimal solution. 
In addition to the optimal routing, the feedback controller should achieve a \emph{balancing} of the inventory levels of \eqref{sys.Inventory1}.
Network balancing is widely accepted as control objective in storage networks, as it promises to make best use of the available storage, \cite{Danielson2013}, \cite{DePersis2013}. 
\begin{definition}
The \emph{time-varying optimal distribution problem} is solvable for the system \eqref{sys.Inventory1}, if there exists a controller of the form
\begin{align} \label{sys.ControllerBasic}
\begin{split}
\dot{\zeta} &= \Phi(\zeta,w,x) \\
\lambda &= \Psi(\zeta,x)
\end{split}
\end{align} 
such that any solution $(x(t),w(t),\zeta(t))$ of the closed-loop system \eqref{sys.Inventory1}, \eqref{sys.Inventory_w}, \eqref{sys.ControllerBasic} is bounded and satisfies
(i) $\lim_{t \rightarrow \infty} B^{T}x(t) = 0$; and
(ii)  $\lim_{t\rightarrow \infty} \mathrm{dist}_{\Gamma}\bigl( \lambda(t),w(t) \bigr) = 0.$ 
\end{definition}

Note that the flow constraints are treated here as soft constraints, that must be satisfied only asymptotically. 
Furthermore, the proposed problem formulation allows to use the supply/demand vector $w$ directly in the controller.

We will construct a controller \eqref{sys.ControllerBasic} solving the time-varying optimal distribution problem on the basis of the feed-forward controller designed in the previous section.
For the sake of a compact presentation, we define the matrix
\begin{align} \label{eqn.Xi}
 \Xi(\zeta):= X(\zeta)^{T} \D \mc{P}(B^{T}\zeta)B^{T}(BB^{T})^{\dagger}. 
 \end{align}
Furthermore, the projection matrix for the space orthogonal to the all ones vector is
\begin{align} \label{eqn.Projection}
Q = (I_{n} - \frac{1}{n}\1_{n}\1_{n}^{T}).
\end{align}
For designing the feedback controller, we need to relax the original assumption on the feasibility of the flows (i.e., Assumption \ref{ass.Cut}).
\begin{assumption} \label{ass.CutInfty}
For any set of nodes $S \subset V$ it holds that $\bar{\lambda}_{Q(S)} = \infty$.
\end{assumption}
Note that this assumption does not imply that all lines of the network have an infinite capacity, but rather that there exists a spanning tree in the network with infinite capacity.
 The main result of this paper is a controller that routes the flow asymptotically optimal through the network.
\begin{theorem}[Feedback Controller]
Consider the inventory dynamics \eqref{sys.Inventory1} with the supply generated by the dynamical system \eqref{sys.Inventory_w}. 
Suppose the Assumptions \ref{ass.Boundedness}, \ref{ass.BalancedFlow}, \ref{ass.Legendre}, and \ref{ass.CutInfty} hold.
Then,  the dynamic controller
\begin{align}
\begin{split} \label{sys.PriceController}
\dot{\zeta} &= - Q\Xi(\zeta)\Bigl( Ps(w)  + x+  k \bigl(B \nabla \mc{P}^{-1}(B^{T}\zeta) + Pw \bigr) \Bigr) \\
\lambda & = \nabla \mc{P}^{-1}(B^{T}\zeta) - B^{T}x.
\end{split}
\end{align}
with $k > 0$ a scalar controller gain, solves the time-varying optimal distribution problem for any initial inventory levels $x(0)$ and any controller initialization $\zeta(0)$, that is, all trajectories remain bounded, and (i) $\lim_{t \rightarrow \infty} B^{T} x  = 0$ and (ii) $\lim_{t\rightarrow \infty} \mathrm{dist}_{\Gamma}\bigl( \lambda(t),w(t) \bigr) = 0.$
\end{theorem}

\begin{proof}
We recall that, due to the assumption of balanced supply and demand, it must hold that 
\begin{align} \label{eqn.xw}
x^{w}(t) = x_{0}^{w}\1_{n},\; \mbox{with}\; x_{0}^{w}= \frac{1}{n} \1_{n}^{T}x(0).
\end{align}
Note that the average inventory level is constant over time since $\1_{n}^{T}\dot{x} = 0$ and therefore
\begin{align*}
(x(t) - x^{w}(t)) \in \mc{R}(B).
\end{align*}
Furthermore, we have that the average of the controller variables $\frac{1}{n}\1_{n}^{T}\zeta(t)$ is constant since $\1_{n}^{T}\dot{\zeta} = 0$.
The optimal controller (i.e., dual) variables satisfy
\begin{align} \label{eqn.zetaw}
 B \nabla \mc{P}^{-1}(B^{T}\zeta^{w}(t) ) + Pw(t) = 0. 
 \end{align}

To prove the desired convergence result, consider a composed storage function of the form
\begin{align}
U(x-x^{w},\zeta,w) = V(x-x^{w}) + W(\zeta,w).
\end{align}
As the first component of the storage function, we consider
\begin{align}
V(x-x^{w}) = \frac{1}{2}\|x - x^{w}\|^{2}.
\end{align}
This function satisfies
\begin{align*}
\dot{V} =& (x - x^{w})^{T}(B \lambda + Pw) \\
=& -(x-x^{w})^{T}BB^{T}(x-x^{w})   \\
&+(x-x^{w})^{T}\left(B\nabla \mc{P}^{-1}(B^{T}\zeta) + Pw) \right),
\end{align*}
where we exploited the two relations \eqref{eqn.xw} and \eqref{eqn.zetaw}.

Now, as the second component of $U$ we consider the positive semidefinite function
\begin{align} \label{eqn.W}
W(\zeta,w) =  \frac{1}{2}\|B\nabla \mc{P}^{-1}(B^{T}\zeta) + Pw\|^{2}.
\end{align}
The function satisfies $W(\zeta,w) = 0$ if and only if $(\zeta,w) \in \Gamma_{D}$, that is if $\zeta = \zeta^{w}$.

In order to keep the notation compact, we introduce the vector valued function
\begin{align} \label{eqn.Z}
Z(\zeta,w) := B\nabla \mc{P}^{-1}(B^{T}\zeta) + Pw.
\end{align}

The directional derivative of $W$ satisfies
\begin{align*}
&\dot{W}=   Z(\zeta,w)^{T}\bigl(B \D \mc{P}^{-1}(B^{T}\zeta) B^{T}\dot{\zeta} + Ps(w)    \bigr) \\
&= Z(\zeta,w)^{T}\Bigl( - L(\zeta) \Xi(\zeta)\bigl( Ps(w) + x  + k Z(\zeta,w) \bigr)  + Ps(w) \Bigr),
\end{align*}
where we defined analogous to \eqref{eqn.WeightedLaplacian1} 
\[ L(\zeta) := \bigl(B \D \mc{P}^{-1}(B^{T}\zeta) B^{T} \bigr) \]
and used the fact that $L(\zeta)Q = L(\zeta).$

Recalling the definition of $\Xi(\zeta)$ in \eqref{eqn.Xi} and the definition of $X(\zeta)$ provided in the appendix, and using that $Ps(w) \in \mc{R}(B)$ it follows that  
\begin{align*}
 L(\zeta) \Xi(\zeta)Ps(w) &= L(\zeta)X(\zeta)^{T} \D \mc{P}(B^{T}\zeta)B^{T}(BB^{T})^{\dagger}Ps(w) \\
&= BB^{T}(BB^{T})^{\dagger}Ps(w) \\
&=(I - \frac{1}{n}\1_{n}\1_{n}^{T})Ps(w) \\
&= Ps(w).
\end{align*}

Futhermore, we have
\begin{align*}
 L(\zeta) \Xi(\zeta)x &= L(\zeta)X(\zeta) \D \mc{P}(B^{T}\zeta)B^{T}(BB^{T})^{\dagger}x \\
 &=L(\zeta)X(\zeta) \D \mc{P}(B^{T}\zeta)B^{T}(BB^{T})^{\dagger}(x -x^{w}) \\
 &=B B^{T}(BB^{T})^{\dagger}(x -x^{w}) \\
 &= (I - \frac{1}{n}\1_{n}\1_{n}^{T})(x -x^{w}) \\
 &= (x - x^{w}).
\end{align*}
To derive the previous conclusion, we exploited that $(BB^{T})^{\dagger}\1_{n} = 0$ (see \cite[Lemma 2]{Gutman2004}) and that $(x - x^{w}) \in \mc{R}(B)$.
Along the same line of argumentation, and by exploiting that $\1^{T}_{n}Z(\zeta,w) = 0$, we obtain
\[ L(\zeta) \Xi(\zeta)Z(\zeta,w) = Z(\zeta,w). \]
After these considerations it follows that
\begin{align*}
\dot{W}=& -Z(\zeta,w)^{T}(x - x^{w}) -k Z(\zeta,w)^{T} Z(\zeta,w).
\end{align*}
To complete the proof of convergence we consider now the composed (positive semidefinite) storage function
\begin{align*} 
U(x - x^{w},\zeta,w)& =  V(x-x^{w}) + W(\zeta,w) \\
&= \frac{1}{2}\|x - x^{w}\|^{2} + \frac{1}{2}\|Z(\zeta,w)\|^{2}.
\end{align*}
The derivative of this function satisfies
\begin{align}\label{eqn.Udot}
\begin{split}
\dot{U} =& -(x-x^{w})^{T}BB^{T}(x-x^{w})    + (x - x^{w})^{T}Z(\zeta,w) \\
&-Z(\zeta,w)^{T}(x - x^{w}) -k Z(\zeta,w)^{T}  Z(\zeta,w) \\
= & -(x-x^{w})^{T}BB^{T}(x-x^{w})   -k Z(\zeta,w)^{T}Z(\zeta,w).
\end{split}
\end{align}
Thus, the function $U$ is non-increasing under the closed-loop dynamics. 
Note that $\1^{T}(x-x^{w})= 0$ and therefore it holds that
\begin{align}
(x-x^{w})^{T}BB^{T}(x-x^{w}) \geq s_{2}\|x-x^{w}\|^{2}
\end{align}
where $s_{2} > 0$ is the smallest non-zero eigenvalue of the Laplacian matrix $BB^{T}$.\footnote{The scalar $s_{2}$ is known as the \emph{algebraic connectivity} of the network.} 

The derivative of $U(x-x^{w},\zeta,w)$ satisfies therefore
\begin{align}
\dot{U} \leq - \sigma U,
\end{align}
where $\sigma = \min \{s_{2},k\}$.
It follows now from the Comparison Lemma \cite[Lemma 3.4]{Khalil2002} that
\begin{align*}
U\bigr(x(t)-x^{w}(t),\zeta(t),w(t) \bigl)  \\
\leq \mathrm{e}^{-\sigma t}U\bigr(x(0)-x^{w}(0),\zeta(0),w(0) \bigl),
\end{align*}
and consequently
$\lim_{t \rightarrow \infty} U\bigr(x(t)-x^{w}(t),\zeta(t),w(t) \bigl)  = 0. $
This proves that $\lim_{t\rightarrow \infty} \| x(t)- x^{w}(t)\| = 0$, which implies $\lim_{t \rightarrow \infty} B^{T}x(t) = 0$ and $\lim_{t \rightarrow \infty} Z(\zeta(t),w(t)) = 0$. These conditions imply $\lim_{t \rightarrow \infty} \mathrm{dist}_{\Gamma}\bigl( \lambda(t),w(t) \bigr) = 0$. 

It remains to show boundedness of all trajectories. The uniform boundedness of $U\bigr( x(t)-x^{w},\zeta(t),w(t) \bigl)$ over time implies directly that $V = \|x(t)-x^{w}\|^{2}$ must remain bounded, and consequently $x(t)$ must remain bounded.
Now, the boundedness of $U$ implies boundedness of $W = \|B \nabla \mc{P}^{-1}(B^{T}\zeta) -Pw(t)\|^{2}$. By Assumption \ref{ass.Boundedness}, $w(t)$ is bounded and consequently also $\|B \nabla \mc{P}^{-1}(B^{T}\zeta(t))\|$ must remain bounded.
Thus, it holds that at each time instant there exists a vector $c(t)$ with finite norm such that
\begin{align*}
B \nabla \mc{P}^{-1}(B^{T}\zeta(t)) = c(t).
\end{align*}
It follows with the same argumentation as used previously to derive \eqref{eqn.DualCond}, that at each time instant the vector $\zeta(t)$ is an optimal dual solution to a network flow problem of the form
\begin{align*}
\min \; &\sum_{k=1}^{m} P_{k}(\tilde{\lambda}_{k}(t))\quad \mathrm{s.t.} \; B\tilde{\lambda}(t) = c(t). 
\end{align*}
Due to Assumption \ref{ass.CutInfty} the problem has a feasible solution, and due to strict convexity of the cost functions (Assumption \ref{ass.Legendre}) the optimal solution is unique.
From the optimality conditions follows that there exists at each time instant a unique $z(t)$ such that $\nabla \mc{P}(\tilde{\lambda}(t)) = z(t)$ and $z(t) = B^{T}\zeta(t)$. 
Since additionally at each time instant $\1^{T}\zeta(t) = \1^{T}\zeta(0)$, the vector $\zeta(t)$ is uniquely defined and must be bounded. This proves that all trajectories remain bounded.
\end{proof}

\begin{remark}
The controller  \eqref{sys.PriceController} has a characteristic structure. The matrix  $\Xi(\zeta)$, as defined in \eqref{eqn.Xi}, can be understood as the \emph{inverse Hessian}  of the dual problem \eqref{prob.OPP}, and the vector $Z(\zeta,w)$, as defined in \eqref{eqn.Z}, is the gradient of \eqref{prob.OPP}. 
This interpretation suggests that  \eqref{sys.PriceController} represents a continuous-time \emph{Newton algorithm} for the dual problem \eqref{prob.OPP}, augmented with a feed-forward component and a regulation term for the inventory levels.\footnote{The projection matrix $Q$ ensures only that the controller generates dual variables with a constant mean value. Note that the optimal dual variables are only defined up to a variation in $\mathrm{span} (\1)$. Thus, for any desired mean value, there exists an optimal dual solution having exactly this mean value.}
Furthermore, as the dual variables of the optimal flow problem \eqref{prob.OFP} are often given an economic interpretation as \emph{prices}, the controller \eqref{sys.PriceController} can be understood as an optimal pricing controller.
\end{remark}

\section{Design Example} \label{sec.Example}
We illustrate the proposed controller design on a small example network. We consider an inventory system with three inventories and three transportation lines as illustrated in Figure \ref{fig.Inventory}.

We assign an arbitrary orientation to the three connecting lines to obtain the incidence matrix
\begin{align*}
B = \begin{bmatrix} -1  &0 & 1 \\   1 & -1 & 0 \\  0  & 1  & -1 \end{bmatrix}.
\end{align*}

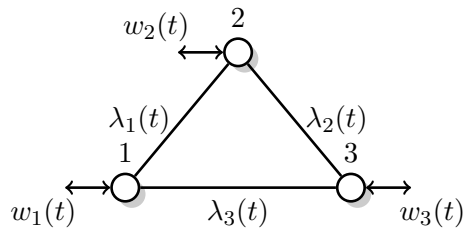
\begin{figure}
\begin{center}
\begin{tikzpicture}
[robot/.style={circle,draw=black,fill=white,minimum size=7pt, drop shadow={color=gray!80}, line width=1pt}]
\node[robot] (1) at (0,0) [label=above:$1$]{};
\node[robot] (2) at (1.5,1.8)[label=above:$2$] {};
\node[robot] (3) at (3,0 ) [label=above:$3$]{};
\draw[ line width=1pt](1)-- node[label=left:$\lambda_{1}(t)$]{}(2);
\draw[ line width=1pt](2)--node[label=right:$\lambda_{2}(t)$]{}(3);
\draw[  line width=1pt](3)--node[label=below:$\lambda_{3}(t)$]{}(1);
\draw[<->,  line width=1pt](-0.8,0)--node[label=below left:$w_{1}(t)$]{}(1);
\draw[<->,  line width=1pt](0.7,1.8)--node[label=above left:$w_{2}(t)$]{}(2);
\draw[<->,  line width=1pt](3.8,0)--node[label=below right:$w_{3}(t)$]{}(3);
\end{tikzpicture}
\caption{Illustration of the inventory system.}
\label{fig.Inventory}
\end{center}
\end{figure}

For the first two lines, we assume that the flows are unconstrained and cost functions are quadratic
\begin{align*}
P_{k}(\lambda_{k}) = \frac{q_{k}}{2}(\lambda_{k} - r_{k})^{2}, \; k \in \{1,2\}.
\end{align*}
For the simulations we choose $q_{1} = 0.2$, $q_{2} = 10$, $r_{1}=2$, and $r_{2} = 0$. The inverse of the gradients for these functions become simply
\begin{align*}
\nabla \mc{P}_{k}^{-1}(z_{k}) = q_{k}^{-1} z_{k} + r_{k}, \; k \in \{1,2\}.
\end{align*}

For the third transportation line, we assume that the the flow is constrained to $|\lambda_{e}(t)| \leq \bar{\lambda}_{3}$. We assign therefore to line three a strictly convex cost function that goes to infinity as $\lambda_{3}$ approaches $\pm \bar{\lambda}_{3}$. We choose here the cost function
\begin{align}
\mc{P}_{3}(\lambda_{3}) =  -c a \log \Bigl( \cos(\frac{\lambda_{3}}{a}) \Bigr) 
\end{align}
where $c > 0$ and $a = \frac{2}{\pi}\bar{\lambda}_{3}$.
For the simulation, we choose $c = 0.1$ and $\bar{\lambda}_{3} = 4$. 
Choosing this cost function, one obtains
\begin{align}
\nabla \mc{P}_{3}^{-1}(z) = a \tan^{-1}(\frac{z}{c}).
\end{align}
The function $\mc{P}_{3}(\lambda_{3})$ and the function $\nabla \mc{P}_{3}^{-1}(z_{3})$ are shown in Figure \ref{fig.dP3}.

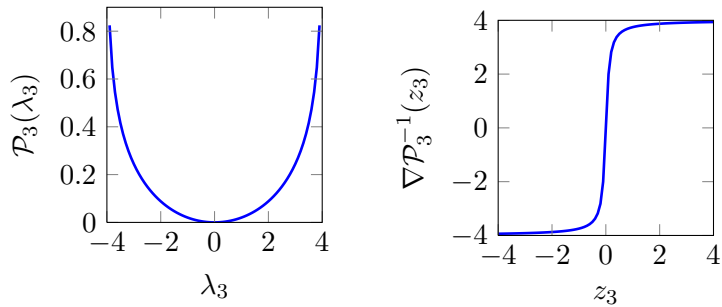
\begin{figure}[b]
\begin{center}
   \begin{subfigure}{0.23\textwidth}
   \input{P3.tikz}
    \label{fig.P3}
   \end{subfigure}
   \hspace{2em}
   \begin{subfigure}{0.23\textwidth}
   \input{dP3inv.tikz}
   \label{fig.dP3inv}
   \end{subfigure}
\end{center}   
   \caption{The cost function $\mc{P}_{3}(\lambda_{3})$ and the corresponding inverse of the gradient $\nabla \mc{P}_{3}^{-1}(z_{3})$.}
   \label{fig.dP3}
  \end{figure}

To construct the controller, the matrix $X(\zeta)$ as defined in \eqref{eqn.TreeSol} is required. For convenience, we will use in the following the notation we write for convenience  $h_{k} = \nabla \mc{P}_{k}^{-1}(\zeta_{k}), k \in \{1,2,3\}$.  Following the procedure described in the appendix, one can determine $X$ as
\begin{align*}
X = \frac{1}{\bar{h}} \begin{bmatrix}
0 & h_{1}h_{3} + h_{1}h_{2} & h_{1}h_{2} \\
h_{2}h_{3} & 0 & h_{2}h_{3} + h_{1}h_{2} \\
h_{2}h_{3} + h_{1}h_{3} & h_{1} h_{3} & 0
\end{bmatrix}
\end{align*}
where $\bar{h} = h_{1}h_{2}+ h_{2}h_{3} + h_{1}h_{3}$ .

For the simulations, we assume that the flow is generated by linear dynamical systems and that it is consequently an harmonic signal. We assume furthermore, that the supply generated at one node is an instantaneous demand at another node, i.e.,
\begin{align*}
w_{i}(t) = \sum_{ k = 1}^{3} b_{ik} \bigl(\kappa_{k} \sin(\phi_{k}t + \rho_{k}) \bigr).
\end{align*}
For the simulation we have chosen $\kappa_{1}=2$, $\kappa_{2} = \kappa_{3}=4$, $\phi_{1}=2$, $\phi_{2} = 4$, $\phi_{3} = 8$, and $\rho_{1} = 0$, $\rho_{2} = 2$, $\rho_{3} = 3.14$.
We consider a scenario, where the supply/demand evolves first continuously, but chances instantaneously at a certain time instant. That is, at time $t=3$, we change the phase shift of the reference signal to 
$\rho_{1} = 4$, $\rho_{2} = 6$, $\rho_{3} = 2$.
The simulation results are shown in Figure \ref{fig.Sim}. The controller balances all inventory levels and routes the flows such that they all approach the optimal flows. Note that the flow $\lambda_{3}$ remains always within its capacity bounds. It seems remarkable, that the transportation line $2$, which has a fairly high cost, is mainly used when the flow in line 3 is close to the capacity bounds.
At time $t=3$ the reference signal is changed instantaneously as described above, and the controller needs to readjust. However, the feedback controller can tolerate this change and regulates the flows to the new optimal strategy. The simulation study illustrates the robustness of the proposed feedback controller.

\begin{figure}[t]
\begin{center}
\begin{subfigure}{0.23\textwidth}
\input{x.tikz}
\end{subfigure}
\hspace{2em}
\begin{subfigure}{0.23\textwidth}
\input{Flow1.tikz}
\end{subfigure}

\begin{subfigure}{0.23\textwidth}
\input{Flow2.tikz}
\end{subfigure}
\hspace{2em}
\begin{subfigure}{0.23\textwidth}
\input{Flow3.tikz}
\end{subfigure}
\end{center}
\caption{Inventory levels (top left) and flows on the lines $\lambda_{k}$ (solid) together with the optimal flows $\lambda_{k}^{w}$ (dashed).} \label{fig.Sim}
\end{figure}
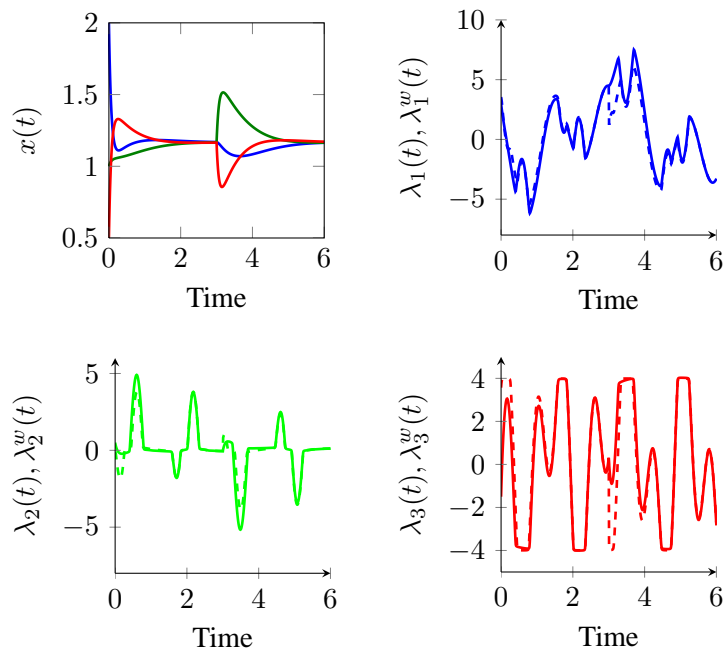

\section{Conclusions}
We proposed a feedback controller for an optimal routing control in distribution networks with storage and time-varying supply. 
The controller is composed of a feed-forward component, generating the time-varying optimal dual variables, and a feedback component, regulating the system to the flows to the optimal trajectory. 
The construction of the feed-forward component of the controller utilizes the analytic solution to a generalized Dirichlet problem.
With certain choices of the cost functions, the proposed controller will also respect capacity constraints on the transportation lines. 
One can expect that the presented results can be generalized in various directions. A relevant extension will be a distributed implementation of the controller. 
Furthermore, future research will focus on a controller design that does note require a measurement of the external supply and demand. 

\appendix
\subsection{Analytic Solution to the Generalized Dirichlet Problem}

We review here the analytic solution for generalized Dirichlet problems provided in \cite{Biggs1997}. Consider the problem of finding a vector $z$ that solves an equation of the form
\begin{align} \label{eqn.AppGDP}
(B H B^{T}) z = B v,
\end{align}
where $B \in \mathbb{R}^{n \times m}$ is the incidence matrix of a graph $\mc{G} = (V,E)$ (with arbitrary orientation), 
$H$ is a diagonal matrix with positive elements on the diagonal, and $v$ is a given vector. 
It is known (\cite{Biggs1997}) that a solution to \eqref{eqn.AppGDP} is 
\begin{align}\label{z}
z = X^{T}H^{-1}v.
\end{align}
The analytic construction of the matrix $X$  is reported and justified in \cite[p. 659]{Biggs1997}, and recalled below. 
In what follows we say that the node $i \in V$ is the positive end of edge $k \in E$ if $[B]_{k,i} = 1$ and the negative end if $[B]_{k,i} = -1$. \\

We call a \emph{spanning tree} of the graph $\mc{G}$ an acyclic, connected subgraph $\mc{T} \subset \mc{G}$, defined on the same node set, 
\[ \mc{T} = (V,T)\]
having an edge set $T$ being exactly of size $|T| = n-1$.

Consider now any edge $l \in T$. After removing $l$ from $T$, the spanning tree will partition into two connected components. The connected component that contains the positive end of $l$ will be denoted with $\mc{T}^{+}_{l}$. 
For the set of nodes forming the connected component $\mc{T}^{+}_{l}$, we will write $U(\mc{T},l)$.  

Each edge $k \in E$ has a weight $h_{k}$, which is the $k$th diagonal element of the matrix $H$. 
The \emph{weight of a spanning tree} is now defined as
\[ h(\mc{T}) = \prod_{k \in T} h_{k}. \]
Furthermore, we define the sum of the weights of all spanning trees of $\mc{G}$ by
\[ \bar{h} = \sum_{\mc{T}} h(\mc{T}). \]

We can now construct the matrix $X \in \mathbb{R}^{m \times n}$ in (\ref{z}) according to the following procedure. For each edge $k \in E$ and each node $i \in V$ define the corresponding entry of $X$ as
\[ (X)_{ki} = \frac{1}{\bar{h}} \sum_{ \{ \mc{T} \; : \; i \in U(\mc{T},k) \} }  h(\mc{T}),\]
that is the sum of the weights of all spanning trees for which the node $i$ is contained in the positive connected component. 

The proposed method can naturally be used to solve the equation \eqref{eqn.vwDynamics}. There, the weighting matrix is $H = \D \mc{P}^{-1}(B^{T}\zeta^{w}(t))$. The time-varying nature of the weights has no influence on the solution. In fact, since the solution is analytically constructed, we will obtain expressions for $X$ and $H$ that depend explicitly on the state. \\

{
\bibliographystyle{IEEEtran}
\bibliography{Literature}
}

\end{document}

%% file: P3.tikz
%
%
%
%
\begin{tikzpicture}

\begin{axis}[%
width=0.7\textwidth,
height=0.7\textwidth,
scale only axis,
xmin=-4, xmax=4,
xlabel={$\lambda_{3}$},
ymin=-0, ymax=0.9,
ylabel={$\mc{P}_{3}(\lambda_{3})$},
axis on top]
\addplot [
color=blue,
solid,
line width=1.0pt
]
coordinates{
 (-3.9,0.824436300940856)(-3.8,0.648124220846293)(-3.7,0.54520084705538)(-3.6,0.472401947673393)(-3.5,0.416169202831292)(-3.4,0.370463615735637)(-3.3,0.332064056666194)(-3.2,0.299048065134134)(-3.1,0.270175891753477)(-3,0.244601330559427)(-2.9,0.221721112119886)(-2.8,0.20109010783208)(-2.7,0.182370593738008)(-2.6,0.165300371799717)(-2.5,0.149671917559193)(-2.4,0.135318272882773)(-2.3,0.122103224635857)(-2.2,0.109914296559583)(-2.1,0.0986576402563968)(-2,0.0882542400610606)(-1.9,0.0786370468149463)(-1.8,0.0697487811454006)(-1.7,0.0615402277048288)(-1.6,0.053968895110108)(-1.5,0.0469979521866832)(-1.4,0.0405953757225536)(-1.3,0.0347332621066663)(-1.2,0.0293872674001494)(-1.1,0.0245361491476362)(-1,0.0201613896237549)(-0.9,0.016246884928227)(-0.8,0.0127786878707594)(-0.7,0.00974479525564405)(-0.6,0.00713497221856434)(-0.5,0.00494060785025597)(-0.4,0.00315459758286232)(-0.3,0.00177124880237771)(-0.2,0.000786206949221938)(-0.0999999999999996,0.000196400027557497)(0,0)(0.0999999999999996,0.000196400027557497)(0.2,0.000786206949221938)(0.3,0.00177124880237771)(0.4,0.00315459758286232)(0.5,0.00494060785025597)(0.6,0.00713497221856434)(0.7,0.00974479525564405)(0.8,0.0127786878707594)(0.9,0.016246884928227)(1,0.0201613896237549)(1.1,0.0245361491476362)(1.2,0.0293872674001494)(1.3,0.0347332621066663)(1.4,0.0405953757225536)(1.5,0.0469979521866832)(1.6,0.053968895110108)(1.7,0.0615402277048288)(1.8,0.0697487811454006)(1.9,0.0786370468149463)(2,0.0882542400610606)(2.1,0.0986576402563968)(2.2,0.109914296559583)(2.3,0.122103224635857)(2.4,0.135318272882773)(2.5,0.149671917559193)(2.6,0.165300371799717)(2.7,0.182370593738008)(2.8,0.20109010783208)(2.9,0.221721112119886)(3,0.244601330559427)(3.1,0.270175891753477)(3.2,0.299048065134134)(3.3,0.332064056666194)(3.4,0.370463615735637)(3.5,0.416169202831292)(3.6,0.472401947673393)(3.7,0.54520084705538)(3.8,0.648124220846293)(3.9,0.824436300940856) 
};

\end{axis}
\end{tikzpicture}

%% file: dP3inv.tikz
%
%
%
%
\begin{tikzpicture}

\begin{axis}[%
width= 0.7\textwidth,
height = 0.7\textwidth,
scale only axis,
xmin=-4, xmax=4,
xlabel={$z_{3}$},
ymin=-4, ymax=4,
ylabel={$\nabla \mc{P}_{3}^{-1}(z_{3})$},
axis on top]
\addplot [
color=blue,
solid,
line width=1.0pt
]
coordinates{
 (-4,-3.93635128070379)(-3.9,-3.93471996824952)(-3.8,-3.93300285516556)(-3.7,-3.93119299115653)(-3.6,-3.92928265461357)(-3.5,-3.92726324259481)(-3.4,-3.92512514142683)(-3.3,-3.92285757382541)(-3.2,-3.92044841741129)(-3.1,-3.91788398817802)(-3,-3.91514878075572)(-2.9,-3.91222515507191)(-2.8,-3.90909295604946)(-2.7,-3.90572904903727)(-2.6,-3.90210674836596)(-2.5,-3.89819510921607)(-2.4,-3.89395804309383)(-2.3,-3.8893532034667)(-2.2,-3.88433056877779)(-2.1,-3.87883062249306)(-2,-3.87278198995059)(-1.9,-3.86609833314741)(-1.8,-3.85867421689492)(-1.7,-3.85037952606979)(-1.6,-3.84105180555567)(-1.5,-3.83048556292043)(-1.4,-3.81841703200112)(-1.3,-3.80450198422242)(-1.2,-3.78828259152328)(-1.1,-3.76913649298956)(-1,-3.74619586055557)(-0.9,-3.71821370018178)(-0.8,-3.68333260671548)(-0.7,-3.63866211759307)(-0.6,-3.57945234630997)(-0.5,-3.49733633448799)(-0.4,-3.37616695698105)(-0.3,-3.18066894120347)(-0.2,-2.81933105879653)(-0.0999999999999996,-2)(0,0)(0.0999999999999996,2)(0.2,2.81933105879653)(0.3,3.18066894120347)(0.4,3.37616695698105)(0.5,3.49733633448799)(0.6,3.57945234630997)(0.7,3.63866211759307)(0.8,3.68333260671548)(0.9,3.71821370018178)(1,3.74619586055557)(1.1,3.76913649298956)(1.2,3.78828259152328)(1.3,3.80450198422242)(1.4,3.81841703200112)(1.5,3.83048556292043)(1.6,3.84105180555567)(1.7,3.85037952606979)(1.8,3.85867421689492)(1.9,3.86609833314741)(2,3.87278198995059)(2.1,3.87883062249306)(2.2,3.88433056877779)(2.3,3.8893532034667)(2.4,3.89395804309383)(2.5,3.89819510921607)(2.6,3.90210674836596)(2.7,3.90572904903727)(2.8,3.90909295604946)(2.9,3.91222515507191)(3,3.91514878075572)(3.1,3.91788398817802)(3.2,3.92044841741129)(3.3,3.92285757382541)(3.4,3.92512514142683)(3.5,3.92726324259481)(3.6,3.92928265461357)(3.7,3.93119299115653)(3.8,3.93300285516556)(3.9,3.93471996824952)(4,3.93635128070379) 
};

\end{axis}
\end{tikzpicture}

%% file: x.tikz
%
%
%
%
\begin{tikzpicture}

\begin{axis}[%
width=0.7\textwidth,
height=0.7\textwidth,
scale only axis,
xmin=0, xmax=6,
xlabel={Time},
ymin=0.5, ymax=2,
ylabel={$x(t)$},
axis on top]
\addplot [
color=blue,
solid,
line width=1.0pt
]
coordinates{
 (0,2)(0.01,1.86563894622623)(0.02,1.7504416995796)(0.03,1.65174082673824)(0.04,1.56723895491672)(0.05,1.4949568028269)(0.06,1.43319001285402)(0.07,1.3804710792954)(0.08,1.33553511862839)(0.09,1.29729366742393)(0.1,1.26480899465763)(0.11,1.23727315766375)(0.12,1.2139912393255)(0.13,1.19436367793584)(0.14,1.1778753924987)(0.15,1.16408168578544)(0.16,1.15260074486434)(0.17,1.1431027539069)(0.18,1.13530498297763)(0.19,1.12896273987434)(0.2,1.12386646692648)(0.21,1.11983531604725)(0.22,1.11671343613463)(0.23,1.11436724432311)(0.24,1.11268141511424)(0.25,1.11155664616641)(0.26,1.11090878111113)(0.27,1.11066242792404)(0.28,1.11075625918724)(0.29,1.11113589745751)(0.3,1.11175449119436)(0.31,1.11257238135086)(0.32,1.11355556339794)(0.33,1.11467389799248)(0.34,1.11590266531331)(0.35,1.11721993270704)(0.36,1.11860753363812)(0.37,1.1200494457845)(0.38,1.12153220376276)(0.39,1.12304427278736)(0.4,1.12457583777528)(0.41,1.12611852102197)(0.42,1.12766521232101)(0.43,1.12920990008574)(0.44,1.13074750874247)(0.45,1.13227376699478)(0.46,1.13378509803763)(0.47,1.13527853331427)(0.48,1.13675151519011)(0.49,1.13820230304797)(0.5,1.13962875022032)(0.51,1.14102956579538)(0.52,1.14240434233338)(0.53,1.14375123231383)(0.54,1.14507066860745)(0.55,1.14636135207341)(0.56,1.14762343800969)(0.57,1.14885670072249)(0.58,1.15006072779719)(0.59,1.15123620056592)(0.6,1.15238233635305)(0.61,1.15350028680438)(0.62,1.15458929009142)(0.63,1.15565046400919)(0.64,1.15668345229165)(0.65,1.15768899045605)(0.66,1.15866731656628)(0.67,1.15961865133273)(0.68,1.1605438170958)(0.69,1.16144259618029)(0.7,1.16231617886386)(0.71,1.16316421540502)(0.72,1.16398781484471)(0.73,1.16478696644094)(0.74,1.16556223693848)(0.75,1.16631420126596)(0.76,1.16704321930894)(0.77,1.16774973087382)(0.78,1.16843418040295)(0.79,1.16909700773732)(0.8,1.16973865092897)(0.81,1.17035953164774)(0.82,1.1709600649299)(0.83,1.17154066139076)(0.84,1.17210179860996)(0.85,1.17264392706037)(0.86,1.17316737615299)(0.87,1.17367257022738)(0.88,1.17415989660212)(0.89,1.17462980611836)(0.9,1.17508269240353)(0.91,1.17551879523593)(0.92,1.17593868547905)(0.93,1.17634256541357)(0.94,1.17673089071635)(0.95,1.17710404260478)(0.96,1.17746229460047)(0.97,1.17780590981667)(0.98,1.1781354380129)(0.99,1.17845115619561)(1,1.17875330600737)(1.01,1.17904213949789)(1.02,1.17931812544352)(1.03,1.17958166401407)(1.04,1.17983268595844)(1.05,1.18007180621445)(1.06,1.18029943757956)(1.07,1.18051492319456)(1.08,1.18072035889647)(1.09,1.18091452715831)(1.1,1.18109805033761)(1.11,1.18127172201978)(1.12,1.18143491837798)(1.13,1.18158853844966)(1.14,1.18173260474208)(1.15,1.18186724603213)(1.16,1.1819931894419)(1.17,1.18210998371086)(1.18,1.1822184198707)(1.19,1.1823185798957)(1.2,1.18241037579548)(1.21,1.18249472234013)(1.22,1.18257097669204)(1.23,1.18264023134833)(1.24,1.18270194175314)(1.25,1.18275676247962)(1.26,1.18280484385076)(1.27,1.18284597424883)(1.28,1.18288115476534)(1.29,1.18290961778789)(1.3,1.18293232513103)(1.31,1.18294911168872)(1.32,1.18295989892123)(1.33,1.1829656964709)(1.34,1.18296552519105)(1.35,1.18296065854586)(1.36,1.18295051480343)(1.37,1.18293538877856)(1.38,1.18291593257615)(1.39,1.1828913343426)(1.4,1.18286295251918)(1.41,1.18282976017132)(1.42,1.18279283248968)(1.43,1.18275170422737)(1.44,1.18270675723281)(1.45,1.18265822524178)(1.46,1.18260582312521)(1.47,1.18255036846858)(1.48,1.18249109661436)(1.49,1.18242912196926)(1.5,1.18236358449048)(1.51,1.18229546178245)(1.52,1.18222421620162)(1.53,1.18215039046691)(1.54,1.18207389164791)(1.55,1.18199476203657)(1.56,1.18191328803832)(1.57,1.18182940093136)(1.58,1.18174328588951)(1.59,1.1816550224251)(1.6,1.18156472225691)(1.61,1.18147247884247)(1.62,1.18137838109213)(1.63,1.18128251654432)(1.64,1.18118496785912)(1.65,1.18108580746699)(1.66,1.18098515118734)(1.67,1.18088295216026)(1.68,1.18077959890084)(1.69,1.18067455308648)(1.7,1.18056884074902)(1.71,1.1804616226474)(1.72,1.1803534261244)(1.73,1.18024440027732)(1.74,1.18013403916812)(1.75,1.18002332779648)(1.76,1.17991124303786)(1.77,1.17979887976238)(1.78,1.17968563166604)(1.79,1.17957171690289)(1.8,1.17945721509486)(1.81,1.1793421864908)(1.82,1.17922668672985)(1.83,1.17911076521749)(1.84,1.17899447551763)(1.85,1.17887788042814)(1.86,1.17876100786842)(1.87,1.17864399689923)(1.88,1.17852668942091)(1.89,1.1784092847331)(1.9,1.17829177779303)(1.91,1.17817419608366)(1.92,1.17805661224417)(1.93,1.17793897407167)(1.94,1.17782145848267)(1.95,1.17770391939947)(1.96,1.17758662490996)(1.97,1.17746942381188)(1.98,1.17735240160644)(1.99,1.1772356296505)(2,1.17711912013302)(2.01,1.17700289015366)(2.02,1.17688699550408)(2.03,1.17677143456378)(2.04,1.17665623118171)(2.05,1.17654141285547)(2.06,1.17642700559064)(2.07,1.17631303805375)(2.08,1.17619950419601)(2.09,1.1760865787707)(2.1,1.17597390701954)(2.11,1.17586217486522)(2.12,1.17575044050144)(2.13,1.17563964505927)(2.14,1.17552926260297)(2.15,1.17541942892195)(2.16,1.17531043908966)(2.17,1.17520169779314)(2.18,1.17509408372664)(2.19,1.17498660777343)(2.2,1.17488031666723)(2.21,1.17477427265903)(2.22,1.17466925384307)(2.23,1.17456475276622)(2.24,1.1744610070334)(2.25,1.17435810636305)(2.26,1.17425569057996)(2.27,1.17415438656006)(2.28,1.17405343326689)(2.29,1.173953496079)(2.3,1.17385430159319)(2.31,1.17375590481929)(2.32,1.17365829968136)(2.33,1.1735614913153)(2.34,1.17346548494101)(2.35,1.1733702879771)(2.36,1.17327589899062)(2.37,1.17318231500853)(2.38,1.17308952685833)(2.39,1.17299749469243)(2.4,1.17290634480999)(2.41,1.17281606409654)(2.42,1.17272662378295)(2.43,1.17263805048587)(2.44,1.17255034297712)(2.45,1.17246340613201)(2.46,1.17237738213224)(2.47,1.17229217751404)(2.48,1.17220788854297)(2.49,1.1721245114986)(2.5,1.17204172418068)(2.51,1.17196007540948)(2.52,1.17187924592213)(2.53,1.17179901856579)(2.54,1.17171983363605)(2.55,1.17164140339798)(2.56,1.17156386368326)(2.57,1.17148715457274)(2.58,1.17141120945024)(2.59,1.17133615536188)(2.6,1.17126192329316)(2.61,1.17118853884851)(2.62,1.17111597363432)(2.63,1.1710442025985)(2.64,1.17097331949645)(2.65,1.17090325545777)(2.66,1.1708340412844)(2.67,1.17076571482361)(2.68,1.17069816118721)(2.69,1.17063137179402)(2.7,1.1705659114689)(2.71,1.17050073061628)(2.72,1.17043577336292)(2.73,1.17037247768165)(2.74,1.17030938185617)(2.75,1.1702475585411)(2.76,1.17018619590974)(2.77,1.17012587881272)(2.78,1.1700660996295)(2.79,1.17000731332592)(2.8,1.16994904272049)(2.81,1.16989181689316)(2.82,1.16983499159786)(2.83,1.16977932252796)(2.84,1.16972396868306)(2.85,1.16966973576091)(2.86,1.1696159727737)(2.87,1.16956298230836)(2.88,1.16951089145493)(2.89,1.16945909596801)(2.9,1.16940857931956)(2.91,1.16935820640179)(2.92,1.16930894043468)(2.93,1.16926019696058)(2.94,1.16921199197635)(2.95,1.16916486154328)(2.96,1.16911783681003)(2.97,1.16907210155567)(2.98,1.16902648634965)(2.99,1.16898172841452)(3,1.1689374059422)(3.01,1.16816918354426)(3.02,1.16693116628594)(3.03,1.16530644352976)(3.04,1.16336609713539)(3.05,1.1611710148406)(3.06,1.15877351395539)(3.07,1.15621802579185)(3.08,1.15354275989854)(3.09,1.15078017233707)(3.1,1.14795780980186)(3.11,1.14509939741299)(3.12,1.14222450112158)(3.13,1.13935039736811)(3.14,1.13649077703331)(3.15,1.13365808179659)(3.16,1.13086171793651)(3.17,1.12811051925781)(3.18,1.12541071758563)(3.19,1.12276853154089)(3.2,1.12018780902658)(3.21,1.11767281226243)(3.22,1.11522579939335)(3.23,1.11284949434701)(3.24,1.11054493985636)(3.25,1.10831348758254)(3.26,1.1061557525243)(3.27,1.10407209990165)(3.28,1.10206257479955)(3.29,1.1001269611885)(3.3,1.09826482053016)(3.31,1.09647553307326)(3.32,1.09475832813684)(3.33,1.09311231288527)(3.34,1.09153647110766)(3.35,1.09002969665906)(3.36,1.08859078730674)(3.37,1.08721896909334)(3.38,1.08591212405019)(3.39,1.08466985814093)(3.4,1.08349054986058)(3.41,1.08237236259541)(3.42,1.08131494270952)(3.43,1.08031581855482)(3.44,1.07937458253705)(3.45,1.07848934059675)(3.46,1.0776588501892)(3.47,1.07688209469362)(3.48,1.07615708221478)(3.49,1.07548339912228)(3.5,1.07485873660686)(3.51,1.07428280361978)(3.52,1.07375348567301)(3.53,1.0732701692833)(3.54,1.07283126495678)(3.55,1.07243563134272)(3.56,1.07208229880074)(3.57,1.07176961600199)(3.58,1.07149711140796)(3.59,1.07126285915902)(3.6,1.07106653970862)(3.61,1.07090640864803)(3.62,1.07078182211292)(3.63,1.07069163530727)(3.64,1.07063468042362)(3.65,1.07061005499961)(3.66,1.07061674929161)(3.67,1.0706538034913)(3.68,1.07072027065821)(3.69,1.07081522322163)(3.7,1.07093775088703)(3.71,1.07108696727497)(3.72,1.07126200828107)(3.73,1.07146206293444)(3.74,1.07168618872715)(3.75,1.07193368413622)(3.76,1.07220365440744)(3.77,1.07249531010132)(3.78,1.07280789617284)(3.79,1.07314072554157)(3.8,1.07349306613178)(3.81,1.07386406002002)(3.82,1.07425323087926)(3.83,1.07465966505336)(3.84,1.07508292758774)(3.85,1.07552210725728)(3.86,1.07597687205818)(3.87,1.0764462867177)(3.88,1.07692998388173)(3.89,1.07742730200625)(3.9,1.07793743722083)(3.91,1.07846034588553)(3.92,1.07899487942197)(3.93,1.07954084856968)(3.94,1.0800980228683)(3.95,1.0806647930344)(3.96,1.08124204796377)(3.97,1.08182878884495)(3.98,1.08242304716439)(3.99,1.08302679466231)(4,1.08363818279761)(4.01,1.08425708560834)(4.02,1.08488321675718)(4.03,1.08551583297129)(4.04,1.0861548320982)(4.05,1.08679963437207)(4.06,1.08744966502106)(4.07,1.08810487910178)(4.08,1.08876536080905)(4.09,1.08942965039474)(4.1,1.0900978176346)(4.11,1.09077039864494)(4.12,1.09144583887718)(4.13,1.09212419315851)(4.14,1.0928057048542)(4.15,1.09348908784052)(4.16,1.09417491059)(4.17,1.09486265380373)(4.18,1.09555154712049)(4.19,1.09624229376068)(4.2,1.09693359467487)(4.21,1.0976258727635)(4.22,1.098318900793)(4.23,1.09901169993483)(4.24,1.09970524105799)(4.25,1.100398023424)(4.26,1.10109080177963)(4.27,1.10178297639579)(4.28,1.10247412844616)(4.29,1.1031647452102)(4.3,1.10385365229413)(4.31,1.10454175203872)(4.32,1.10522782124858)(4.33,1.10591259625714)(4.34,1.10659522020059)(4.35,1.10727597981066)(4.36,1.10795454594293)(4.37,1.10863068344579)(4.38,1.10930454536064)(4.39,1.10997557276788)(4.4,1.11064405570073)(4.41,1.11130946602445)(4.42,1.11197201556321)(4.43,1.11263145988583)(4.44,1.11328771009033)(4.45,1.11394059569586)(4.46,1.1145900675316)(4.47,1.11523600967021)(4.48,1.11587831442931)(4.49,1.11651688930093)(4.5,1.11715164723532)(4.51,1.11778250472541)(4.52,1.11840938547228)(4.53,1.11903220330384)(4.54,1.11965092051668)(4.55,1.12026544717)(4.56,1.12087588721681)(4.57,1.12148147612029)(4.58,1.1220832128653)(4.59,1.12268018487724)(4.6,1.12327271302923)(4.61,1.12386090468709)(4.62,1.12444409728051)(4.63,1.12502319623098)(4.64,1.12559708171558)(4.65,1.1261667287224)(4.66,1.12673130348786)(4.67,1.12729122985105)(4.68,1.12784639581389)(4.69,1.128396478848)(4.7,1.12894189775074)(4.71,1.12948243473373)(4.72,1.13001806336182)(4.73,1.13054876561301)(4.74,1.13107452621403)(4.75,1.13159532580797)(4.76,1.13211115113493)(4.77,1.13262195370815)(4.78,1.13312775236663)(4.79,1.13362854104057)(4.8,1.1341244081903)(4.81,1.13461525005832)(4.82,1.13510107611834)(4.83,1.13558193386595)(4.84,1.13605782318983)(4.85,1.13652881535611)(4.86,1.13699459778904)(4.87,1.13745549926813)(4.88,1.13791143473397)(4.89,1.13836238654391)(4.9,1.1388083576146)(4.91,1.13924938013457)(4.92,1.13968545321999)(4.93,1.14011659823179)(4.94,1.14054284423806)(4.95,1.14096420678684)(4.96,1.14138070281764)(4.97,1.14179234990169)(4.98,1.14219915913882)(4.99,1.14260109098793)(5,1.14299838115161)(5.01,1.14339061459311)(5.02,1.14377841133803)(5.03,1.14416146842868)(5.04,1.14453961536655)(5.05,1.14491347030451)(5.06,1.14528217441059)(5.07,1.14564679489636)(5.08,1.1460062934897)(5.09,1.14636166500351)(5.1,1.14671217925022)(5.11,1.14705835695659)(5.12,1.14740004310714)(5.13,1.1477371395028)(5.14,1.14807009617167)(5.15,1.14839828018027)(5.16,1.14872253956243)(5.17,1.14904204913373)(5.18,1.1493575940472)(5.19,1.14966878774698)(5.2,1.1499757010191)(5.21,1.15027841644036)(5.22,1.15057697115026)(5.23,1.15087140055293)(5.24,1.15116173914624)(5.25,1.15144802094627)(5.26,1.15173028645694)(5.27,1.15200860729305)(5.28,1.15228295430868)(5.29,1.1525534320706)(5.3,1.1528200121043)(5.31,1.15308270324105)(5.32,1.15334155107783)(5.33,1.15359662562478)(5.34,1.15384802881277)(5.35,1.15409564457753)(5.36,1.15433960791264)(5.37,1.15457994471475)(5.38,1.15481667324237)(5.39,1.1550498120904)(5.4,1.155279446181)(5.41,1.15550560783187)(5.42,1.15572828199524)(5.43,1.15594751235011)(5.44,1.1561633975218)(5.45,1.15637601985205)(5.46,1.1565853042048)(5.47,1.15679114302342)(5.48,1.15699393267393)(5.49,1.15719353166629)(5.5,1.15738971908374)(5.51,1.15758329828268)(5.52,1.15777342903815)(5.53,1.15796019901892)(5.54,1.15814448190366)(5.55,1.15832529085854)(5.56,1.15850362136495)(5.57,1.15867858285772)(5.58,1.15885098818621)(5.59,1.15902037785431)(5.6,1.1591868958449)(5.61,1.15935082314705)(5.62,1.15951192865741)(5.63,1.1596696600011)(5.64,1.15982588856855)(5.65,1.15997934870969)(5.66,1.16012959798635)(5.67,1.16027773227334)(5.68,1.16042284814459)(5.69,1.16056575855373)(5.7,1.16070610283914)(5.71,1.16084383630942)(5.72,1.16097953857639)(5.73,1.16111235653013)(5.74,1.1612432104187)(5.75,1.16137161465663)(5.76,1.16149754805398)(5.77,1.1616216841298)(5.78,1.16174298296647)(5.79,1.16186257851324)(5.8,1.161979788831)(5.81,1.16209473886864)(5.82,1.16220804395067)(5.83,1.16231860381052)(5.84,1.16242776802149)(5.85,1.16253440296697)(5.86,1.16263932164902)(5.87,1.16274223389284)(5.88,1.16284305338412)(5.89,1.16294226658158)(5.9,1.16303917849588)(5.91,1.1631347150439)(5.92,1.16322790659603)(5.93,1.1633197861471)(5.94,1.16340946521208)(5.95,1.16349773792216)(5.96,1.1635840475625)(5.97,1.16366880572367)(5.98,1.16375181352795)(5.99,1.16383311621764)(6,1.16391285306214) 
};

\addplot [
color=green!50!black,
solid,
line width=1.0pt
]
coordinates{
 (0,1)(0.01,1.00800380786037)(0.02,1.01487906968766)(0.03,1.0207869533549)(0.04,1.02586623315512)(0.05,1.0302364041122)(0.06,1.03400031326228)(0.07,1.03724645859703)(0.08,1.040051004654)(0.09,1.04247944948913)(0.1,1.04458811187391)(0.11,1.04642541329374)(0.12,1.04803294492667)(0.13,1.04944643334344)(0.14,1.05069654507251)(0.15,1.0518095783075)(0.16,1.05280809157867)(0.17,1.05371139036242)(0.18,1.05453600410116)(0.19,1.05529606504329)(0.2,1.05600363166674)(0.21,1.05666898900653)(0.22,1.05730089653112)(0.23,1.0579067897618)(0.24,1.05849296570675)(0.25,1.0590647810173)(0.26,1.05962665752451)(0.27,1.06018250835557)(0.28,1.06073537524958)(0.29,1.06128790350545)(0.3,1.06184229258677)(0.31,1.06240029099508)(0.32,1.06296331872762)(0.33,1.06353261936519)(0.34,1.06410905125997)(0.35,1.0646933436131)(0.36,1.06528600658602)(0.37,1.06588743785137)(0.38,1.06649788145056)(0.39,1.06711747749026)(0.4,1.06774627419637)(0.41,1.06838424301207)(0.42,1.06903129017599)(0.43,1.06968726758138)(0.44,1.07035198192734)(0.45,1.07102520296813)(0.46,1.07170667154909)(0.47,1.07239610695798)(0.48,1.07309318228182)(0.49,1.07379765169196)(0.5,1.07450906914167)(0.51,1.07522710533126)(0.52,1.07595156167348)(0.53,1.07668187719049)(0.54,1.07741790716144)(0.55,1.07815914157914)(0.56,1.07890531353278)(0.57,1.07965606792048)(0.58,1.08041099976154)(0.59,1.08116988658832)(0.6,1.08193223692523)(0.61,1.08269789996481)(0.62,1.08346637195353)(0.63,1.08423749798232)(0.64,1.0850108374029)(0.65,1.08578617085361)(0.66,1.08656317176044)(0.67,1.08734151860532)(0.68,1.08812102587318)(0.69,1.08890125970462)(0.7,1.08968216406751)(0.71,1.09046323946713)(0.72,1.09124446406157)(0.73,1.09202540821417)(0.74,1.09280589140097)(0.75,1.09358573326816)(0.76,1.09436467579159)(0.77,1.09514249710204)(0.78,1.09591898589432)(0.79,1.09669393787233)(0.8,1.09746715516162)(0.81,1.09823844647629)(0.82,1.09900762692909)(0.83,1.09977451802445)(0.84,1.1005389448039)(0.85,1.10130074088013)(0.86,1.10205974842302)(0.87,1.10281581243556)(0.88,1.10356878709943)(0.89,1.10431852354062)(0.9,1.10506486900708)(0.91,1.10580774121814)(0.92,1.10654692925515)(0.93,1.10728243158067)(0.94,1.10801395905537)(0.95,1.10874163389984)(0.96,1.10946514331493)(0.97,1.11018430800398)(0.98,1.11089949034653)(0.99,1.11161014378089)(1,1.11231601776624)(1.01,1.11301758749715)(1.02,1.11371454727265)(1.03,1.11440660829704)(1.04,1.11509399645834)(1.05,1.1157765039105)(1.06,1.11645401176184)(1.07,1.11712661356103)(1.08,1.11779406862917)(1.09,1.11845643390848)(1.1,1.11911364167922)(1.11,1.11976560853742)(1.12,1.12041233021825)(1.13,1.12105373805229)(1.14,1.1216897887628)(1.15,1.12232048406584)(1.16,1.12294571643826)(1.17,1.12356554503434)(1.18,1.12417986817831)(1.19,1.12478868483082)(1.2,1.12539201388944)(1.21,1.12598972393188)(1.22,1.12658195358778)(1.23,1.12716850392052)(1.24,1.12774953639159)(1.25,1.12832492171686)(1.26,1.1288946687195)(1.27,1.12945888472063)(1.28,1.13001732198555)(1.29,1.13057027264365)(1.3,1.13111749115665)(1.31,1.13165908576937)(1.32,1.13219513933559)(1.33,1.1327253980313)(1.34,1.13325021280423)(1.35,1.13376926646762)(1.36,1.13428277898218)(1.37,1.13479072437228)(1.38,1.13529297646297)(1.39,1.13578981149182)(1.4,1.1362809142123)(1.41,1.13676662463764)(1.42,1.13724669620967)(1.43,1.13772132468763)(1.44,1.13819045420142)(1.45,1.13865407782842)(1.46,1.13911234164487)(1.47,1.13956506453743)(1.48,1.14001252425224)(1.49,1.14045447583174)(1.5,1.14089119896861)(1.51,1.14132250089284)(1.52,1.1417485872898)(1.53,1.14216934781881)(1.54,1.1425849109483)(1.55,1.14299524814081)(1.56,1.14340041147012)(1.57,1.14380042860689)(1.58,1.14419533287729)(1.59,1.14458515676739)(1.6,1.14496993351743)(1.61,1.14534969697591)(1.62,1.14572448148403)(1.63,1.14609432218103)(1.64,1.1464592536875)(1.65,1.14681930949345)(1.66,1.14717453599663)(1.67,1.1475249405664)(1.68,1.14787062528772)(1.69,1.14821152127986)(1.7,1.14854780589797)(1.71,1.14887940003844)(1.72,1.14920638840027)(1.73,1.14952882292502)(1.74,1.14984669347386)(1.75,1.15016010248465)(1.76,1.15046902889609)(1.77,1.15077355912028)(1.78,1.15107370758613)(1.79,1.15136951983116)(1.8,1.15166103736276)(1.81,1.15194830056703)(1.82,1.15223134976584)(1.83,1.1525102254355)(1.84,1.15278496803446)(1.85,1.1530556180483)(1.86,1.15332221659255)(1.87,1.15358480260384)(1.88,1.15384342193212)(1.89,1.15409811028642)(1.9,1.15434890956606)(1.91,1.15459586101692)(1.92,1.15483900248798)(1.93,1.15507838177584)(1.94,1.15531402799767)(1.95,1.155545998492)(1.96,1.15577430857018)(1.97,1.15599901494711)(1.98,1.15622015870349)(1.99,1.15643777667993)(2,1.15665190878051)(2.01,1.15686259497297)(2.02,1.15706987519594)(2.03,1.15727378898488)(2.04,1.1574743756535)(2.05,1.15767167451276)(2.06,1.15786572469867)(2.07,1.15805656554562)(2.08,1.15824423185114)(2.09,1.15842878137703)(2.1,1.15861020573643)(2.11,1.15878862263715)(2.12,1.15896395870435)(2.13,1.1591363564551)(2.14,1.15930579258701)(2.15,1.15947231733778)(2.16,1.15963599860658)(2.17,1.15979680461026)(2.18,1.15995487052617)(2.19,1.16011011507057)(2.2,1.1602626993287)(2.21,1.16041253924529)(2.22,1.16055977205618)(2.23,1.16070435920259)(2.24,1.16084637189288)(2.25,1.16098585339734)(2.26,1.16112277979824)(2.27,1.16125729519828)(2.28,1.16138928032936)(2.29,1.16151890826116)(2.3,1.16164614946516)(2.31,1.1617710499369)(2.32,1.16189363965426)(2.33,1.16201395044649)(2.34,1.16213201464031)(2.35,1.16224786419086)(2.36,1.16236153083988)(2.37,1.16247304590414)(2.38,1.16258244035202)(2.39,1.16268974665213)(2.4,1.16279499263937)(2.41,1.16289820912103)(2.42,1.16299942772065)(2.43,1.16309867597342)(2.44,1.16319597882756)(2.45,1.16329138611152)(2.46,1.163384902863)(2.47,1.16347656995506)(2.48,1.16356641061122)(2.49,1.1636544945867)(2.5,1.16374070426206)(2.51,1.1638252657018)(2.52,1.16390808611791)(2.53,1.16398917381036)(2.54,1.16406866482772)(2.55,1.16414648111218)(2.56,1.16422272581671)(2.57,1.16429736582249)(2.58,1.16437042930341)(2.59,1.16444198359584)(2.6,1.16451201960336)(2.61,1.16458057859301)(2.62,1.16464767132193)(2.63,1.16471333280066)(2.64,1.1647775761612)(2.65,1.16484043255318)(2.66,1.16490192009191)(2.67,1.16496206496775)(2.68,1.16502088340453)(2.69,1.16507841486087)(2.7,1.16513461941972)(2.71,1.16518961661508)(2.72,1.16524343040481)(2.73,1.1652959522572)(2.74,1.16534733657973)(2.75,1.16539750096986)(2.76,1.16544654949459)(2.77,1.1654944340065)(2.78,1.1655412393907)(2.79,1.16558691706336)(2.8,1.16563156795353)(2.81,1.16567510324366)(2.82,1.16571769029029)(2.83,1.1657591495473)(2.84,1.16579973659811)(2.85,1.1658392238725)(2.86,1.16587783381644)(2.87,1.16591548298471)(2.88,1.16595214566275)(2.89,1.16598803552851)(2.9,1.16602284716092)(2.91,1.16605695888295)(2.92,1.16609008322439)(2.93,1.16612241237523)(2.94,1.16615395754392)(2.95,1.16618457196709)(2.96,1.16621455451811)(2.97,1.16624357156476)(2.98,1.16627198174684)(2.99,1.16629957872612)(3,1.16636462850045)(3.01,1.22585544797129)(3.02,1.27645259696477)(3.03,1.3193934884876)(3.04,1.35574102523059)(3.05,1.38641080117868)(3.06,1.41219229827305)(3.07,1.43376473624126)(3.08,1.45171401043832)(3.09,1.46654459542689)(3.1,1.47869103062184)(3.11,1.4885281678847)(3.12,1.49637835571371)(3.13,1.50252010004136)(3.14,1.50719264786329)(3.15,1.51060289834864)(3.16,1.51292850848317)(3.17,1.5143231873121)(3.18,1.51491892178265)(3.19,1.51482991126097)(3.2,1.51415428659736)(3.21,1.51297688068259)(3.22,1.51137072068847)(3.23,1.50939884097005)(3.24,1.5071155929235)(3.25,1.50456786474352)(3.26,1.50179611630834)(3.27,1.49883526190466)(3.28,1.49571543767682)(3.29,1.49246266083039)(3.3,1.48909939764355)(3.31,1.48564505269716)(3.32,1.48211639063349)(3.33,1.4785279002939)(3.34,1.47489210146824)(3.35,1.47121982344364)(3.36,1.46752041928934)(3.37,1.46380208517782)(3.38,1.46007167576265)(3.39,1.45633542001059)(3.4,1.45259853612388)(3.41,1.44886549685543)(3.42,1.44514041538662)(3.43,1.44142654566254)(3.44,1.43772700545605)(3.45,1.43404429346184)(3.46,1.43038066292669)(3.47,1.42673810724823)(3.48,1.42311820867735)(3.49,1.41952257413871)(3.5,1.4159523085974)(3.51,1.4124086803548)(3.52,1.4088924947458)(3.53,1.40540469492735)(3.54,1.40194592323927)(3.55,1.39851681320517)(3.56,1.39511793787684)(3.57,1.39174964508635)(3.58,1.38841248803625)(3.59,1.38510658213636)(3.6,1.38183243989885)(3.61,1.37859007214885)(3.62,1.37537983679178)(3.63,1.37220183999665)(3.64,1.36905615206038)(3.65,1.36594290696719)(3.66,1.36286216262537)(3.67,1.35981396288664)(3.68,1.35679833229122)(3.69,1.35381527340112)(3.7,1.35086477139281)(3.71,1.34794679519483)(3.72,1.34506129912824)(3.73,1.34220822342086)(3.74,1.33938749990196)(3.75,1.33659904033709)(3.76,1.33384275407192)(3.77,1.33111853849571)(3.78,1.32842628253576)(3.79,1.32576585844523)(3.8,1.32313712918749)(3.81,1.32053999518502)(3.82,1.31797425162882)(3.83,1.31543982531868)(3.84,1.3129364587681)(3.85,1.31046414160685)(3.86,1.30802248873793)(3.87,1.30561157558676)(3.88,1.30323103841971)(3.89,1.30088072981638)(3.9,1.2985606145795)(3.91,1.29627021979639)(3.92,1.29400961178054)(3.93,1.29177847998914)(3.94,1.28957652830469)(3.95,1.28740382163881)(3.96,1.28525987920471)(3.97,1.28314456495346)(3.98,1.2810578484441)(3.99,1.27899924480487)(4,1.27696867387439)(4.01,1.27496589323177)(4.02,1.27299066722961)(4.03,1.27104279569843)(4.04,1.26912203608524)(4.05,1.26722815504984)(4.06,1.26536098919581)(4.07,1.26352022961121)(4.08,1.26170554884437)(4.09,1.25991695189987)(4.1,1.25815411027351)(4.11,1.25641658191843)(4.12,1.25470443266951)(4.13,1.25301732702983)(4.14,1.25135487428772)(4.15,1.24971712567928)(4.16,1.24810358134593)(4.17,1.24651407528814)(4.18,1.24494852159813)(4.19,1.24340638015806)(4.2,1.24188772834401)(4.21,1.24039212587059)(4.22,1.23891932587389)(4.23,1.23746930156556)(4.24,1.23604148130642)(4.25,1.23463597004186)(4.26,1.23325227400129)(4.27,1.23189025614035)(4.28,1.23054973400336)(4.29,1.22923029276686)(4.3,1.22793195760618)(4.31,1.22665419404839)(4.32,1.22539706384572)(4.33,1.22416006731524)(4.34,1.22294317481299)(4.35,1.22174602383053)(4.36,1.22056843035506)(4.37,1.21941018807247)(4.38,1.21827099431649)(4.39,1.21715071176624)(4.4,1.21604903616817)(4.41,1.2149657801765)(4.42,1.21390071196403)(4.43,1.21285360930958)(4.44,1.21182424605969)(4.45,1.21081239878274)(4.46,1.20981784757101)(4.47,1.20884037318234)(4.48,1.20787975777897)(4.49,1.20693578488327)(4.5,1.20600823971311)(4.51,1.20509690865526)(4.52,1.2042015808205)(4.53,1.20332203993463)(4.54,1.20245809252864)(4.55,1.20160951390767)(4.56,1.20077615412711)(4.57,1.19995761870198)(4.58,1.19915396482201)(4.59,1.19836481283437)(4.6,1.19759002378049)(4.61,1.19682943078338)(4.62,1.19608274723804)(4.63,1.19534989609924)(4.64,1.19463056652052)(4.65,1.19392466351289)(4.66,1.19323193044516)(4.67,1.19255221146893)(4.68,1.19188531719624)(4.69,1.19123104265011)(4.7,1.19058923182626)(4.71,1.18995969130054)(4.72,1.18934224110089)(4.73,1.18873670321587)(4.74,1.18814290139256)(4.75,1.18756066111241)(4.76,1.18698980947543)(4.77,1.18643017578503)(4.78,1.1858815901784)(4.79,1.18534388420469)(4.8,1.18481689048015)(4.81,1.18430044796428)(4.82,1.18379439471472)(4.83,1.18329856791201)(4.84,1.18281280263758)(4.85,1.18233694638347)(4.86,1.18187086632532)(4.87,1.1814143780252)(4.88,1.18096733694044)(4.89,1.18052959620102)(4.9,1.18010100499164)(4.91,1.17968141696965)(4.92,1.17927068663163)(4.93,1.17886867015164)(4.94,1.17847522555356)(4.95,1.17809021255455)(4.96,1.17771349257186)(4.97,1.17734492865613)(4.98,1.17698438473981)(4.99,1.17663172189996)(5,1.17628682740959)(5.01,1.17594953166109)(5.02,1.17561975910325)(5.03,1.17529735316941)(5.04,1.17498217110326)(5.05,1.17467414291877)(5.06,1.17437306296186)(5.07,1.1740789118574)(5.08,1.17379146170306)(5.09,1.17351069834828)(5.1,1.17323641891615)(5.11,1.17296856952634)(5.12,1.17270701042716)(5.13,1.17245161610367)(5.14,1.17220233455552)(5.15,1.17195895340408)(5.16,1.17172150718263)(5.17,1.17148972233215)(5.18,1.17126366337532)(5.19,1.17104312396742)(5.2,1.17082801414983)(5.21,1.17061824431158)(5.22,1.17041371689554)(5.23,1.17021433466933)(5.24,1.17002000182982)(5.25,1.16983062377773)(5.26,1.16964610735566)(5.27,1.16946636016213)(5.28,1.16929129376922)(5.29,1.16912081593351)(5.3,1.16895484010348)(5.31,1.16879328087922)(5.32,1.16863605381555)(5.33,1.16848307358816)(5.34,1.1683342444616)(5.35,1.16818950656988)(5.36,1.16804877114582)(5.37,1.16791194033117)(5.38,1.16777897332723)(5.39,1.16764976685899)(5.4,1.16752420351746)(5.41,1.16740238777412)(5.42,1.16728391293191)(5.43,1.16716897251626)(5.44,1.16705767765183)(5.45,1.16694947994109)(5.46,1.16684450960221)(5.47,1.16674294598213)(5.48,1.1666443773255)(5.49,1.16654883114216)(5.5,1.16645638677273)(5.51,1.16636676892047)(5.52,1.16628006605721)(5.53,1.16619621248358)(5.54,1.16611505216417)(5.55,1.16603661284387)(5.56,1.16596078045727)(5.57,1.16588754316106)(5.58,1.16581681403451)(5.59,1.16574854726998)(5.6,1.16568271203325)(5.61,1.16561920035857)(5.62,1.16555800060099)(5.63,1.16549915036708)(5.64,1.16544232167508)(5.65,1.16538766428627)(5.66,1.16533520275247)(5.67,1.16528470120759)(5.68,1.16523630742842)(5.69,1.16518979242569)(5.7,1.1651452144737)(5.71,1.16510254519261)(5.72,1.1650615823312)(5.73,1.16502255292821)(5.74,1.16498513388181)(5.75,1.16494944927219)(5.76,1.16491547530605)(5.77,1.16488296606673)(5.78,1.16485221726415)(5.79,1.16482285378156)(5.8,1.16479504660893)(5.81,1.16476873420812)(5.82,1.16474371430162)(5.83,1.16472025876711)(5.84,1.1646979805702)(5.85,1.16467715138704)(5.86,1.16465752849794)(5.87,1.1646391662035)(5.88,1.16462205910553)(5.89,1.16460604898213)(5.9,1.16459131057994)(5.91,1.16457755043929)(5.92,1.16456502849268)(5.93,1.16455342922011)(5.94,1.164542972233)(5.95,1.16453343039744)(5.96,1.16452492540253)(5.97,1.16451731534018)(5.98,1.16451067324187)(5.99,1.16450488393318)(6,1.1644999834813) 
};

\addplot [
color=red,
solid,
line width=1.0pt
]
coordinates{
 (0,0.5)(0.01,0.626357245913395)(0.02,0.73467923073274)(0.03,0.827472219906855)(0.04,0.906894811928158)(0.05,0.9748067930609)(0.06,1.0328096738837)(0.07,1.08228246210757)(0.08,1.12441387671761)(0.09,1.16022688308694)(0.1,1.19060289346846)(0.11,1.21630142904251)(0.12,1.23797581574783)(0.13,1.25618988872072)(0.14,1.27142806242879)(0.15,1.28410873590706)(0.16,1.29459116355699)(0.17,1.30318585573067)(0.18,1.31015901292121)(0.19,1.31574119508237)(0.2,1.32012990140678)(0.21,1.32349569494622)(0.22,1.32598566733426)(0.23,1.32772596591508)(0.24,1.328825619179)(0.25,1.3293785728163)(0.26,1.32946456136436)(0.27,1.32915506372039)(0.28,1.32850836556318)(0.29,1.32757619903703)(0.3,1.32640321621888)(0.31,1.32502732765405)(0.32,1.32348111787444)(0.33,1.32179348264233)(0.34,1.31998828342672)(0.35,1.31808672367986)(0.36,1.31610645977586)(0.37,1.31406311636413)(0.38,1.31196991478669)(0.39,1.30983824972238)(0.4,1.30767788802835)(0.41,1.30549723596595)(0.42,1.303303497503)(0.43,1.30110283233288)(0.44,1.29890050933019)(0.45,1.29670103003708)(0.46,1.29450823041328)(0.47,1.29232535972774)(0.48,1.29015530252808)(0.49,1.28800004526007)(0.5,1.28586218063801)(0.51,1.28374332887336)(0.52,1.28164409599314)(0.53,1.27956689049568)(0.54,1.27751142423111)(0.55,1.27547950634745)(0.56,1.27347124845753)(0.57,1.27148723135703)(0.58,1.26952827244127)(0.59,1.26759391284576)(0.6,1.26568542672172)(0.61,1.26380181323081)(0.62,1.26194433795506)(0.63,1.26011203800848)(0.64,1.25830571030544)(0.65,1.25652483869034)(0.66,1.25476951167329)(0.67,1.25303983006195)(0.68,1.25133515703102)(0.69,1.24965614411509)(0.7,1.24800165706864)(0.71,1.24637254512785)(0.72,1.24476772109372)(0.73,1.24318762534489)(0.74,1.24163187166055)(0.75,1.24010006546588)(0.76,1.23859210489947)(0.77,1.23710777202414)(0.78,1.23564683370273)(0.79,1.23420905439035)(0.8,1.23279419390941)(0.81,1.23140202187596)(0.82,1.23003230814101)(0.83,1.22868482058479)(0.84,1.22735925658613)(0.85,1.22605533205949)(0.86,1.224772875424)(0.87,1.22351161733707)(0.88,1.22227131629845)(0.89,1.22105167034102)(0.9,1.2198524385894)(0.91,1.21867346354593)(0.92,1.21751438526581)(0.93,1.21637500300576)(0.94,1.21525515022828)(0.95,1.21415432349538)(0.96,1.2130725620846)(0.97,1.21200978217935)(0.98,1.21096507164057)(0.99,1.20993870002349)(1,1.20893067622639)(1.01,1.20794027300497)(1.02,1.20696732728383)(1.03,1.2060117276889)(1.04,1.20507331758322)(1.05,1.20415168987505)(1.06,1.2032465506586)(1.07,1.20235846324441)(1.08,1.20148557247437)(1.09,1.20062903893322)(1.1,1.19978830798316)(1.11,1.1989626694428)(1.12,1.19815275140377)(1.13,1.19735772349805)(1.14,1.19657760649513)(1.15,1.19581226990202)(1.16,1.19506109411984)(1.17,1.1943244712548)(1.18,1.19360171195099)(1.19,1.19289273527349)(1.2,1.19219761031508)(1.21,1.19151555372798)(1.22,1.19084706972018)(1.23,1.19019126473115)(1.24,1.18954852185527)(1.25,1.18891831580352)(1.26,1.18830048742974)(1.27,1.18769514103054)(1.28,1.18710152324911)(1.29,1.18652010956846)(1.3,1.18595018371231)(1.31,1.18539180254191)(1.32,1.18484496174318)(1.33,1.1843089054978)(1.34,1.18378426200471)(1.35,1.18327007498652)(1.36,1.18276670621439)(1.37,1.18227388684916)(1.38,1.18179109096087)(1.39,1.18131885416558)(1.4,1.18085613326853)(1.41,1.18040361519105)(1.42,1.17996047130066)(1.43,1.179526971085)(1.44,1.17910278856577)(1.45,1.1786876969298)(1.46,1.17828183522992)(1.47,1.17788456699399)(1.48,1.1774963791334)(1.49,1.177116402199)(1.5,1.17674521654091)(1.51,1.17638203732471)(1.52,1.17602719650859)(1.53,1.17568026171428)(1.54,1.17534119740378)(1.55,1.17500998982262)(1.56,1.17468630049156)(1.57,1.17437017046175)(1.58,1.17406138123321)(1.59,1.17375982080751)(1.6,1.17346534422565)(1.61,1.17317782418161)(1.62,1.17289713742383)(1.63,1.17262316127466)(1.64,1.17235577845338)(1.65,1.17209488303957)(1.66,1.17184031281602)(1.67,1.17159210727334)(1.68,1.17134977581144)(1.69,1.17111392563366)(1.7,1.17088335335301)(1.71,1.17065897731416)(1.72,1.17044018547533)(1.73,1.17022677679766)(1.74,1.17001926735802)(1.75,1.16981656971887)(1.76,1.16961972806605)(1.77,1.16942756111734)(1.78,1.16924066074784)(1.79,1.16905876326594)(1.8,1.16888174754238)(1.81,1.16870951294217)(1.82,1.16854196350431)(1.83,1.16837900934701)(1.84,1.16822055644791)(1.85,1.16806650152356)(1.86,1.16791677553902)(1.87,1.16777120049693)(1.88,1.16762988864697)(1.89,1.16749260498047)(1.9,1.1673593126409)(1.91,1.16722994289942)(1.92,1.16710438526785)(1.93,1.1669826441525)(1.94,1.16686451351966)(1.95,1.16675008210853)(1.96,1.16663906651986)(1.97,1.16653156124101)(1.98,1.16642743969007)(1.99,1.16632659366958)(2,1.16622897108647)(2.01,1.16613451487337)(2.02,1.16604312929998)(2.03,1.16595477645134)(2.04,1.16586939316479)(2.05,1.16578691263178)(2.06,1.16570726971069)(2.07,1.16563039640063)(2.08,1.16555626395285)(2.09,1.16548463985227)(2.1,1.16541588724403)(2.11,1.16534920249763)(2.12,1.16528560079421)(2.13,1.16522399848563)(2.14,1.16516494481002)(2.15,1.16510825374027)(2.16,1.16505356230376)(2.17,1.16500149759659)(2.18,1.16495104574719)(2.19,1.164903277156)(2.2,1.16485698400407)(2.21,1.16481318809568)(2.22,1.16477097410076)(2.23,1.16473088803119)(2.24,1.16469262107372)(2.25,1.1646560402396)(2.26,1.1646215296218)(2.27,1.16458831824166)(2.28,1.16455728640376)(2.29,1.16452759565984)(2.3,1.16449954894165)(2.31,1.16447304524381)(2.32,1.16444806066438)(2.33,1.16442455823821)(2.34,1.16440250041868)(2.35,1.16438184783204)(2.36,1.1643625701695)(2.37,1.16434463908733)(2.38,1.16432803278965)(2.39,1.16431275865544)(2.4,1.16429866255064)(2.41,1.16428572678242)(2.42,1.16427394849639)(2.43,1.16426327354071)(2.44,1.16425367819532)(2.45,1.16424520775647)(2.46,1.16423771500476)(2.47,1.1642312525309)(2.48,1.16422570084581)(2.49,1.1642209939147)(2.5,1.16421757155726)(2.51,1.16421465888872)(2.52,1.16421266795996)(2.53,1.16421180762386)(2.54,1.16421150153624)(2.55,1.16421211548984)(2.56,1.16421341050003)(2.57,1.16421547960477)(2.58,1.16421836124635)(2.59,1.16422186104228)(2.6,1.16422605710348)(2.61,1.16423088255848)(2.62,1.16423635504375)(2.63,1.16424246460085)(2.64,1.16424910434236)(2.65,1.16425631198905)(2.66,1.16426403862369)(2.67,1.16427222020864)(2.68,1.16428095540826)(2.69,1.16429021334511)(2.7,1.16429946911138)(2.71,1.16430965276864)(2.72,1.16432079623228)(2.73,1.16433157006115)(2.74,1.1643432815641)(2.75,1.16435494048905)(2.76,1.16436725459567)(2.77,1.16437968718079)(2.78,1.16439266097981)(2.79,1.16440576961073)(2.8,1.16441938932598)(2.81,1.16443307986318)(2.82,1.16444731811185)(2.83,1.16446152792474)(2.84,1.16447629471883)(2.85,1.16449104036659)(2.86,1.16450619340986)(2.87,1.16452153470693)(2.88,1.16453696288233)(2.89,1.16455286850348)(2.9,1.16456857351952)(2.91,1.16458483471527)(2.92,1.16460097634094)(2.93,1.16461739066419)(2.94,1.16463405047973)(2.95,1.16465056648963)(2.96,1.16466760867186)(2.97,1.16468432687957)(2.98,1.16470153190351)(2.99,1.16471869285937)(3,1.16469796555735)(3.01,1.10597536848445)(3.02,1.0566162367493)(3.03,1.01530006798265)(3.04,0.980892877634022)(3.05,0.952418183980721)(3.06,0.929034187771566)(3.07,0.910017237966887)(3.08,0.894743229663139)(3.09,0.882675232236038)(3.1,0.873351159576306)(3.11,0.866372434702317)(3.12,0.861397143164714)(3.13,0.858129502590532)(3.14,0.856316575103408)(3.15,0.855739019854777)(3.16,0.856209773580321)(3.17,0.857566293430089)(3.18,0.859670360631716)(3.19,0.862401557198143)(3.2,0.865657904376059)(3.21,0.869350307054982)(3.22,0.873403479918185)(3.23,0.877751664682943)(3.24,0.882339467220132)(3.25,0.887118647673939)(3.26,0.892048131167365)(3.27,0.897092638193691)(3.28,0.902221987523625)(3.29,0.907410377981115)(3.3,0.91263578182629)(3.31,0.917879414229583)(3.32,0.923125281229672)(3.33,0.928359786820832)(3.34,0.933571427424096)(3.35,0.9387504798973)(3.36,0.943888793403916)(3.37,0.948978945728842)(3.38,0.954016200187157)(3.39,0.95899472184848)(3.4,0.963910914015539)(3.41,0.968762140549159)(3.42,0.973544641903859)(3.43,0.978257635782637)(3.44,0.9828984120069)(3.45,0.987466365941403)(3.46,0.991960486884116)(3.47,0.996379798058149)(3.48,1.00072470910787)(3.49,1.00499402673902)(3.5,1.00918895479574)(3.51,1.01330851602541)(3.52,1.01735401958119)(3.53,1.02132513578935)(3.54,1.02522281180394)(3.55,1.0290475554521)(3.56,1.03279976332243)(3.57,1.03648073891166)(3.58,1.04009040055579)(3.59,1.04363055870462)(3.6,1.04710102039254)(3.61,1.05050351920312)(3.62,1.05383834109531)(3.63,1.05710652469608)(3.64,1.06030916751601)(3.65,1.06344703803321)(3.66,1.06652108808302)(3.67,1.06953223362206)(3.68,1.07248139705057)(3.69,1.07536950337725)(3.7,1.07819747772016)(3.71,1.0809662375302)(3.72,1.08367669259069)(3.73,1.0863297136447)(3.74,1.0889263113709)(3.75,1.09146727552669)(3.76,1.09395359152064)(3.77,1.09638615140297)(3.78,1.09876582129139)(3.79,1.10109341601319)(3.8,1.10336980468073)(3.81,1.10559594479496)(3.82,1.10777251749192)(3.83,1.10990050962796)(3.84,1.11198061364416)(3.85,1.11401375113588)(3.86,1.11600063920389)(3.87,1.11794213769554)(3.88,1.11983897769857)(3.89,1.12169196817737)(3.9,1.12350194819968)(3.91,1.12526943431809)(3.92,1.12699550879749)(3.93,1.12868067144118)(3.94,1.130325448827)(3.95,1.13193138532679)(3.96,1.13349807283152)(3.97,1.13502664620159)(3.98,1.13651910439152)(3.99,1.13797396053282)(4,1.139393143328)(4.01,1.14077702115989)(4.02,1.14212611601321)(4.03,1.14344137133028)(4.04,1.14472313181656)(4.05,1.14597221057809)(4.06,1.14718934578313)(4.07,1.14837489128701)(4.08,1.14952909034658)(4.09,1.15065339770539)(4.1,1.1517480720919)(4.11,1.15281301943664)(4.12,1.15384972845331)(4.13,1.15485847981166)(4.14,1.15583942085807)(4.15,1.15679378648019)(4.16,1.15772150806407)(4.17,1.15862327090813)(4.18,1.15949993128138)(4.19,1.16035132608126)(4.2,1.16117867698112)(4.21,1.16198200136591)(4.22,1.16276177333311)(4.23,1.16351899849961)(4.24,1.1642532776356)(4.25,1.16496600653414)(4.26,1.16565692421908)(4.27,1.16632676746386)(4.28,1.16697613755048)(4.29,1.16760496202295)(4.3,1.1682143900997)(4.31,1.1688040539129)(4.32,1.16937511490571)(4.33,1.16992733642762)(4.34,1.17046160498641)(4.35,1.1709779963588)(4.36,1.17147702370201)(4.37,1.17195912848173)(4.38,1.17242446032287)(4.39,1.17287371546588)(4.4,1.1733069081311)(4.41,1.17372475379905)(4.42,1.17412727247276)(4.43,1.1745149308046)(4.44,1.17488804384998)(4.45,1.1752470055214)(4.46,1.17559208489739)(4.47,1.17592361714745)(4.48,1.17624192779172)(4.49,1.17654732581579)(4.5,1.17684011305157)(4.51,1.17712058661933)(4.52,1.17738903370722)(4.53,1.17764575676152)(4.54,1.17789098695468)(4.55,1.17812503892234)(4.56,1.17834795865608)(4.57,1.17856090517773)(4.58,1.17876282231269)(4.59,1.17895500228839)(4.6,1.17913726319028)(4.61,1.17930966452953)(4.62,1.17947315548145)(4.63,1.17962690766977)(4.64,1.17977235176391)(4.65,1.17990860776472)(4.66,1.18003676606699)(4.67,1.18015655868002)(4.68,1.18026828698987)(4.69,1.18037247850189)(4.7,1.180468870423)(4.71,1.18055787396573)(4.72,1.18063969553728)(4.73,1.18071453117112)(4.74,1.18078257239341)(4.75,1.18084401307962)(4.76,1.18089903938964)(4.77,1.18094787050682)(4.78,1.18099065745497)(4.79,1.18102757475474)(4.8,1.18105870132955)(4.81,1.1810843019774)(4.82,1.18110452916694)(4.83,1.18111949822203)(4.84,1.18112937417259)(4.85,1.18113423826043)(4.86,1.18113453588564)(4.87,1.18113012270667)(4.88,1.18112122832559)(4.89,1.18110801725507)(4.9,1.18109063739376)(4.91,1.18106920289578)(4.92,1.18104386014839)(4.93,1.18101473161657)(4.94,1.18098193020838)(4.95,1.18094558065861)(4.96,1.18090580461051)(4.97,1.18086272144218)(4.98,1.18081645612137)(4.99,1.18076718711211)(5,1.1807147914388)(5.01,1.1806598537458)(5.02,1.18060182955873)(5.03,1.18054117840191)(5.04,1.18047821353018)(5.05,1.18041238677673)(5.06,1.18034476262755)(5.07,1.18027429324624)(5.08,1.18020224480724)(5.09,1.18012763664821)(5.1,1.18005140183363)(5.11,1.17997307351707)(5.12,1.1798929464657)(5.13,1.17981124439353)(5.14,1.17972756927281)(5.15,1.17964276641565)(5.16,1.17955595325494)(5.17,1.17946822853412)(5.18,1.17937874257748)(5.19,1.1792880882856)(5.2,1.17919628483107)(5.21,1.17910333924806)(5.22,1.1790093119542)(5.23,1.17891426477774)(5.24,1.17881825902395)(5.25,1.178721355276)(5.26,1.1786236061874)(5.27,1.17852503254482)(5.28,1.1784257519221)(5.29,1.17832575199588)(5.3,1.17822514779222)(5.31,1.17812401587974)(5.32,1.17802239510663)(5.33,1.17792030078706)(5.34,1.17781772672562)(5.35,1.17771484885259)(5.36,1.17761162094154)(5.37,1.17750811495408)(5.38,1.1774043534304)(5.39,1.17730042105061)(5.4,1.17719635030153)(5.41,1.177092004394)(5.42,1.17698780507286)(5.43,1.17688351513363)(5.44,1.17677892482637)(5.45,1.17667450020686)(5.46,1.176570186193)(5.47,1.17646591099445)(5.48,1.17636169000057)(5.49,1.17625763719154)(5.5,1.17615389414353)(5.51,1.17604993279685)(5.52,1.17594650490464)(5.53,1.1758435884975)(5.54,1.17574046593217)(5.55,1.17563809629759)(5.56,1.17553559817778)(5.57,1.17543387398122)(5.58,1.17533219777928)(5.59,1.17523107487571)(5.6,1.17513039212184)(5.61,1.17502997649438)(5.62,1.17493007074161)(5.63,1.17483118963182)(5.64,1.17473178975637)(5.65,1.17463298700404)(5.66,1.17453519926118)(5.67,1.17443756651907)(5.68,1.174340844427)(5.69,1.17424444902058)(5.7,1.17414868268716)(5.71,1.17405361849797)(5.72,1.17395887909241)(5.73,1.17386509054166)(5.74,1.17377165569949)(5.75,1.17367893607118)(5.76,1.17358697663997)(5.77,1.17349534980347)(5.78,1.17340479976938)(5.79,1.1733145677052)(5.8,1.17322516456007)(5.81,1.17313652692324)(5.82,1.17304824174771)(5.83,1.17296113742236)(5.84,1.17287425140831)(5.85,1.17278844564599)(5.86,1.17270314985304)(5.87,1.17261859990366)(5.88,1.17253488751035)(5.89,1.17245168443629)(5.9,1.17236951092418)(5.91,1.17228773451682)(5.92,1.17220706491129)(5.93,1.1721267846328)(5.94,1.17204756255492)(5.95,1.1719688316804)(5.96,1.17189102703497)(5.97,1.17181387893615)(5.98,1.17173751323018)(5.99,1.17166199984918)(6,1.17158716345656) 
};

\end{axis}
\end{tikzpicture}

%% file: Flow1.tikz
%
%
%
%
\begin{tikzpicture}

\begin{axis}[%
width=0.7\textwidth,
height=0.7\textwidth,
scale only axis,
xmin=0, xmax=6,
xlabel={Time},
ymin=-8, ymax=10,
ylabel={$\lambda_{1}(t), \lambda_{1}^{w}(t)$},
axis lines=left,
axis on top]
\addplot [
color=blue,
solid,
line width=1.0pt
]
coordinates{
 (0,3)(0.01,2.74750141621176)(0.02,2.51012619572103)(0.03,2.28516653616728)(0.04,2.07031379370463)(0.05,1.8636035521294)(0.06,1.66337053518922)(0.07,1.46820977038115)(0.08,1.27694322965583)(0.09,1.08859637567945)(0.1,0.902377384884772)(0.11,0.717663029801252)(0.12,0.533988102725702)(0.13,0.351030101484627)(0.14,0.168589891989807)(0.15,-0.013445109587181)(0.16,-0.195144795899757)(0.17,-0.376582840214539)(0.18,-0.557861248329164)(0.19,-0.739110415351784)(0.2,-0.920460110454971)(0.21,-1.10200958827027)(0.22,-1.28380359431402)(0.23,-1.46582369353946)(0.24,-1.64799387775108)(0.25,-1.83019048813282)(0.26,-2.01225137533824)(0.27,-2.19399169181262)(0.28,-2.37520165760096)(0.29,-2.55565817003255)(0.3,-2.73512502167205)(0.31,-2.91335247179746)(0.32,-3.09007484721434)(0.33,-3.26500693915526)(0.34,-3.43782785769167)(0.35,-3.60815702044106)(0.36,-3.77548574321838)(0.37,-3.93901307390887)(0.38,-4.09710498465987)(0.39,-4.245002345448)(0.4,-4.35937996343569)(0.41,-4.3156001794019)(0.42,-4.10104735744774)(0.43,-3.85454959588621)(0.44,-3.60624917245282)(0.45,-3.3632987303345)(0.46,-3.12912800818536)(0.47,-2.90615542777446)(0.48,-2.69638954010036)(0.49,-2.50160858889071)(0.5,-2.32341997692074)(0.51,-2.16328215366425)(0.52,-2.02250947284048)(0.53,-1.90227389620922)(0.54,-1.80359910691142)(0.55,-1.7273597593388)(0.56,-1.67427478371946)(0.57,-1.64490563455369)(0.58,-1.63965284365541)(0.59,-1.65875316304767)(0.6,-1.70228016161846)(0.61,-1.77014134678764)(0.62,-1.86208172564371)(0.63,-1.97768259525308)(0.64,-2.11636644135408)(0.65,-2.27739878240333)(0.66,-2.45989314303737)(0.67,-2.6628163067523)(0.68,-2.88499280652607)(0.69,-3.1251124711583)(0.7,-3.38173401840259)(0.71,-3.65329173222408)(0.72,-3.93809490367506)(0.73,-4.23432556349883)(0.74,-4.54001207866064)(0.75,-4.8529711342209)(0.76,-5.17061280421191)(0.77,-5.4892960814692)(0.78,-5.8014063661119)(0.79,-6.07496696797394)(0.8,-6.17703934288773)(0.81,-6.1306940043508)(0.82,-6.04960766065335)(0.83,-5.95546484710339)(0.84,-5.85281632897027)(0.85,-5.74320477931604)(0.86,-5.62736110057982)(0.87,-5.50573063046363)(0.88,-5.3786410339419)(0.89,-5.24636672049772)(0.9,-5.10915718405395)(0.91,-4.96725064341447)(0.92,-4.82088033783167)(0.93,-4.67027851713665)(0.94,-4.51567712799998)(0.95,-4.35730802105812)(0.96,-4.19540202461172)(0.97,-4.03018627981238)(0.98,-3.86188169246435)(0.99,-3.69070036456431)(1,-3.51684363632505)(1.01,-3.34050185106645)(1.02,-3.161857255702)(1.03,-2.98109474501699)(1.04,-2.79841076265599)(1.05,-2.61402501574498)(1.06,-2.42818899192556)(1.07,-2.24118706135731)(1.08,-2.05332582128485)(1.09,-1.86492958791891)(1.1,-1.67632335637132)(1.11,-1.48782452089207)(1.12,-1.29974063456531)(1.13,-1.11236299799135)(1.14,-0.925970796014418)(1.15,-0.740831406511548)(1.16,-0.557201805357539)(1.17,-0.375332708889127)(1.18,-0.195466884359581)(1.19,-0.017842540107573)(1.2,0.15730665419491)(1.21,0.329752831103384)(1.22,0.499270562262598)(1.23,0.66564104002276)(1.24,0.828647996416011)(1.25,0.988081474454947)(1.26,1.143735916049)(1.27,1.29541069384218)(1.28,1.44291242680153)(1.29,1.58605162909814)(1.3,1.72464768372534)(1.31,1.85852568775287)(1.32,1.98751871023593)(1.33,2.11146960442088)(1.34,2.23022775208038)(1.35,2.34365499649511)(1.36,2.45162077479672)(1.37,2.55400585789983)(1.38,2.65070171533394)(1.39,2.74160800404992)(1.4,2.82663696847407)(1.41,2.90570712987434)(1.42,2.97874856984053)(1.43,3.0456977216182)(1.44,3.10649975266863)(1.45,3.16110666403933)(1.46,3.20947572999497)(1.47,3.25157039640058)(1.48,3.28735470380612)(1.49,3.31679455307803)(1.5,3.3398476081775)(1.51,3.35645844846609)(1.52,3.3665382403754)(1.53,3.36993358727159)(1.54,3.36635548908278)(1.55,3.35521431631109)(1.56,3.33517857924925)(1.57,3.30276492931153)(1.58,3.24711255236207)(1.59,3.13564069140787)(1.6,2.94094724192847)(1.61,2.70364371685301)(1.62,2.45782921158552)(1.63,2.21556789894827)(1.64,1.98201841770479)(1.65,1.76026450702079)(1.66,1.55259812596123)(1.67,1.36092113128658)(1.68,1.18688917273432)(1.69,1.03196708901378)(1.7,0.897452894215539)(1.71,0.7844821828558)(1.72,0.694030855690971)(1.73,0.626909186972228)(1.74,0.583755605119811)(1.75,0.565026595330179)(1.76,0.570976680897844)(1.77,0.601630220137346)(1.78,0.65670583298311)(1.79,0.735455893111485)(1.8,0.836181831250721)(1.81,0.954505997992429)(1.82,1.07521359506955)(1.83,1.13905455993486)(1.84,1.09160938903233)(1.85,0.995466905579123)(1.86,0.885618869534177)(1.87,0.770733479475969)(1.88,0.653610251035216)(1.89,0.535493439660466)(1.9,0.417094941071607)(1.91,0.298909610200672)(1.92,0.181333633573681)(1.93,0.064717592634034)(1.94,-0.0506033517118365)(1.95,-0.164285858519031)(1.96,-0.275946111339166)(1.97,-0.385102367367175)(1.98,-0.491030415947721)(1.99,-0.592337593211534)(2,-0.68527173145907)(2.01,-0.753536556875087)(2.02,-0.700703090220261)(2.03,-0.462306256612032)(2.04,-0.185869648461633)(2.05,0.0857728230280259)(2.06,0.343978120302311)(2.07,0.585269213032241)(2.08,0.80748998945432)(2.09,1.00900582741108)(2.1,1.18847580099412)(2.11,1.34477432571828)(2.12,1.47695952394362)(2.13,1.58426405443895)(2.14,1.66608850236625)(2.15,1.72200170464976)(2.16,1.75174028261054)(2.17,1.7552077007536)(2.18,1.73247587894216)(2.19,1.68378139463445)(2.2,1.60952702391665)(2.21,1.51027606338019)(2.22,1.38675208986217)(2.23,1.2398331640161)(2.24,1.07054916535687)(2.25,0.880077748501797)(2.26,0.669741408724999)(2.27,0.441008954382778)(2.28,0.195502271707642)(2.29,-0.0649769324469249)(2.3,-0.33837674312277)(2.31,-0.622190587772971)(2.32,-0.912767205891988)(2.33,-1.20227506615058)(2.34,-1.45915342206431)(2.35,-1.56737002559355)(2.36,-1.5444525365903)(2.37,-1.48821190216315)(2.38,-1.41985118295575)(2.39,-1.34406262005119)(2.4,-1.26244523226782)(2.41,-1.1757454438652)(2.42,-1.08440662595628)(2.43,-0.988746427335448)(2.44,-0.889025671919156)(2.45,-0.785478829075513)(2.46,-0.678328614732243)(2.47,-0.567794087795447)(2.48,-0.454094690321716)(2.49,-0.337452604668493)(2.5,-0.218094040271632)(2.51,-0.0962477634891645)(2.52,0.0278547579251105)(2.53,0.153983802830592)(2.54,0.281916557458989)(2.55,0.411445030093466)(2.56,0.542386244555505)(2.57,0.674592137243823)(2.58,0.807954264637769)(2.59,0.942398958668141)(2.6,1.077862055062)(2.61,1.21425308118108)(2.62,1.35141312086325)(2.63,1.4890922277348)(2.64,1.62695452853216)(2.65,1.76460776652183)(2.66,1.90164510366494)(2.67,2.0376769700268)(2.68,2.1723494102668)(2.69,2.30534789748113)(2.7,2.43639284217588)(2.71,2.56522882258351)(2.72,2.69161959879892)(2.73,2.81534377519028)(2.74,2.93618356339759)(2.75,3.0539302742199)(2.76,3.16837667922839)(2.77,3.27932091966041)(2.78,3.38656302119504)(2.79,3.48990773466053)(2.8,3.58916238377352)(2.81,3.68413967238745)(2.82,3.77465533784399)(2.83,3.8605317847227)(2.84,3.94159471156906)(2.85,4.01767768368126)(2.86,4.08861859289755)(2.87,4.15426331785962)(2.88,4.21446473288438)(2.89,4.26908265097132)(2.9,4.31798735755594)(2.91,4.36105540893584)(2.92,4.39817561864752)(2.93,4.42924469717241)(2.94,4.45417031787747)(2.95,4.47287108166863)(2.96,4.48527377812981)(2.97,4.49131791772487)(2.98,4.49094947296798)(2.99,4.48412594240407)(3,4.47082749360224)(3.01,4.48647775781574)(3.02,4.51403242782103)(3.03,4.55226018484971)(3.04,4.60006357720069)(3.05,4.65646161882502)(3.06,4.72057387704493)(3.07,4.79160840811344)(3.08,4.86884935747955)(3.09,4.95164766948781)(3.1,5.03941188810938)(3.11,5.131599779519)(3.12,5.2277107064168)(3.13,5.32727853882197)(3.14,5.42986480485986)(3.15,5.53505335284281)(3.16,5.64244415011584)(3.17,5.75164898797479)(3.18,5.86228467187078)(3.19,5.97396722087824)(3.2,6.08629927334675)(3.21,6.198853182026)(3.22,6.3111294600273)(3.23,6.42247327539174)(3.24,6.53182738566811)(3.25,6.63694420939727)(3.26,6.7307365931247)(3.27,6.77646611112438)(3.28,6.65109494661385)(3.29,6.4070498407302)(3.3,6.13868526298956)(3.31,5.86550204536449)(3.32,5.59342303981932)(3.33,5.32572618799032)(3.34,5.06489769566162)(3.35,4.81309085797759)(3.36,4.57226861154729)(3.37,4.34425439837729)(3.38,4.13074612962471)(3.39,3.93332199484344)(3.4,3.75343479176408)(3.41,3.59240856029843)(3.42,3.45143310142106)(3.43,3.33155554275763)(3.44,3.23367799090666)(3.45,3.15854929898669)(3.46,3.10676295540086)(3.47,3.07875283999963)(3.48,3.07479006909854)(3.49,3.0949832983443)(3.5,3.13927549055254)(3.51,3.20744738516558)(3.52,3.29911603821508)(3.53,3.41373962186717)(3.54,3.55061925621908)(3.55,3.70890386892587)(3.56,3.88759574324414)(3.57,4.08555510846035)(3.58,4.30150849622964)(3.59,4.53405333671486)(3.6,4.78166689989097)(3.61,5.04271040431119)(3.62,5.31543443987835)(3.63,5.59797611183477)(3.64,5.88834789533166)(3.65,6.18439220836218)(3.66,6.48365929301752)(3.67,6.78300520040819)(3.68,7.07699494305081)(3.69,7.3484965890867)(3.7,7.50545162805559)(3.71,7.46640862546761)(3.72,7.36668298995089)(3.73,7.24991108246505)(3.74,7.12370248692866)(3.75,6.99024631259797)(3.76,6.85045814987529)(3.77,6.70484183441356)(3.78,6.55373958087577)(3.79,6.39742249142294)(3.8,6.23612831603951)(3.81,6.07007890678569)(3.82,5.89948934193945)(3.83,5.7245714154271)(3.84,5.54553654153418)(3.85,5.36259541687799)(3.86,5.17595916390429)(3.87,4.9858375132712)(3.88,4.79243983654087)(3.89,4.59597407164476)(3.9,4.39664802930213)(3.91,4.19467207924859)(3.92,3.99026001959502)(3.93,3.78363534479356)(3.94,3.57503280115104)(3.95,3.36469940454988)(3.96,3.15289873368861)(3.97,2.93990441077566)(3.98,2.7259970234794)(3.99,2.51146702199572)(4,2.29660086842314)(4.01,2.08168509970658)(4.02,1.86700225175193)(4.03,1.65282937308719)(4.04,1.43943930024318)(4.05,1.22709935305687)(4.06,1.01607227939573)(4.07,0.806617317065537)(4.08,0.598990027624863)(4.09,0.393440276169912)(4.1,0.190216341061987)(4.11,-0.0104364429431765)(4.12,-0.208279001529619)(4.13,-0.403073931194923)(4.14,-0.594587520536313)(4.15,-0.78259215813763)(4.16,-0.966862339813973)(4.17,-1.14717887488488)(4.18,-1.32332821328957)(4.19,-1.49510092365751)(4.2,-1.66229672822336)(4.21,-1.82472045982812)(4.22,-1.98218587186771)(4.23,-2.13451630607084)(4.24,-2.28154187031009)(4.25,-2.42310535585687)(4.26,-2.55905635643307)(4.27,-2.68925576618058)(4.28,-2.81357358614792)(4.29,-2.93188752927662)(4.3,-3.04408589146649)(4.31,-3.15006215642782)(4.32,-3.24971980803403)(4.33,-3.34296641412927)(4.34,-3.42971717683048)(4.35,-3.50989031229715)(4.36,-3.58340696740359)(4.37,-3.65018752273556)(4.38,-3.71014485187481)(4.39,-3.76317679885093)(4.4,-3.80914359422502)(4.41,-3.84783455759353)(4.42,-3.87887750856266)(4.43,-3.90153372680218)(4.44,-3.91408092821127)(4.45,-3.91161191415412)(4.46,-3.8764531678362)(4.47,-3.75628156932368)(4.48,-3.54077483606293)(4.49,-3.29162742505302)(4.5,-3.03652989956994)(4.51,-2.78429694706822)(4.52,-2.53910472064071)(4.53,-2.30373438555745)(4.54,-2.0804109571157)(4.55,-1.8710680877994)(4.56,-1.67744319713339)(4.57,-1.50111573523674)(4.58,-1.34351775781352)(4.59,-1.20593949876883)(4.6,-1.08952494737365)(4.61,-0.995269530153054)(4.62,-0.924015017707954)(4.63,-0.876442838785076)(4.64,-0.853071061463959)(4.65,-0.854244301913082)(4.66,-0.880129252404679)(4.67,-0.930697328324067)(4.68,-1.00570702892115)(4.69,-1.10464673295577)(4.7,-1.22662579082628)(4.71,-1.3700339792591)(4.72,-1.53137369700772)(4.73,-1.69961839257183)(4.74,-1.82248027978337)(4.75,-1.80047150166902)(4.76,-1.70251089784843)(4.77,-1.58611485850252)(4.78,-1.46329252249356)(4.79,-1.33737854030383)(4.8,-1.20973782204584)(4.81,-1.08111996493592)(4.82,-0.95203757703726)(4.83,-0.822900985640103)(4.84,-0.694077185269401)(4.85,-0.565922627873066)(4.86,-0.438814753705149)(4.87,-0.313192048540963)(4.88,-0.189646766483971)(4.89,-0.0691615275987389)(4.9,0.0460890985440203)(4.91,0.148964480895663)(4.92,0.19993150760208)(4.93,0.0674575016363841)(4.94,-0.170684749527894)(4.95,-0.414511077637747)(4.96,-0.645895215608207)(4.97,-0.859764762859327)(4.98,-1.05360968344946)(4.99,-1.22574674655801)(5,-1.37488882318118)(5.01,-1.50000545196423)(5.02,-1.60027241278986)(5.03,-1.67505079526649)(5.04,-1.72387852335768)(5.05,-1.74646517167419)(5.06,-1.7426922396021)(5.07,-1.71260871243619)(5.08,-1.65643293275591)(5.09,-1.57454715887366)(5.1,-1.46749794454797)(5.11,-1.33599072056207)(5.12,-1.18088720275546)(5.13,-1.00320114436228)(5.14,-0.804093414104819)(5.15,-0.584870266675947)(5.16,-0.346980187545262)(5.17,-0.092020715840206)(5.18,0.178248797596804)(5.19,0.461853612101939)(5.2,0.756478728094105)(5.21,1.0590593918248)(5.22,1.36424984897421)(5.23,1.65637549146294)(5.24,1.86465098579794)(5.25,1.91374408238822)(5.26,1.90042656123018)(5.27,1.86889681128843)(5.28,1.82794959676392)(5.29,1.78011325465599)(5.3,1.72640691853721)(5.31,1.6673596772728)(5.32,1.60330820366944)(5.33,1.5345044039)(5.34,1.46116057591011)(5.35,1.38347026633212)(5.36,1.30161924903488)(5.37,1.21579078702489)(5.38,1.12616840704564)(5.39,1.03293720682991)(5.4,0.936284008159928)(5.41,0.836396370181295)(5.42,0.733461174130807)(5.43,0.627662977509693)(5.44,0.51918146913119)(5.45,0.408190083312373)(5.46,0.294854926441281)(5.47,0.179338287523303)(5.48,0.0618028011209908)(5.49,-0.0575817247368553)(5.5,-0.178627610237158)(5.51,-0.301122796684151)(5.52,-0.424829429527348)(5.53,-0.549484898248029)(5.54,-0.674807547463323)(5.55,-0.800509995412993)(5.56,-0.926302021050049)(5.57,-1.05190203378527)(5.58,-1.1770351620872)(5.59,-1.30143837554299)(5.6,-1.4248571663311)(5.61,-1.54704503498072)(5.62,-1.66776272805547)(5.63,-1.78677695532677)(5.64,-1.90385575625023)(5.65,-2.01877496284582)(5.66,-2.13131291769389)(5.67,-2.24124929733799)(5.68,-2.34836959270256)(5.69,-2.45246052545232)(5.7,-2.55331386151353)(5.71,-2.65072472755154)(5.72,-2.74449153448809)(5.73,-2.83441889465851)(5.74,-2.92031392962418)(5.75,-3.00199085709775)(5.76,-3.07926906819116)(5.77,-3.1519735874823)(5.78,-3.21993875972985)(5.79,-3.28300405175419)(5.8,-3.34101985085718)(5.81,-3.3938446320999)(5.82,-3.4413458333224)(5.83,-3.48340283895793)(5.84,-3.51990163626773)(5.85,-3.55074091592019)(5.86,-3.57582602528407)(5.87,-3.59507236089649)(5.88,-3.60840321934814)(5.89,-3.61574863672324)(5.9,-3.61704705247941)(5.91,-3.61224081968603)(5.92,-3.60127957371928)(5.93,-3.58411360696732)(5.94,-3.56069591936668)(5.95,-3.53097352351344)(5.96,-3.49488240072319)(5.97,-3.45233142335279)(5.98,-3.40317314465166)(5.99,-3.34714489144456)(6,-3.28372163854834) 
};

\addplot [
color=blue,
dashed,
line width=1.0pt
]
coordinates{
 (0,3.54577985588562)(0.01,3.38792258908536)(0.02,3.15286323686312)(0.03,2.84820757924092)(0.04,2.5209039452134)(0.05,2.19183361945785)(0.06,1.86864984422598)(0.07,1.55527653584464)(0.08,1.254392978194)(0.09,0.968142607182479)(0.1,0.698370546765313)(0.11,0.446715288709745)(0.12,0.214642178951874)(0.13,0.00345747635117455)(0.14,-0.185688879769772)(0.15,-0.351805244790692)(0.16,-0.494063496380532)(0.17,-0.611805186807303)(0.18,-0.704553498760123)(0.19,-0.772027755441955)(0.2,-0.814174463711052)(0.21,-0.831227582810765)(0.22,-0.823854461253894)(0.23,-0.793564863602958)(0.24,-0.744145169270903)(0.25,-0.688282380665097)(0.26,-0.678484576880517)(0.27,-0.784214469262011)(0.28,-0.947325610894532)(0.29,-1.1254154038683)(0.3,-1.30814731799955)(0.31,-1.49234044626318)(0.32,-1.67662232727358)(0.33,-1.8602190013807)(0.34,-2.04259560201334)(0.35,-2.22332485473022)(0.36,-2.40202933888408)(0.37,-2.57834961533627)(0.38,-2.75192040438615)(0.39,-2.92233844539877)(0.4,-3.08910720990869)(0.41,-3.25150035710605)(0.42,-3.40816128686645)(0.43,-3.55554877561161)(0.44,-3.6788494114186)(0.45,-3.68924201263785)(0.46,-3.50864708732353)(0.47,-3.28348242025928)(0.48,-3.06111879597217)(0.49,-2.85080749911092)(0.5,-2.65612871652978)(0.51,-2.47925035390224)(0.52,-2.32179440673885)(0.53,-2.18508101544743)(0.54,-2.07021176616545)(0.55,-1.97810225592234)(0.56,-1.90949425251387)(0.57,-1.86496025267932)(0.58,-1.84490517798847)(0.59,-1.84956648475515)(0.6,-1.87901447888581)(0.61,-1.93315270600548)(0.62,-2.01171910944274)(0.63,-2.11428775257016)(0.64,-2.24027127064643)(0.65,-2.3889241096321)(0.66,-2.55934591931404)(0.67,-2.75048567321891)(0.68,-2.96114477458204)(0.69,-3.18997908797059)(0.7,-3.43549783106721)(0.71,-3.69605088146777)(0.72,-3.96979989546279)(0.73,-4.25462592522702)(0.74,-4.5478698586354)(0.75,-4.84540851418806)(0.76,-5.13727080503268)(0.77,-5.37835967098212)(0.78,-5.45017305730805)(0.79,-5.41791756495587)(0.8,-5.35872114372427)(0.81,-5.28733867902195)(0.82,-5.20738831878822)(0.83,-5.12020188763603)(0.84,-5.02643830176115)(0.85,-4.92650794639087)(0.86,-4.82071606936261)(0.87,-4.70932020084451)(0.88,-4.592556098233)(0.89,-4.47065080655644)(0.9,-4.34382944945069)(0.91,-4.21231907249486)(0.92,-4.07635074265968)(0.93,-3.93616036397685)(0.94,-3.79198805114606)(0.95,-3.64407674646664)(0.96,-3.49266835449713)(0.97,-3.33799801718616)(0.98,-3.18028419200092)(0.99,-3.01971829439527)(1,-2.85645251033641)(1.01,-2.69059419139877)(1.02,-2.5222129608016)(1.03,-2.35136926340601)(1.04,-2.17815653619266)(1.05,-2.00274862113067)(1.06,-1.82542297563603)(1.07,-1.64655171325991)(1.08,-1.46656432081573)(1.09,-1.28590214623001)(1.1,-1.10498181330242)(1.11,-0.924179536235405)(1.12,-0.743827968614344)(1.13,-0.564223854715164)(1.14,-0.385638218281517)(1.15,-0.208323003518796)(1.16,-0.0325203070101345)(1.17,0.141535888756091)(1.18,0.313615366273964)(1.19,0.483491612691423)(1.2,0.650940471244272)(1.21,0.815740739278268)(1.22,0.97767389028721)(1.23,1.13652477572863)(1.24,1.29208178251552)(1.25,1.44413741489765)(1.26,1.59248864129668)(1.27,1.7369373724678)(1.28,1.87729090153492)(1.29,2.0133624131448)(1.3,2.14497144045559)(1.31,2.27194453112077)(1.32,2.39411570554354)(1.33,2.51132718695379)(1.34,2.62342991085772)(1.35,2.73028404715457)(1.36,2.83175944747122)(1.37,2.92773585423998)(1.38,3.01810300225573)(1.39,3.10276055414532)(1.4,3.18161782616976)(1.41,3.25459342904981)(1.42,3.3216148449834)(1.43,3.38261779228472)(1.44,3.43754567038949)(1.45,3.48634867804261)(1.46,3.52898276188731)(1.47,3.56540813134777)(1.48,3.59558676468536)(1.49,3.61947899416643)(1.5,3.63703654009667)(1.51,3.64819267889594)(1.52,3.65283949425218)(1.53,3.6507898086719)(1.54,3.64168334376394)(1.55,3.62477125535945)(1.56,3.59829920481114)(1.57,3.55744832078807)(1.58,3.4866602420509)(1.59,3.34705205384395)(1.6,3.13130236455362)(1.61,2.88610141107252)(1.62,2.63677237131709)(1.63,2.39233426895867)(1.64,2.15709829875036)(1.65,1.93386619672845)(1.66,1.72481840456073)(1.67,1.53180724827049)(1.68,1.35646266099045)(1.69,1.20023759063245)(1.7,1.0644219082245)(1.71,0.950147992126776)(1.72,0.858389253652041)(1.73,0.789955383555257)(1.74,0.745484780555239)(1.75,0.725434937880578)(1.76,0.730062321291929)(1.77,0.759394238185041)(1.78,0.813157072611233)(1.79,0.890617756318591)(1.8,0.990115004697778)(1.81,1.10739003742911)(1.82,1.22769834544236)(1.83,1.29250896080556)(1.84,1.24398927510231)(1.85,1.14506692064794)(1.86,1.03215695693614)(1.87,0.914188188273579)(1.88,0.794011636498025)(1.89,0.672884693565339)(1.9,0.551523364799501)(1.91,0.430423817789494)(1.92,0.309982615924572)(1.93,0.190550664200089)(1.94,0.0724633636584155)(1.95,-0.0439346745210205)(1.96,-0.158257373302832)(1.97,-0.27001658997151)(1.98,-0.378471821118626)(1.99,-0.482179483939699)(2,-0.577173757862149)(2.01,-0.645719207638657)(2.02,-0.584720042776081)(2.03,-0.342524628122489)(2.04,-0.0665269718616908)(2.05,0.204022379553759)(2.06,0.460977055998211)(2.07,0.700965001735135)(2.08,0.921862491696793)(2.09,1.12204747143657)(2.1,1.30018517080276)(2.11,1.45515256751835)(2.12,1.58601041395669)(2.13,1.69199170517201)(2.14,1.77249827619755)(2.15,1.82709940614074)(2.16,1.85553187830273)(2.17,1.85769999753334)(2.18,1.8336751469336)(2.19,1.78369493305279)(2.2,1.70816130509114)(2.21,1.60763837575476)(2.22,1.48284894028069)(2.23,1.33467122722599)(2.24,1.16413440392464)(2.25,0.972415174536836)(2.26,0.760834785835536)(2.27,0.530858806025474)(2.28,0.284104747020988)(2.29,0.0223637543059745)(2.3,-0.252333873136982)(2.31,-0.537537007663251)(2.32,-0.829767165936612)(2.33,-1.12193970864832)(2.34,-1.3872424606101)(2.35,-1.51563226435292)(2.36,-1.50014118575147)(2.37,-1.44643295155201)(2.38,-1.37970690960524)(2.39,-1.3052885477336)(2.4,-1.22492750707001)(2.41,-1.13941915716271)(2.42,-1.04922650039278)(2.43,-0.95467630772915)(2.44,-0.85603404879618)(2.45,-0.753536588495623)(2.46,-0.647408155228939)(2.47,-0.537868334013428)(2.48,-0.425136760725403)(2.49,-0.309435246285599)(2.5,-0.190989030016833)(2.51,-0.0700260026110566)(2.52,0.0532240790498955)(2.53,0.178533546071615)(2.54,0.305680992033896)(2.55,0.434459563476779)(2.56,0.564685509311629)(2.57,0.696206871395969)(2.58,0.828906769165213)(2.59,0.962698353112166)(2.6,1.09750179258492)(2.61,1.233211537465)(2.62,1.36966005826237)(2.63,1.50659787675885)(2.64,1.64369780811095)(2.65,1.78058165898429)(2.66,1.91685679986706)(2.67,2.05214446417522)(2.68,2.18609805527154)(2.69,2.31840588589624)(2.7,2.44878806971608)(2.71,2.57698936083848)(2.72,2.70277181486448)(2.73,2.82590997201715)(2.74,2.94618593934863)(2.75,3.06338802622652)(2.76,3.17730846209193)(2.77,3.2877434618497)(2.78,3.39449259512088)(2.79,3.49735919864662)(2.8,3.59615041380301)(2.81,3.69067765665986)(2.82,3.78075689018239)(2.83,3.8662090496889)(2.84,3.94686044187235)(2.85,4.0225431375756)(2.86,4.09309546553848)(2.87,4.15836250213445)(2.88,4.21819656989028)(2.89,4.27245792695155)(2.9,4.32101522337221)(2.91,4.36374621811221)(2.92,4.4005382067396)(2.93,4.43128845192171)(2.94,4.45590444629498)(2.95,4.47430397727715)(2.96,4.48641501055248)(2.97,4.49217548295428)(2.98,4.49153286813648)(2.99,4.48444371918743)(3,4.47087715583188)(3.01,1.22414194389247)(3.02,1.43486396430166)(3.03,1.62412722944367)(3.04,1.79052603710015)(3.05,1.93302733811018)(3.06,2.05088473162728)(3.07,2.14361838879881)(3.08,2.21104470815291)(3.09,2.25336558212339)(3.1,2.27142678601273)(3.11,2.26746019409272)(3.12,2.24766706892016)(3.13,2.23196830565435)(3.14,2.26771293803905)(3.15,2.36850197372787)(3.16,2.49880072211528)(3.17,2.64007271876572)(3.18,2.78622264443964)(3.19,2.93486802566528)(3.2,3.08482801548475)(3.21,3.23538264531947)(3.22,3.38601406347718)(3.23,3.53630162665767)(3.24,3.68587309601165)(3.25,3.83437778952507)(3.26,3.98146868388648)(3.27,4.12678192848584)(3.28,4.26990995825289)(3.29,4.41034126586844)(3.3,4.54731876190284)(3.31,4.6794076072699)(3.32,4.80279450567995)(3.33,4.90110422649038)(3.34,4.87928240097706)(3.35,4.67412148414906)(3.36,4.42951761636311)(3.37,4.18799288349003)(3.38,3.95822098841786)(3.39,3.74369291536293)(3.4,3.54657146448853)(3.41,3.36849483303429)(3.42,3.21080785367258)(3.43,3.07463962841686)(3.44,2.96093571281984)(3.45,2.87046869222824)(3.46,2.80384285913856)(3.47,2.76149548769884)(3.48,2.74369675373502)(3.49,2.75054996525309)(3.5,2.78199175373587)(3.51,2.83779325311169)(3.52,2.91756163363423)(3.53,3.02074275328884)(3.54,3.14662431339661)(3.55,3.29433971351971)(3.56,3.4628726127903)(3.57,3.65106117286616)(3.58,3.85760252687321)(3.59,4.08105523039696)(3.6,4.31983735562061)(3.61,4.57221879941022)(3.62,4.83628421242469)(3.63,5.10985062637159)(3.64,5.39019960941624)(3.65,5.67316073508726)(3.66,5.94879423927976)(3.67,6.17332170020949)(3.68,6.23308042430369)(3.69,6.18938480313065)(3.7,6.11834775053185)(3.71,6.03496219829257)(3.72,5.94290659008856)(3.73,5.84352035709301)(3.74,5.73745337969688)(3.75,5.6251009949544)(3.76,5.50675063101987)(3.77,5.38264041463235)(3.78,5.2529856284324)(3.79,5.11799152386875)(3.8,4.97786015043629)(3.81,4.83279354939893)(3.82,4.68299544602895)(3.83,4.52867135593069)(3.84,4.37002830942267)(3.85,4.20727265401022)(3.86,4.040609340686)(3.87,3.87023886988423)(3.88,3.69635528514097)(3.89,3.51914649671911)(3.9,3.33879654741013)(3.91,3.15548845676757)(3.92,2.96941361949063)(3.93,2.78078085385247)(3.94,2.58982205170113)(3.95,2.39679742581751)(3.96,2.20199165874243)(3.97,2.00570843002954)(3.98,1.80826035450501)(3.99,1.60996138136477)(4,1.41111961616245)(4.01,1.21203492164851)(4.02,1.01299697737302)(4.03,0.814286506282547)(4.04,0.61617618031766)(4.05,0.418932098015985)(4.06,0.222815169789459)(4.07,0.0280820010889626)(4.08,-0.165014598892463)(4.09,-0.356224902202051)(4.1,-0.54530247190192)(4.11,-0.732003872119359)(4.12,-0.916088916702924)(4.13,-1.09732079761202)(4.14,-1.27546634015746)(4.15,-1.45029626706226)(4.16,-1.6215855915939)(4.17,-1.7891140107827)(4.18,-1.95266635258617)(4.19,-2.11203327069815)(4.2,-2.26701171740136)(4.21,-2.41740579562703)(4.22,-2.56302740980543)(4.23,-2.70369688489938)(4.24,-2.83924362757998)(4.25,-2.96950643540413)(4.26,-3.09433371542384)(4.27,-3.2135834897041)(4.28,-3.32712316031352)(4.29,-3.43482913267438)(4.3,-3.53658643552877)(4.31,-3.63228808424218)(4.32,-3.72183457281357)(4.33,-3.80513315984861)(4.34,-3.88209691978345)(4.35,-3.95264366023361)(4.36,-4.01669375034593)(4.37,-4.07416752705064)(4.38,-4.12497938690097)(4.39,-4.16902992261418)(4.4,-4.20618697156135)(4.41,-4.23625510841572)(4.42,-4.2588981882062)(4.43,-4.27346632997325)(4.44,-4.27849629040788)(4.45,-4.27001618647058)(4.46,-4.23444777696407)(4.47,-4.12741760252761)(4.48,-3.92100553648052)(4.49,-3.67199314679424)(4.5,-3.41436037823122)(4.51,-3.15877553373741)(4.52,-2.90991890931488)(4.53,-2.67074327444128)(4.54,-2.44354427265639)(4.55,-2.23028861134434)(4.56,-2.03273173261175)(4.57,-1.85246248571066)(4.58,-1.6909200177287)(4.59,-1.54939773186422)(4.6,-1.42904235541162)(4.61,-1.33085076132597)(4.62,-1.25566461396033)(4.63,-1.20416543190635)(4.64,-1.1768683353458)(4.65,-1.17411451822355)(4.66,-1.19606243980795)(4.67,-1.24266957411752)(4.68,-1.3136657693577)(4.69,-1.40848330173626)(4.7,-1.52609537088395)(4.71,-1.66451856421385)(4.72,-1.81891489037818)(4.73,-1.97145281200192)(4.74,-2.04378428506945)(4.75,-1.98034729252245)(4.76,-1.86972881973307)(4.77,-1.74659913951007)(4.78,-1.61850351557286)(4.79,-1.48785942592207)(4.8,-1.35576646836452)(4.81,-1.22287670319368)(4.82,-1.08966134791999)(4.83,-0.956512050015923)(4.84,-0.823788808916609)(4.85,-0.691848977167975)(4.86,-0.561081501949277)(4.87,-0.431955164855567)(4.88,-0.30513764438402)(4.89,-0.181824907191641)(4.9,-0.0649571636760951)(4.91,0.0342138446663713)(4.92,0.04702034462464)(4.93,-0.137871055190933)(4.94,-0.383793422547917)(4.95,-0.627102516230861)(4.96,-0.856349371339433)(4.97,-1.06760095781568)(4.98,-1.25865059911743)(4.99,-1.42792211118181)(5,-1.57417384972012)(5.01,-1.69639718415253)(5.02,-1.79377990572757)(5.03,-1.8656895073015)(5.04,-1.91166765958902)(5.05,-1.93142689115066)(5.06,-1.92484942403286)(5.07,-1.89198604798799)(5.08,-1.83305461953217)(5.09,-1.74843833913411)(5.1,-1.63868286967114)(5.11,-1.50449346634118)(5.12,-1.34673062227298)(5.13,-1.16640631342488)(5.14,-0.964679041889706)(5.15,-0.742850261641818)(5.16,-0.502362043448203)(5.17,-0.244798699635644)(5.18,0.0281020145641069)(5.19,0.314413715759708)(5.2,0.611927974843963)(5.21,0.917873492165704)(5.22,1.22793130386711)(5.23,1.53159014481464)(5.24,1.78295258327772)(5.25,1.87502294643284)(5.26,1.87232179895049)(5.27,1.84463240186746)(5.28,1.80599426921792)(5.29,1.75994980552559)(5.3,1.70780200672884)(5.31,1.65018225631816)(5.32,1.58747111327736)(5.33,1.51994174059711)(5.34,1.44781765872773)(5.35,1.3712989503523)(5.36,1.29057504962336)(5.37,1.2058313610343)(5.38,1.1172527750184)(5.39,1.02502496733324)(5.4,0.929335133182908)(5.41,0.830371048015161)(5.42,0.728320098549623)(5.43,0.623368170750457)(5.44,0.515696819386361)(5.45,0.40548293835226)(5.46,0.292899099137401)(5.47,0.178114280788446)(5.48,0.0612986825693798)(5.49,-0.0573686843009398)(5.5,-0.177694645163768)(5.51,-0.29946461562946)(5.52,-0.422441198291867)(5.53,-0.546365882120483)(5.54,-0.670964801090942)(5.55,-0.795955895843416)(5.56,-0.921056602729228)(5.57,-1.04598842225986)(5.58,-1.17048046484574)(5.59,-1.29427031696861)(5.6,-1.41710380832149)(5.61,-1.53873414641085)(5.62,-1.65892082578982)(5.63,-1.77742842780128)(5.64,-1.89402613087283)(5.65,-2.00848693686128)(5.66,-2.12058732200975)(5.67,-2.23010730028397)(5.68,-2.33683016331243)(5.69,-2.44054271768317)(5.7,-2.54103537420275)(5.71,-2.63810238517566)(5.72,-2.73154205632666)(5.73,-2.82115713580918)(5.74,-2.90675516011333)(5.75,-2.98814889636904)(5.76,-3.0651570412286)(5.77,-3.13760468647684)(5.78,-3.20532421354753)(5.79,-3.26815597563743)(5.8,-3.32594894627503)(5.81,-3.37856143164794)(5.82,-3.42586138749076)(5.83,-3.46772664586823)(5.84,-3.50404489735755)(5.85,-3.53471342596348)(5.86,-3.55963867558359)(5.87,-3.57873584965116)(5.88,-3.59192823097915)(5.89,-3.59914669249711)(5.9,-3.6003290174407)(5.91,-3.59541904863887)(5.92,-3.58436578469324)(5.93,-3.56712152384752)(5.94,-3.54363976010866)(5.95,-3.51387035040929)(5.96,-3.47775322313234)(5.97,-3.43520356224906)(5.98,-3.38608858649873)(5.99,-3.33017035774101)(6,-3.26698513684715) 
};

\end{axis}
\end{tikzpicture}

%% file: Flow2.tikz
%
%
%
%
\begin{tikzpicture}

\begin{axis}[%
width=0.7\textwidth,
height=0.7\textwidth,
scale only axis,
xmin=0, xmax=6,
xlabel={Time},
ymin=-8, ymax=6,
ylabel={$\lambda_{2}(t), \lambda_{2}^{w}(t)$},
axis lines=left,
axis on top]
\addplot [
color=green,
solid,
line width=1.0pt
]
coordinates{
 (0,0.5)(0.01,0.382674932753223)(0.02,0.282355713473206)(0.03,0.196673608043596)(0.04,0.123586216535838)(0.05,0.061331586052204)(0.06,0.00839082868956458)(0.07,-0.0365439275963026)(0.08,-0.0745969961042333)(0.09,-0.1067274633276)(0.1,-0.133745571301821)(0.11,-0.156322874169209)(0.12,-0.174999375245763)(0.13,-0.190196041905769)(0.14,-0.202231141717589)(0.15,-0.211355228964497)(0.16,-0.217792333179311)(0.17,-0.221785244015614)(0.18,-0.223617911991516)(0.19,-0.223614281460914)(0.2,-0.22210799194895)(0.21,-0.219411111978963)(0.22,-0.21578895180358)(0.23,-0.211450859717382)(0.24,-0.206554435817919)(0.25,-0.20121459793196)(0.26,-0.195512912235568)(0.27,-0.189510225139303)(0.28,-0.183247016741801)(0.29,-0.176751939344137)(0.3,-0.170041975564316)(0.31,-0.163121822511021)(0.32,-0.155980702691643)(0.33,-0.14858775890171)(0.34,-0.140876155813062)(0.35,-0.13271784034403)(0.36,-0.123855319780439)(0.37,-0.113736424815368)(0.38,-0.100973315833053)(0.39,-0.0810499348408875)(0.4,-0.0308806325219648)(0.41,0.173936700345659)(0.42,0.545786530977147)(0.43,0.945614366887825)(0.44,1.34305670953299)(0.45,1.7307452468479)(0.46,2.10504049356333)(0.47,2.4633205620935)(0.48,2.80337996673619)(0.49,3.1232507483328)(0.5,3.42114248853384)(0.51,3.69542165978056)(0.52,3.94460703666812)(0.53,4.16736689260782)(0.54,4.36252734617222)(0.55,4.52907112111318)(0.56,4.66614622027928)(0.57,4.77306692188065)(0.58,4.84931756302518)(0.59,4.89455621631653)(0.6,4.90861299998144)(0.61,4.89149482838829)(0.62,4.84338002735696)(0.63,4.76462165800932)(0.64,4.65574091942881)(0.65,4.5174268431349)(0.66,4.35053045122967)(0.67,4.15605989102099)(0.68,3.93517635078622)(0.69,3.68918544923337)(0.7,3.41953543533107)(0.71,3.12780815595261)(0.72,2.81572223169464)(0.73,2.48513273449782)(0.74,2.13805943267346)(0.75,1.77674401538869)(0.76,1.40384470493505)(0.77,1.02308095166557)(0.78,0.642155749603336)(0.79,0.293143988402373)(0.8,0.109092060310531)(0.81,0.0670467938549702)(0.82,0.0534581239094407)(0.83,0.0467772583672266)(0.84,0.0425980192031535)(0.85,0.0395309773466783)(0.86,0.0370069891036735)(0.87,0.0347505774648587)(0.88,0.0326118578597273)(0.89,0.0305018212577488)(0.9,0.0283637107899417)(0.91,0.026159166282203)(0.92,0.0238613046605936)(0.93,0.021451016132784)(0.94,0.0189149805253748)(0.95,0.0162464948456364)(0.96,0.0134444138939917)(0.97,0.0105165678281909)(0.98,0.00748253429522489)(0.99,0.00437347896764165)(1,0.00123577839183722)(1.01,-0.00186861806721558)(1.02,-0.00486787133814401)(1.03,-0.0076898730413695)(1.04,-0.0102711150713282)(1.05,-0.0125700729037928)(1.06,-0.0145753575795597)(1.07,-0.0163068357632155)(1.08,-0.0178072585810828)(1.09,-0.0191352193944021)(1.1,-0.0203510279845125)(1.11,-0.0215080645716741)(1.12,-0.0226500422280008)(1.13,-0.0238066470726223)(1.14,-0.0249967642590846)(1.15,-0.0262296679672549)(1.16,-0.0275072150509182)(1.17,-0.0288272697585295)(1.18,-0.0301838382148252)(1.19,-0.0315697058998054)(1.2,-0.0329771732235432)(1.21,-0.0343978065050068)(1.22,-0.035824667430823)(1.23,-0.0372505338440815)(1.24,-0.0386700788485916)(1.25,-0.0400784784576208)(1.26,-0.041472319304626)(1.27,-0.0428495461131471)(1.28,-0.0442087448902079)(1.29,-0.0455502083188558)(1.3,-0.0468743698259224)(1.31,-0.0481826301268262)(1.32,-0.0494764966800794)(1.33,-0.0507568747109608)(1.34,-0.0520248530454413)(1.35,-0.0532797035225167)(1.36,-0.0545203908454403)(1.37,-0.0557443789040263)(1.38,-0.0569479817760199)(1.39,-0.0581275255238505)(1.4,-0.0592781309155359)(1.41,-0.0603963300324816)(1.42,-0.0614787035785009)(1.43,-0.0625240602332817)(1.44,-0.0635332233441878)(1.45,-0.064510115897249)(1.46,-0.0654630492304396)(1.47,-0.0664052914039056)(1.48,-0.06735833406796)(1.49,-0.0683535356027897)(1.5,-0.0694391312954357)(1.51,-0.0706878756517889)(1.52,-0.0722151154878412)(1.53,-0.0742112491319682)(1.54,-0.0770121862343967)(1.55,-0.0812639804551425)(1.56,-0.0883648139133647)(1.57,-0.101873901930862)(1.58,-0.132737689855385)(1.59,-0.213631782038865)(1.6,-0.372062321997298)(1.61,-0.567530791227421)(1.62,-0.76605985771143)(1.63,-0.955715669990985)(1.64,-1.13147824210884)(1.65,-1.29041058383934)(1.66,-1.43037528740785)(1.67,-1.54963272722542)(1.68,-1.64669651512377)(1.69,-1.72027855733603)(1.7,-1.76926360362887)(1.71,-1.79270592255872)(1.72,-1.78982511145153)(1.73,-1.76001220128988)(1.74,-1.70283566278278)(1.75,-1.6180506415423)(1.76,-1.50561928105789)(1.77,-1.36573793133874)(1.78,-1.19891283449992)(1.79,-1.00611998311164)(1.8,-0.789289523203278)(1.81,-0.553033499986446)(1.82,-0.312803532216965)(1.83,-0.128088527444266)(1.84,-0.0535484689023139)(1.85,-0.0268362517000334)(1.86,-0.0132026329610657)(1.87,-0.00422244151842109)(1.88,0.00266101300094789)(1.89,0.00844929967570343)(1.9,0.0136125520801265)(1.91,0.0184050263118468)(1.92,0.0229838961773958)(1.93,0.027462748144887)(1.94,0.0319419929803335)(1.95,0.0365332251651161)(1.96,0.0413914515484972)(1.97,0.0467733070344864)(1.98,0.053181652706382)(1.99,0.0617921305511805)(2,0.0761445527760308)(2.01,0.112327816823384)(2.02,0.266568813846708)(2.03,0.603136254513924)(2.04,0.9743161467748)(2.05,1.33709048179739)(2.06,1.68263873747331)(2.07,2.00731330502533)(2.08,2.30879467076463)(2.09,2.58529233367693)(2.1,2.8353170297311)(2.11,3.05760314763453)(2.12,3.25107645176686)(2.13,3.41484626899927)(2.14,3.54819800527607)(2.15,3.65059435330456)(2.16,3.72167482159148)(2.17,3.76125465205185)(2.18,3.76932730472607)(2.19,3.74605973624567)(2.2,3.69179533051016)(2.21,3.60704696363137)(2.22,3.49249815417763)(2.23,3.34899622052524)(2.24,3.17755047656735)(2.25,2.97932783147401)(2.26,2.75564977236814)(2.27,2.5079945018135)(2.28,2.23800246800366)(2.29,1.94750466748462)(2.3,1.63859156724557)(2.31,1.31381829721407)(2.32,0.976894339566175)(2.33,0.635719070532273)(2.34,0.321930595366496)(2.35,0.151647905864242)(2.36,0.107439343004403)(2.37,0.091599058133662)(2.38,0.0830380218710515)(2.39,0.0771867798661249)(2.4,0.0725775235368055)(2.41,0.0686034772572362)(2.42,0.0649690345168872)(2.43,0.0615121573920571)(2.44,0.0581352657027948)(2.45,0.054774489981995)(2.46,0.0513847591723925)(2.47,0.0479314666483599)(2.48,0.0443861122187452)(2.49,0.0407236912840074)(2.5,0.0369207209058922)(2.51,0.0329573591586294)(2.52,0.0288154556366375)(2.53,0.0244839072570559)(2.54,0.0199632405109376)(2.55,0.0152722972560461)(2.56,0.0104591515068878)(2.57,0.0056099676654345)(2.58,0.000853618478288032)(2.59,-0.00364376995513201)(2.6,-0.00770423337715593)(2.61,-0.0111743958418018)(2.62,-0.0139678024889174)(2.63,-0.0160881009090644)(2.64,-0.017624238997393)(2.65,-0.0187213424158094)(2.66,-0.0195393066978052)(2.67,-0.0202213957989261)(2.68,-0.0208761870285277)(2.69,-0.021574270949011)(2.7,-0.0223535324154531)(2.71,-0.0232288884996042)(2.72,-0.0241992730287349)(2.73,-0.0252527014708802)(2.74,-0.0263752599253234)(2.75,-0.0275491747206918)(2.76,-0.028758297438491)(2.77,-0.0299864218722631)(2.78,-0.0312197269331881)(2.79,-0.0324454617893049)(2.8,-0.0336533391858787)(2.81,-0.0348342622071713)(2.82,-0.0359815481648785)(2.83,-0.0370894851577369)(2.84,-0.0381546563867865)(2.85,-0.0391743951294606)(2.86,-0.0401478591618953)(2.87,-0.0410748777584968)(2.88,-0.0419561128887296)(2.89,-0.0427928793967739)(2.9,-0.0435860211994321)(2.91,-0.0443368489592764)(2.92,-0.0450453172977192)(2.93,-0.0457111902792859)(2.94,-0.0463331435855785)(2.95,-0.0469088755124179)(2.96,-0.0474361937738398)(2.97,-0.0479118828009018)(2.98,-0.0483340593105196)(2.99,-0.0487010335081747)(3,-0.0489370812750203)(3.01,0.0686186752250318)(3.02,0.167737382399212)(3.03,0.250983219105996)(3.04,0.32055439796139)(3.05,0.378341373804844)(3.06,0.425972685799611)(3.07,0.464848957824581)(3.08,0.496179790824415)(3.09,0.521008948594678)(3.1,0.54023847224724)(3.11,0.554649886068867)(3.12,0.5649173190227)(3.13,0.571625277988575)(3.14,0.575274962407094)(3.15,0.576298642146883)(3.16,0.575062243276971)(3.17,0.571876100799944)(3.18,0.566994256697871)(3.19,0.560619798890369)(3.2,0.552896834014658)(3.21,0.543901339733443)(3.22,0.533603512821837)(3.23,0.521788865356763)(3.24,0.507815677241599)(3.25,0.489828654920324)(3.26,0.4611133482236)(3.27,0.385286262177988)(3.28,0.139648658025569)(3.29,-0.223047987962951)(3.3,-0.608137556175662)(3.31,-0.995819626483327)(3.32,-1.37988298604429)(3.33,-1.75677188291365)(3.34,-2.12373233784254)(3.35,-2.4783528905814)(3.36,-2.81842176823214)(3.37,-3.14187549333663)(3.38,-3.44678575282438)(3.39,-3.73135188633301)(3.4,-3.99390801478117)(3.41,-4.23292590487713)(3.42,-4.44701971037587)(3.43,-4.63495609422858)(3.44,-4.79565480808984)(3.45,-4.92819858512024)(3.46,-5.03183442043678)(3.47,-5.10597813819156)(3.48,-5.15021817713631)(3.49,-5.16431429873857)(3.5,-5.14820250207906)(3.51,-5.10198986940303)(3.52,-5.02595798001171)(3.53,-4.92055656255782)(3.54,-4.78640306899364)(3.55,-4.62427710195819)(3.56,-4.43511514738269)(3.57,-4.22000653885337)(3.58,-3.98018397059213)(3.59,-3.71702079862436)(3.6,-3.43201976778688)(3.61,-3.12681151420638)(3.62,-2.80314672189124)(3.63,-2.46290070908929)(3.64,-2.10808393408413)(3.65,-1.74088710754395)(3.66,-1.36380366600885)(3.67,-0.980031547638248)(3.68,-0.595069972173824)(3.69,-0.226125262447762)(3.7,0.0346596490090248)(3.71,0.105738999741204)(3.72,0.122324112400449)(3.73,0.127937641775368)(3.74,0.130066158482636)(3.75,0.130766399053722)(3.76,0.130812689875884)(3.77,0.130558740534962)(3.78,0.130188060676835)(3.79,0.12980472240573)(3.8,0.129471399840624)(3.81,0.129227144239235)(3.82,0.129096558297825)(3.83,0.129094019192052)(3.84,0.129226246783501)(3.85,0.129493226744688)(3.86,0.129888277475164)(3.87,0.130398446872682)(3.88,0.131003902625631)(3.89,0.131678641092295)(3.9,0.132391940995468)(3.91,0.133110197862278)(3.92,0.133800360391013)(3.93,0.134434490154576)(3.94,0.134992610044985)(3.95,0.135465003525922)(3.96,0.135854534552333)(3.97,0.136172987761475)(3.98,0.136437949446423)(3.99,0.136673551602771)(4,0.13690093403012)(4.01,0.137140014127749)(4.02,0.137406681293652)(4.03,0.137711943383184)(4.04,0.138063043031704)(4.05,0.138463051573852)(4.06,0.138912218327841)(4.07,0.139408758787988)(4.08,0.139949596207751)(4.09,0.140529840166482)(4.1,0.141145061570134)(4.11,0.141791156365203)(4.12,0.142463291356253)(4.13,0.14315803380565)(4.14,0.143872840668717)(4.15,0.144605419258427)(4.16,0.145355175908641)(4.17,0.146121967051812)(4.18,0.146906306829776)(4.19,0.14770974085272)(4.2,0.14853317395199)(4.21,0.149377947510293)(4.22,0.150244414392161)(4.23,0.151131582338127)(4.24,0.152037995288569)(4.25,0.152959691350558)(4.26,0.153892480303938)(4.27,0.154830609066775)(4.28,0.155767911418691)(4.29,0.156698842199157)(4.3,0.15761787200782)(4.31,0.158521974877446)(4.32,0.159409500615571)(4.33,0.160282751130388)(4.34,0.161147616755013)(4.35,0.162015957744364)(4.36,0.162907258585223)(4.37,0.163851861507388)(4.38,0.164897465908761)(4.39,0.166117629351265)(4.4,0.167633129051585)(4.41,0.169646637303564)(4.42,0.172531902336994)(4.43,0.177039934407024)(4.44,0.184915415103471)(4.45,0.201098210570159)(4.46,0.243304494533601)(4.47,0.363910153503674)(4.48,0.573300168616604)(4.49,0.809852592074623)(4.5,1.04595906041903)(4.51,1.27289663261494)(4.52,1.48659028937374)(4.53,1.68436919047559)(4.54,1.86412772470485)(4.55,2.02406036689569)(4.56,2.16256663370713)(4.57,2.27821166839761)(4.58,2.36971741638042)(4.59,2.43595426496319)(4.6,2.47594710045887)(4.61,2.48887662006705)(4.62,2.47408383257349)(4.63,2.43107739458975)(4.64,2.35953506123358)(4.65,2.25931469853537)(4.66,2.13045739541266)(4.67,1.97320516480776)(4.68,1.7880178985891)(4.69,1.57563022276904)(4.7,1.33716025391491)(4.71,1.07444950313249)(4.72,0.791230454323727)(4.73,0.498767716139978)(4.74,0.249589916145388)(4.75,0.143428148066843)(4.76,0.111608798223725)(4.77,0.0968623255917395)(4.78,0.0874275730012959)(4.79,0.0802186853413661)(4.8,0.074119840000508)(4.81,0.0686306429655788)(4.82,0.0634871724727113)(4.83,0.058526938091499)(4.84,0.0536296348293142)(4.85,0.0486836583833015)(4.86,0.0435546628348935)(4.87,0.0380445798396453)(4.88,0.031798880059914)(4.89,0.0240694132792978)(4.9,0.0129125235490367)(4.91,-0.00858552230096849)(4.92,-0.0797352253393287)(4.93,-0.331851361378427)(4.94,-0.686948290672439)(4.95,-1.04483436243143)(4.96,-1.38718154562673)(4.97,-1.70872171786637)(4.98,-2.00675548189338)(4.99,-2.27941692214889)(5,-2.52524311127524)(5.01,-2.74303519147645)(5.02,-2.93180779217995)(5.03,-3.09076886057762)(5.04,-3.21931113660341)(5.05,-3.31700707329921)(5.06,-3.38360987592868)(5.07,-3.41904839671653)(5.08,-3.42343027874722)(5.09,-3.39703552181071)(5.1,-3.34031814494024)(5.11,-3.25389996348781)(5.12,-3.13856879697605)(5.13,-2.99527404015072)(5.14,-2.82512168981005)(5.15,-2.6293733429901)(5.16,-2.40944196931418)(5.17,-2.16690035462937)(5.18,-1.90349204526875)(5.19,-1.62118670816792)(5.2,-1.32230387025174)(5.21,-1.00992274344161)(5.22,-0.689413463433339)(5.23,-0.376484874137181)(5.24,-0.141966332683014)(5.25,-0.0612438222446741)(5.26,-0.0376089196090605)(5.27,-0.0269361237146503)(5.28,-0.020512866263776)(5.29,-0.0159023274397279)(5.3,-0.0121858882837855)(5.31,-0.00894392207243947)(5.32,-0.00595797514288121)(5.33,-0.00310289339264296)(5.34,-0.000301445020622454)(5.35,0.00249681217432014)(5.36,0.00532665049191196)(5.37,0.00821270209524927)(5.38,0.0111726630314956)(5.39,0.0142186104795058)(5.4,0.0173575420434026)(5.41,0.0205911758658794)(5.42,0.0239134448114125)(5.43,0.0273111998948846)(5.44,0.0307607259746318)(5.45,0.0342258046511352)(5.46,0.0376592572959997)(5.47,0.04100556256934)(5.48,0.0442046829207216)(5.49,0.0472006696053015)(5.5,0.0499521152076406)(5.51,0.052438433691952)(5.52,0.054663051887718)(5.53,0.0566515992845692)(5.54,0.0584465768427057)(5.55,0.0600956684027121)(5.56,0.0616474398469566)(5.57,0.0631422005369266)(5.58,0.0646124883499993)(5.59,0.0660798419643972)(5.6,0.0675576689922638)(5.61,0.0690522029328835)(5.62,0.0705642385618265)(5.63,0.0720905971508034)(5.64,0.0736275871579037)(5.65,0.0751681680661193)(5.66,0.0767052960310798)(5.67,0.0782330905617032)(5.68,0.0797448976681099)(5.69,0.0812357760065217)(5.7,0.0827011512087244)(5.71,0.0841376749319522)(5.72,0.0855434467701246)(5.73,0.0869171715590531)(5.74,0.0882593517780006)(5.75,0.0895709291564644)(5.76,0.0908537628598106)(5.77,0.0921103199023222)(5.78,0.0933424290404573)(5.79,0.0945522908134076)(5.8,0.0957404861657717)(5.81,0.0969067601182095)(5.82,0.0980497030845955)(5.83,0.0991658134438727)(5.84,0.100251635952202)(5.85,0.101301910781476)(5.86,0.102312179280228)(5.87,0.103278088629019)(5.88,0.104196600770595)(5.89,0.105067011399365)(5.9,0.10589089383095)(5.91,0.10667430906256)(5.92,0.107427556151623)(5.93,0.108168574868572)(5.94,0.108924200000198)(5.95,0.109736084981526)(5.96,0.110667965793628)(5.97,0.111820337896525)(5.98,0.113360293308367)(5.99,0.115580601508946)(6,0.119045908796228) 
};

\addplot [
color=green,
dashed,
line width=1.0pt
]
coordinates{
 (0,-0.0893688346950474)(0.01,-0.13806379539123)(0.02,-0.258260194867891)(0.03,-0.442488708429957)(0.04,-0.643946383240675)(0.05,-0.841906020744194)(0.06,-1.02887694672146)(0.07,-1.20110605579463)(0.08,-1.35609286054076)(0.09,-1.49188044598819)(0.1,-1.60681765285086)(0.11,-1.69946710395633)(0.12,-1.76857140488287)(0.13,-1.81303881723967)(0.14,-1.83194007987889)(0.15,-1.8245101611354)(0.16,-1.79015303622085)(0.17,-1.72844756530384)(0.18,-1.63915904954578)(0.19,-1.52225337066439)(0.2,-1.37792768840844)(0.21,-1.20667028722598)(0.22,-1.00940615359603)(0.23,-0.787905595315345)(0.24,-0.546218044309336)(0.25,-0.297295662201519)(0.26,-0.0939133854221308)(0.27,-0.00580207130055762)(0.28,0.0249156937548898)(0.29,0.0403730248589389)(0.3,0.0506368599719938)(0.31,0.0586190412516191)(0.32,0.0654235549242878)(0.33,0.0715572304547312)(0.34,0.0772893304863507)(0.35,0.0827835377922453)(0.36,0.0881561711858024)(0.37,0.0935082281948582)(0.38,0.0989500253889298)(0.39,0.104633229501795)(0.4,0.11080694935525)(0.41,0.117954407013482)(0.42,0.127193815439342)(0.43,0.141833191382938)(0.44,0.17645825246177)(0.45,0.319668269173988)(0.46,0.649327067060863)(0.47,1.01880768500662)(0.48,1.38053699625055)(0.49,1.72506852680965)(0.5,2.04863435083917)(0.51,2.34888656641702)(0.52,2.62403119376325)(0.53,2.87258440660577)(0.54,3.09328944120719)(0.55,3.28508434184239)(0.56,3.44708991731768)(0.57,3.57860551918013)(0.58,3.67910743333529)(0.59,3.74824902454933)(0.6,3.78586055586778)(0.61,3.79194894271154)(0.62,3.76669677447138)(0.63,3.71046063770543)(0.64,3.62376886783474)(0.65,3.50731828332902)(0.66,3.36197098619456)(0.67,3.18875018199946)(0.68,2.98883721157398)(0.69,2.76356947213922)(0.7,2.51444156579455)(0.71,2.24311798907389)(0.72,1.95146194808232)(0.73,1.6416277458037)(0.74,1.3163202536681)(0.75,0.979718750399111)(0.76,0.641860665930253)(0.77,0.34791953889044)(0.78,0.216483915603201)(0.79,0.182443572976614)(0.8,0.168776665581019)(0.81,0.160843933662671)(0.82,0.155152248973203)(0.83,0.150503869021249)(0.84,0.146383048530398)(0.85,0.14253123547173)(0.86,0.138803494667787)(0.87,0.135110881746924)(0.88,0.131394132272978)(0.89,0.127610398779833)(0.9,0.123726181726274)(0.91,0.119713044719756)(0.92,0.115545414660598)(0.93,0.111199382810002)(0.94,0.10665268609891)(0.95,0.101886185310206)(0.96,0.0968873560845896)(0.97,0.0916552359546315)(0.98,0.0862103805757829)(0.99,0.0806045429886556)(1,0.0749329919332908)(1.01,0.0693389139838564)(1.02,0.0640067228242678)(1.03,0.0591325375272647)(1.04,0.0548820147470637)(1.05,0.0513422594654548)(1.06,0.0484982961078953)(1.07,0.0462417324595176)(1.08,0.0444071547406176)(1.09,0.0428182668264044)(1.1,0.0413225435928301)(1.11,0.0398081714834969)(1.12,0.0382053414939484)(1.13,0.0364790413926899)(1.14,0.0346186532661842)(1.15,0.0326294073364054)(1.16,0.030525089651772)(1.17,0.02832360147816)(1.18,0.0260439960389595)(1.19,0.0237048264804636)(1.2,0.0213232537273553)(1.21,0.018914634267974)(1.22,0.0164923685074978)(1.23,0.014067886571991)(1.24,0.0116507078416572)(1.25,0.00924852402151299)(1.26,0.00686730376419832)(1.27,0.00451141868812389)(1.28,0.00218377739423178)(1.29,-0.00011396400679204)(1.3,-0.00238120186934639)(1.31,-0.00461806407609453)(1.32,-0.00682510273499735)(1.33,-0.00900293365900291)(1.34,-0.0111518689234703)(1.35,-0.0132716515996843)(1.36,-0.0153612323197793)(1.37,-0.0174187167615305)(1.38,-0.0194414573186641)(1.39,-0.0214263058849366)(1.4,-0.0233699823440864)(1.41,-0.0252696197173352)(1.42,-0.0271232738946486)(1.43,-0.0289306626503687)(1.44,-0.0306938143422894)(1.45,-0.0324180202707295)(1.46,-0.0341129154914752)(1.47,-0.0357940415390289)(1.48,-0.0374853095503269)(1.49,-0.0392224905318325)(1.5,-0.0410600271527358)(1.51,-0.0430811526281073)(1.52,-0.0454199567002986)(1.53,-0.0483005388300262)(1.54,-0.0521290637472948)(1.55,-0.0577108359074181)(1.56,-0.0668661965017778)(1.57,-0.0844896701861843)(1.58,-0.126224666280137)(1.59,-0.231048528502788)(1.6,-0.406386223461131)(1.61,-0.605660048606882)(1.62,-0.803668126137339)(1.63,-0.991521286576724)(1.64,-1.16504651574843)(1.65,-1.32158834639629)(1.66,-1.45912044195162)(1.67,-1.57595213408151)(1.68,-1.67062247810084)(1.69,-1.74185446772197)(1.7,-1.78854090195186)(1.71,-1.80973846990141)(1.72,-1.80466886206646)(1.73,-1.77272328107171)(1.74,-1.71346956553958)(1.75,-1.62666155719764)(1.76,-1.51225878883446)(1.77,-1.37045429538335)(1.78,-1.20174599393263)(1.79,-1.00709483817451)(1.8,-0.788393200229042)(1.81,-0.550133761780268)(1.82,-0.307297358545626)(1.83,-0.118653311089954)(1.84,-0.0422742120790117)(1.85,-0.015473654456657)(1.86,-0.00207861125774939)(1.87,0.00659719812012137)(1.88,0.0131624374772567)(1.89,0.0186323233606428)(1.9,0.0234815742520789)(1.91,0.0279662462223912)(1.92,0.0322443839618916)(1.93,0.0364303532171249)(1.94,0.040625327710963)(1.95,0.0449425913225278)(1.96,0.0495399708880602)(1.97,0.0546809239428896)(1.98,0.0608850705978906)(1.99,0.0693794593823357)(2,0.083918021502804)(2.01,0.122029270605153)(2.02,0.286607387632891)(2.03,0.629108097618218)(2.04,1.0019474514689)(2.05,1.36569126047462)(2.06,1.71201612294783)(2.07,2.03737985330548)(2.08,2.33949578392535)(2.09,2.61658636792103)(2.1,2.86716907901261)(2.11,3.08998106748405)(2.12,3.28395149814445)(2.13,3.44819015010082)(2.14,3.58198429338422)(2.15,3.68479744765961)(2.16,3.75626961593075)(2.17,3.79621750420788)(2.18,3.80463415754192)(2.19,3.78168832713022)(2.2,3.72772260444371)(2.21,3.64325144886785)(2.22,3.52895769375679)(2.23,3.38568944882543)(2.24,3.21445560120031)(2.25,3.01642247937174)(2.26,2.7929108844172)(2.27,2.54539583667819)(2.28,2.27551425155413)(2.29,1.98508652463)(2.3,1.67618207165682)(2.31,1.35130091063146)(2.32,1.01398012369812)(2.33,0.671372561826275)(2.34,0.350368119540856)(2.35,0.161097058360586)(2.36,0.110623666153891)(2.37,0.0933896630270801)(2.38,0.0843100804836127)(2.39,0.0781825410610381)(2.4,0.0733890030277204)(2.41,0.0692740478151525)(2.42,0.0655227706990441)(2.43,0.061964349676535)(2.44,0.0584968955623145)(2.45,0.0550544058757958)(2.46,0.051590684464411)(2.47,0.0480708802720094)(2.48,0.0444666177894182)(2.49,0.0407534711821787)(2.5,0.0369096397497965)(2.51,0.0329158339948213)(2.52,0.0287563214681561)(2.53,0.0244224038130709)(2.54,0.0199159496724562)(2.55,0.0152575641076996)(2.56,0.0104945776927158)(2.57,0.00570968277346354)(2.58,0.00102370260353268)(2.59,-0.00341040204278987)(2.6,-0.00742989647637776)(2.61,-0.0108962931386523)(2.62,-0.0137314266895115)(2.63,-0.0159382387607462)(2.64,-0.0175967911078386)(2.65,-0.0188377022207808)(2.66,-0.0198067296977339)(2.67,-0.0206356927423308)(2.68,-0.0214259292313946)(2.69,-0.0222447453103087)(2.7,-0.0231298350565935)(2.71,-0.0240963214515649)(2.72,-0.0251445575361614)(2.73,-0.0262657269448446)(2.74,-0.0274464179063649)(2.75,-0.0286712356622046)(2.76,-0.029924656827719)(2.77,-0.0311919542650959)(2.78,-0.0324597633048851)(2.79,-0.0337163202793718)(2.8,-0.0349515652517316)(2.81,-0.0361571565537717)(2.82,-0.0373264329269892)(2.83,-0.0384543565781775)(2.84,-0.0395374203226766)(2.85,-0.0405735206836537)(2.86,-0.0415618026328283)(2.87,-0.0425024078414641)(2.88,-0.0433962345938104)(2.89,-0.04424456204193)(2.9,-0.0450487219865698)(2.91,-0.0458097308967763)(2.92,-0.0465279766907274)(2.93,-0.0472030266790254)(2.94,-0.0478335543332515)(2.95,-0.0484174186397158)(2.96,-0.0489519358427213)(2.97,-0.0494342573279889)(2.98,-0.049861869938226)(2.99,-0.0502332094493222)(3,-0.0505482869210551)(3.01,0.81944407802834)(3.02,0.914630569649914)(3.03,0.984299687837716)(3.04,1.0272139536981)(3.05,1.0425195962882)(3.06,1.02965968063764)(3.07,0.988353131384707)(3.08,0.91862366543545)(3.09,0.820888261784532)(3.1,0.696214916996154)(3.11,0.54706418449555)(3.12,0.379871851095748)(3.13,0.214798588380515)(3.14,0.0994378960039344)(3.15,0.0476393867438054)(3.16,0.0241205816107178)(3.17,0.0106003231140388)(3.18,0.00124118975384148)(3.19,-0.00607936802561258)(3.2,-0.0122805869556648)(3.21,-0.0178197926874671)(3.22,-0.0229514766614511)(3.23,-0.0278327246136133)(3.24,-0.0325725325295071)(3.25,-0.0372589533376626)(3.26,-0.0419777042070755)(3.27,-0.0468327182299976)(3.28,-0.0519738546900385)(3.29,-0.0576571296630261)(3.3,-0.0643876359465237)(3.31,-0.0733503988251698)(3.32,-0.0881134480834461)(3.33,-0.124810586409331)(3.34,-0.278259246690709)(3.35,-0.611434878513567)(3.36,-0.980214464191572)(3.37,-1.34185462060308)(3.38,-1.68746628742993)(3.39,-2.01334940254342)(3.4,-2.31713870437305)(3.41,-2.59700046623762)(3.42,-2.85140157879761)(3.43,-3.07903219897863)(3.44,-3.27877389075971)(3.45,-3.44968924229751)(3.46,-3.59101756595925)(3.47,-3.70217375010411)(3.48,-3.78274867903019)(3.49,-3.8325091455325)(3.5,-3.8513979066413)(3.51,-3.83953269311271)(3.52,-3.79720481869561)(3.53,-3.72487675523891)(3.54,-3.62317903132353)(3.55,-3.49290661772265)(3.56,-3.33501436109705)(3.57,-3.15061295596098)(3.58,-2.94096445657136)(3.59,-2.70747996491079)(3.6,-2.45172155070649)(3.61,-2.17540996868926)(3.62,-1.88046178657412)(3.63,-1.56907170655885)(3.64,-1.24398043228776)(3.65,-0.909391148643642)(3.66,-0.57528681928292)(3.67,-0.285499437317854)(3.68,-0.153755566357801)(3.69,-0.118814875181292)(3.7,-0.104648614049174)(3.71,-0.0963579588329871)(3.72,-0.0903682648524708)(3.73,-0.0854537032461541)(3.74,-0.0810873585227468)(3.75,-0.0770061645855504)(3.76,-0.0730640388060965)(3.77,-0.0691730268272084)(3.78,-0.0652766919955316)(3.79,-0.0613369109303843)(3.8,-0.0573269495067795)(3.81,-0.0532278600905252)(3.82,-0.0490266083480057)(3.83,-0.0447155102851164)(3.84,-0.0402926078714416)(3.85,-0.0357626986267821)(3.86,-0.0311388662064652)(3.87,-0.0264445980767676)(3.88,-0.0217146292470526)(3.89,-0.0169959820263814)(3.9,-0.0123443791229431)(3.91,-0.00782026438765056)(3.92,-0.0034806659009312)(3.93,0.000631265719779005)(3.94,0.00449234554858827)(3.95,0.00810444846740579)(3.96,0.0114915242254795)(3.97,0.0146943591265808)(3.98,0.0177609236702051)(3.99,0.0207390220627479)(4,0.0236700303971651)(4.01,0.0265861346249267)(4.02,0.0295096654836783)(4.03,0.0324540288679303)(4.04,0.0354254125041661)(4.05,0.0384246603162192)(4.06,0.0414488249544561)(4.07,0.0444927093924299)(4.08,0.0475497510142737)(4.09,0.0506129097642178)(4.1,0.0536751908385637)(4.11,0.0567300350561878)(4.12,0.059771633025553)(4.13,0.0627950940197026)(4.14,0.0657965983979293)(4.15,0.0687734598897805)(4.16,0.0717241379111295)(4.17,0.0746481533293761)(4.18,0.0775459644236702)(4.19,0.0804186870408049)(4.2,0.083267791138923)(4.21,0.0860947114157214)(4.22,0.0889003921027579)(4.23,0.0916849583346972)(4.24,0.0944473843605856)(4.25,0.097185380023264)(4.26,0.0998954281370523)(4.27,0.102573062867762)(4.28,0.105213243029046)(4.29,0.107810987034637)(4.3,0.110361938955373)(4.31,0.11286314070389)(4.32,0.115313742781343)(4.33,0.117715847795951)(4.34,0.120075558547182)(4.35,0.12240423510981)(4.36,0.124720688045969)(4.37,0.127053933298014)(4.38,0.129449054372323)(4.39,0.131975314932857)(4.4,0.134744759577346)(4.41,0.137943471143497)(4.42,0.141907810178386)(4.43,0.147298711682782)(4.44,0.155600521155246)(4.45,0.170816502980613)(4.46,0.206566273786991)(4.47,0.307274787212765)(4.48,0.500923597018143)(4.49,0.730802378409243)(4.5,0.963012124524347)(4.51,1.18697502122137)(4.52,1.39811065247773)(4.53,1.59357534443957)(4.54,1.77119156989812)(4.55,1.92911957116955)(4.56,2.06573943623765)(4.57,2.17960626732115)(4.58,2.26943292282372)(4.59,2.33408594043138)(4.6,2.37258609660374)(4.61,2.38411134209522)(4.62,2.36800188677611)(4.63,2.32376475896322)(4.64,2.25107984517093)(4.65,2.14980698242511)(4.66,2.01999453209606)(4.67,1.86189721160411)(4.68,1.67600242466086)(4.69,1.46309983444427)(4.7,1.22444250257189)(4.71,0.962244051025958)(4.72,0.681576248306522)(4.73,0.400507593224745)(4.74,0.1976258822566)(4.75,0.128735287386525)(4.76,0.105492991267358)(4.77,0.0934744630534847)(4.78,0.0853814767740871)(4.79,0.079044122390834)(4.8,0.0736110445835322)(4.81,0.0686782629472606)(4.82,0.0640222416640771)(4.83,0.0594981597204189)(4.84,0.0549917579442118)(4.85,0.0503894557466039)(4.86,0.0455443574138375)(4.87,0.0402273320130931)(4.88,0.03400737912256)(4.89,0.0259220691456636)(4.9,0.0132611131904303)(4.91,-0.0150012545434657)(4.92,-0.12731109443333)(4.93,-0.434784832300743)(4.94,-0.800543580351466)(4.95,-1.16073651111121)(4.96,-1.50371377988702)(4.97,-1.82534751300077)(4.98,-2.12324253760393)(4.99,-2.39564082380436)(5,-2.64112584115092)(5.01,-2.85852127736574)(5.02,-3.04685477067154)(5.03,-3.20534139454883)(5.04,-3.33337840484378)(5.05,-3.43054216558218)(5.06,-3.49658715279296)(5.07,-3.5314451092385)(5.08,-3.53522364005427)(5.09,-3.50820476724987)(5.1,-3.45084210901738)(5.11,-3.36375824348167)(5.12,-3.24774041706069)(5.13,-3.10373694470371)(5.14,-2.93285234322418)(5.15,-2.73634381693743)(5.16,-2.51561904474046)(5.17,-2.27223777660814)(5.18,-2.00792316067655)(5.19,-1.72459686673818)(5.2,-1.42447255758417)(5.21,-1.1103366815486)(5.22,-0.786533238900782)(5.23,-0.463608373311024)(5.24,-0.187504174477008)(5.25,-0.0652706743204082)(5.26,-0.0324512881751084)(5.27,-0.0193362165345932)(5.28,-0.0119688326716214)(5.29,-0.00689873147414421)(5.3,-0.00292410485240306)(5.31,0.000476424557539765)(5.32,0.0035646254532955)(5.33,0.00648634773411136)(5.34,0.0093295219009285)(5.35,0.0121505364456475)(5.36,0.0149872735795481)(5.37,0.0178660777215006)(5.38,0.0208053201228477)(5.39,0.0238172950157879)(5.4,0.026908847284336)(5.41,0.0300811038789639)(5.42,0.033328758679796)(5.43,0.0366386468957259)(5.44,0.0399884993263283)(5.45,0.0433459528019112)(5.46,0.0466689873396479)(5.47,0.0499084369017309)(5.48,0.0530118425963633)(5.49,0.0559315820871375)(5.5,0.058631381301859)(5.51,0.061093005791284)(5.52,0.0633185863156833)(5.53,0.0653293030569967)(5.54,0.0671599396997446)(5.55,0.0688517093175223)(5.56,0.0704458856823251)(5.57,0.0719786178695572)(5.58,0.0734785305804127)(5.59,0.0749659543280563)(5.6,0.0764535304358713)(5.61,0.0779476631505066)(5.62,0.0794498753635532)(5.63,0.0809583098953812)(5.64,0.0824689568890147)(5.65,0.0839764750486094)(5.66,0.0854750474111294)(5.67,0.0869588231023162)(5.68,0.0884223441433928)(5.69,0.0898608523438309)(5.7,0.0912704736202609)(5.71,0.0926483963155939)(5.72,0.0939929491741209)(5.73,0.0953036344448536)(5.74,0.0965810554783773)(5.75,0.0978267973077838)(5.76,0.0990431366934934)(5.77,0.100232716202149)(5.78,0.101398121736709)(5.79,0.102541371542925)(5.8,0.103663546427628)(5.81,0.104764404271652)(5.82,0.105842237373031)(5.83,0.106893894236269)(5.84,0.107915071414643)(5.85,0.108900704433013)(5.86,0.109845639620726)(5.87,0.110745229207868)(5.88,0.111596108048433)(5.89,0.112396900340626)(5.9,0.113148998523746)(5.91,0.113857510627677)(5.92,0.11453232708943)(5.93,0.115189983812611)(5.94,0.115855911755521)(5.95,0.116569237440488)(5.96,0.11738924556615)(5.97,0.118409977514187)(5.98,0.119783897352891)(5.99,0.121778354093553)(6,0.124896917908692) 
};

\end{axis}
\end{tikzpicture}

%% file: Flow3.tikz
%
%
%
%
\begin{tikzpicture}

\begin{axis}[%
width=0.7\textwidth,
height=0.7\textwidth,
scale only axis,
xmin=0, xmax=6,
xlabel={Time},
ymin=-5, ymax=5,
ylabel={$\lambda_{3}(t), \lambda_{3}^{w}(t)$},
axis lines=left,
axis on top]
\addplot [
color=red,
solid,
line width=1.0pt
]
coordinates{
 (0,-1.5)(0.01,-0.941611021162676)(0.02,-0.427315912009483)(0.03,0.045945724491864)(0.04,0.480576887177375)(0.05,0.878434029894253)(0.06,1.2409189357223)(0.07,1.56906586421273)(0.08,1.86361952182977)(0.09,2.12509358363611)(0.1,2.3538517275731)(0.11,2.55014279149358)(0.12,2.71416702539525)(0.13,2.84611469925006)(0.14,2.94618181958957)(0.15,3.0145896609808)(0.16,3.0515673508572)(0.17,3.05733554991144)(0.18,3.03210683676797)(0.19,2.97612062262665)(0.2,2.88966059661444)(0.21,2.77314153743867)(0.22,2.6271062243677)(0.23,2.4522658685849)(0.24,2.24949511277972)(0.25,2.01982988541031)(0.26,1.7644491485529)(0.27,1.48467643013773)(0.28,1.18197537590868)(0.29,0.857922085883099)(0.3,0.514201287907688)(0.31,0.152694222806322)(0.32,-0.224753298618702)(0.33,-0.616061338913054)(0.34,-1.01913344505656)(0.35,-1.43156931710645)(0.36,-1.85107776943853)(0.37,-2.27474692931012)(0.38,-2.69894500177815)(0.39,-3.11684483579623)(0.4,-3.50305263073866)(0.41,-3.73086187287466)(0.42,-3.7855975554889)(0.43,-3.80405599801658)(0.44,-3.81438497191511)(0.45,-3.82179270169372)(0.46,-3.82782704115706)(0.47,-3.83309792814349)(0.48,-3.83789105278919)(0.49,-3.84235762809914)(0.5,-3.84658273311318)(0.51,-3.8506189528478)(0.52,-3.85450127807285)(0.53,-3.85824970053971)(0.54,-3.86188173336257)(0.55,-3.86540612057573)(0.56,-3.86883134965683)(0.57,-3.87216260879526)(0.58,-3.87540282596836)(0.59,-3.87855585087326)(0.6,-3.88162112610574)(0.61,-3.88460121997907)(0.62,-3.88749347382709)(0.63,-3.89029824077904)(0.64,-3.89301140323059)(0.65,-3.89562975544031)(0.66,-3.89814708432515)(0.67,-3.90055459510031)(0.68,-3.90284126925903)(0.69,-3.90498858313128)(0.7,-3.90697337833348)(0.71,-3.90875660603158)(0.72,-3.91028207002738)(0.73,-3.91145467897465)(0.74,-3.91211181289038)(0.75,-3.91194625236944)(0.76,-3.91030384719623)(0.77,-3.90552481531223)(0.78,-3.8920083533952)(0.79,-3.83981195632126)(0.8,-3.61803967305587)(0.81,-3.25179628197117)(0.82,-2.85677550887269)(0.83,-2.45664592533107)(0.84,-2.05790468560499)(0.85,-1.66396503123398)(0.86,-1.27737648956986)(0.87,-0.900317742176305)(0.88,-0.534756059512628)(0.89,-0.182501107522748)(0.9,0.154765147050447)(0.91,0.475506242559808)(0.92,0.778313061226139)(0.93,1.06191745601484)(0.94,1.32519510864336)(0.95,1.56716642276036)(0.96,1.7870343132929)(0.97,1.98411139053847)(0.98,2.15793728982719)(0.99,2.30817235555792)(1,2.434685155204)(1.01,2.5375316552288)(1.02,2.61685449168881)(1.03,2.67306996965804)(1.04,2.70667175878863)(1.05,2.71828117437832)(1.06,2.70865313245121)(1.07,2.67863955575293)(1.08,2.62922004673959)(1.09,2.56144379884331)(1.1,2.4765243925531)(1.11,2.37574891701206)(1.12,2.26053520331013)(1.13,2.13239180537757)(1.14,1.99291372486827)(1.15,1.84377356417703)(1.16,1.68670049040253)(1.17,1.52346584091802)(1.18,1.35588009230648)(1.19,1.18575802862133)(1.2,1.01492888858617)(1.21,0.845202147128647)(1.22,0.678369606588606)(1.23,0.51618308405016)(1.24,0.360347547089965)(1.25,0.212506191435409)(1.26,0.0742302593823023)(1.27,-0.0529923768973698)(1.28,-0.167768676263193)(1.29,-0.26880537831168)(1.3,-0.354927503314089)(1.31,-0.425072208073543)(1.32,-0.478312167593249)(1.33,-0.513849148895483)(1.34,-0.531023170475004)(1.35,-0.529322239516351)(1.36,-0.50837806311876)(1.37,-0.4679724188739)(1.38,-0.408043592370513)(1.39,-0.328676571377323)(1.4,-0.230116093742747)(1.41,-0.112757212774531)(1.42,0.0228556017125838)(1.43,0.176022471115905)(1.44,0.34590937329704)(1.45,0.531526144814762)(1.46,0.731765238111988)(1.47,0.945370091273779)(1.48,1.17097552661051)(1.49,1.40709123009833)(1.5,1.65210312786889)(1.51,1.90430914726521)(1.52,2.16184590642928)(1.53,2.42274941042406)(1.54,2.68484516493272)(1.55,2.94558750877543)(1.56,3.20164423614096)(1.57,3.4475481826297)(1.58,3.67039541163819)(1.59,3.83558202149168)(1.6,3.91368754080213)(1.61,3.94333141334655)(1.62,3.95665884814135)(1.63,3.96382943970475)(1.64,3.96816044318049)(1.65,3.97096817872663)(1.66,3.97286338620529)(1.67,3.97416321881915)(1.68,3.97504436196356)(1.69,3.97561189518301)(1.7,3.97592468416368)(1.71,3.97601784770851)(1.72,3.9759030293671)(1.73,3.97557439992575)(1.74,3.97500757305371)(1.75,3.97414870498555)(1.76,3.9729073602226)(1.77,3.97111766570754)(1.78,3.96847859589943)(1.79,3.96438172670972)(1.8,3.95743129731328)(1.81,3.94371430338377)(1.82,3.90863709460572)(1.83,3.79372627394603)(1.84,3.54548745055331)(1.85,3.22757736997351)(1.86,2.87619537212308)(1.87,2.5013396171124)(1.88,2.10727059871918)(1.89,1.69679094932392)(1.9,1.27228339863608)(1.91,0.836001329041757)(1.92,0.390179518153199)(1.93,-0.0629356146543377)(1.94,-0.521026254287067)(1.95,-0.981756120281515)(1.96,-1.44273409621197)(1.97,-1.90136044803032)(1.98,-2.35490401083541)(1.99,-2.79990676176565)(2,-3.2305920470274)(2.01,-3.62867052844937)(2.02,-3.89576557736583)(2.03,-3.96552502032103)(2.04,-3.98365754701555)(2.05,-3.99126468746347)(2.06,-3.99535180791246)(2.07,-3.99786206740385)(2.08,-3.99953252065669)(2.09,-4.00070188892718)(2.1,-4.00154571064679)(2.11,-4.0021658547918)(2.12,-4.00262134245772)(2.13,-4.00295299699963)(2.14,-4.00318519953306)(2.15,-4.00333520916315)(2.16,-4.00341485521946)(2.17,-4.00342992586564)(2.18,-4.00338608401704)(2.19,-4.00328227310455)(2.2,-4.00311911860636)(2.21,-4.00289023251377)(2.22,-4.00258957267051)(2.23,-4.00220408415669)(2.24,-4.00171639517573)(2.25,-4.00110009413726)(2.26,-4.00031422641895)(2.27,-3.99929759840578)(2.28,-3.99794666313689)(2.29,-3.99608644353575)(2.3,-3.99338390112172)(2.31,-3.98913448296152)(2.32,-3.98156100118376)(2.33,-3.96476632994639)(2.34,-3.9091705493987)(2.35,-3.7007573585669)(2.36,-3.35909201641655)(2.37,-2.98403189352957)(2.38,-2.598822059195)(2.39,-2.21017546022119)(2.4,-1.8216919474041)(2.41,-1.43606314360297)(2.42,-1.05562746116522)(2.43,-0.682530758232515)(2.44,-0.318785293877428)(2.45,0.0337198838371605)(2.46,0.37320936834915)(2.47,0.698009737723375)(2.48,1.00658614561792)(2.49,1.29751641286082)(2.5,1.56951176803375)(2.51,1.82142697628786)(2.52,2.0522835489512)(2.53,2.26119599596279)(2.54,2.44752373964441)(2.55,2.61074766406415)(2.56,2.75055810107345)(2.57,2.86682680595003)(2.58,2.95965431231274)(2.59,3.02930760245649)(2.6,3.07625291990183)(2.61,3.10105755228915)(2.62,3.10439375593399)(2.63,3.08698163393998)(2.64,3.04958601136161)(2.65,2.99307062150241)(2.66,2.9183842487471)(2.67,2.82662959497947)(2.68,2.71904237313953)(2.69,2.59699261400429)(2.7,2.46198053077849)(2.71,2.31558428697354)(2.72,2.15950476179781)(2.73,1.99545282957284)(2.74,1.82522377223819)(2.75,1.6506233659194)(2.76,1.47348350629246)(2.77,1.29563169041207)(2.78,1.11888834870585)(2.79,0.945047145287627)(2.8,0.77586714971108)(2.81,0.61305722659842)(2.82,0.458267211307337)(2.83,0.313073109086843)(2.84,0.178970743629024)(2.85,0.0573592187871169)(2.86,-0.0504601986366084)(2.87,-0.143304097657663)(2.88,-0.220102650377117)(2.89,-0.279917579029781)(2.9,-0.321943075470369)(2.91,-0.345509286516173)(2.92,-0.350095120021768)(2.93,-0.335322818274055)(2.94,-0.300965341237673)(2.95,-0.246951313606394)(2.96,-0.173357125093568)(2.97,-0.0804183870155478)(2.98,0.0314786742063522)(2.99,0.161797335697534)(3,0.309581741813916)(3.01,0.0340487732726786)(3.02,-0.205025556557858)(3.03,-0.408031217535682)(3.04,-0.575309715415641)(3.05,-0.707199875830257)(3.06,-0.804070102574412)(3.07,-0.866332465951639)(3.08,-0.894463682704418)(3.09,-0.88901972364919)(3.1,-0.850648659397866)(3.11,-0.780098620971185)(3.12,-0.67823180773904)(3.13,-0.546024372621848)(3.14,-0.384578133637367)(3.15,-0.195118176642077)(3.16,0.0210013321387267)(3.17,0.262300429392619)(3.18,0.527170563420926)(3.19,0.813875811449024)(3.2,1.12055699338662)(3.21,1.44518228147821)(3.22,1.78566389222503)(3.23,2.13939774759584)(3.24,2.50366403046022)(3.25,2.87421931358192)(3.26,3.24201929012107)(3.27,3.56831679684013)(3.28,3.72803023818579)(3.29,3.77152626454809)(3.3,3.79109506466295)(3.31,3.80418065271934)(3.32,3.81467087709608)(3.33,3.82384225832307)(3.34,3.83222609564339)(3.35,3.84007971066493)(3.36,3.84754153827057)(3.37,3.85469092254348)(3.38,3.86157924618599)(3.39,3.86823659659941)(3.4,3.87468524760047)(3.41,3.88094143336607)(3.42,3.88701398088562)(3.43,3.89291350132898)(3.44,3.89864425431287)(3.45,3.90421235367689)(3.46,3.90962153037639)(3.47,3.91487427330043)(3.48,3.9199745619757)(3.49,3.92492243355362)(3.5,3.92972151031803)(3.51,3.93437051473544)(3.52,3.93887153600168)(3.53,3.94322268568157)(3.54,3.94742332870747)(3.55,3.951470693257)(3.56,3.95536009633653)(3.57,3.95908624488365)(3.58,3.96263820842518)(3.59,3.96600396333256)(3.6,3.96916122859686)(3.61,3.97208071670971)(3.62,3.97471369439151)(3.63,3.97698355047771)(3.64,3.97875995975468)(3.65,3.97980355139202)(3.66,3.97962924579934)(3.67,3.97709314151591)(3.68,3.96878284890448)(3.69,3.93959740908442)(3.7,3.7995062303272)(3.71,3.4690676265901)(3.72,3.08557726543031)(3.73,2.69460861088797)(3.74,2.30565829589076)(3.75,1.92271934977878)(3.76,1.54845047162979)(3.77,1.18499870629796)(3.78,0.834251228217463)(3.79,0.497911886795061)(3.8,0.177528323857216)(3.81,-0.125489006689424)(3.82,-0.409878329068413)(3.83,-0.674524414583432)(3.84,-0.918459102835373)(3.85,-1.14087016974047)(3.86,-1.34110367846109)(3.87,-1.51866561436025)(3.88,-1.67323197210973)(3.89,-1.80462621401446)(3.9,-1.91286189542819)(3.91,-1.99809586354133)(3.92,-2.06062842655508)(3.93,-2.10094600322893)(3.94,-2.11963868127703)(3.95,-2.11744859020185)(3.96,-2.09523360053534)(3.97,-2.05400005473845)(3.98,-1.99483009657777)(3.99,-1.91894932858776)(4,-1.82768595400504)(4.01,-1.72245640687452)(4.02,-1.60477791481793)(4.03,-1.47624403093821)(4.04,-1.33851592405859)(4.05,-1.19331243333278)(4.06,-1.04239824293046)(4.07,-0.887554012695708)(4.08,-0.730599755711351)(4.09,-0.573344054745082)(4.1,-0.417581631913046)(4.11,-0.2650970159642)(4.12,-0.117631839435479)(4.13,0.0231217437259851)(4.14,0.155526921377609)(4.15,0.278023386070371)(4.16,0.389128715608494)(4.17,0.487459129839298)(4.18,0.571729451998806)(4.19,0.640766869065635)(4.2,0.693522705668925)(4.21,0.729066840667062)(4.22,0.746611733739815)(4.23,0.745503282543253)(4.24,0.725230591485871)(4.25,0.685433147404172)(4.26,0.625900447588114)(4.27,0.546569904536981)(4.28,0.447542461918559)(4.29,0.329063030540577)(4.3,0.191539940182605)(4.31,0.0355289358173637)(4.32,-0.138262685746581)(4.33,-0.328976134816862)(4.34,-0.535621927680112)(4.35,-0.757047782753167)(4.36,-0.992008535470793)(4.37,-1.23906869976193)(4.38,-1.49675158321798)(4.39,-1.76334099032074)(4.4,-2.03713645211521)(4.41,-2.31606694430817)(4.42,-2.5980519407931)(4.43,-2.88042123684023)(4.44,-3.15952380900413)(4.45,-3.42845596645526)(4.46,-3.66754676095081)(4.47,-3.82244398548669)(4.48,-3.88079416106481)(4.49,-3.90227127497844)(4.5,-3.91256829611937)(4.51,-3.91853804489646)(4.52,-3.92244289087518)(4.53,-3.92521139512031)(4.54,-3.9272882393627)(4.55,-3.92891056218301)(4.56,-3.93021541266333)(4.57,-3.93128419096871)(4.58,-3.93217124085562)(4.59,-3.93290730079791)(4.6,-3.93351274525691)(4.61,-3.93399774404517)(4.62,-3.93436255978093)(4.63,-3.93460267543047)(4.64,-3.93469992770859)(4.65,-3.93462815413408)(4.66,-3.93433902124486)(4.67,-3.93375692169478)(4.68,-3.93275219771768)(4.69,-3.93109214585138)(4.7,-3.92832311286846)(4.71,-3.92343147055327)(4.72,-3.91367006756362)(4.73,-3.88891319870016)(4.74,-3.79791844437147)(4.75,-3.54237934364119)(4.76,-3.19252815943117)(4.77,-2.80731335115162)(4.78,-2.40029234923352)(4.79,-1.97645320605765)(4.8,-1.53889503586402)(4.81,-1.09019103679175)(4.82,-0.632748140861512)(4.83,-0.168922975311)(4.84,0.298935633927015)(4.85,0.768377728312389)(4.86,1.23709127853478)(4.87,1.70248018296889)(4.88,2.16196220172581)(4.89,2.6125249820907)(4.9,3.04998513437987)(4.91,3.46526427235057)(4.92,3.81693108236335)(4.93,3.97163855804866)(4.94,4.00541872647739)(4.95,4.01660968782791)(4.96,4.02179447552785)(4.97,4.02461576041842)(4.98,4.02627561484713)(4.99,4.02727965946923)(5,4.02787555386246)(5.01,4.0281995661668)(5.02,4.02833007404276)(5.03,4.02831718498357)(5.04,4.02819311112507)(5.05,4.02797744468494)(5.06,4.02768559456435)(5.07,4.02732376849594)(5.08,4.02689861744627)(5.09,4.0264091315906)(5.1,4.02585471264979)(5.11,4.02522864125713)(5.12,4.02452146750545)(5.13,4.02371813049217)(5.14,4.02279462805983)(5.15,4.02171755956169)(5.16,4.02043155798755)(5.17,4.01885300231598)(5.18,4.01683645976764)(5.19,4.01412314639146)(5.2,4.01019448792321)(5.21,4.00385656155612)(5.22,3.99169637654236)(5.23,3.96002059502972)(5.24,3.84006165680856)(5.25,3.55852890278542)(5.26,3.21424649098218)(5.27,2.85345945419583)(5.28,2.48700291675303)(5.29,2.11940025633701)(5.3,1.75364039362604)(5.31,1.39216066420488)(5.32,1.03714123388704)(5.33,0.690600021030515)(5.34,0.354424533474344)(5.35,0.0303833947815149)(5.36,-0.279864275815328)(5.37,-0.574794804322641)(5.38,-0.852999227300315)(5.39,-1.11320878933334)(5.4,-1.35429950966416)(5.41,-1.57528246372838)(5.42,-1.77533032032666)(5.43,-1.95378849185572)(5.44,-2.11013393976284)(5.45,-2.24405853893071)(5.46,-2.35535639537946)(5.47,-2.44406552403271)(5.48,-2.51036302480217)(5.49,-2.55452670639992)(5.5,-2.57704553577236)(5.51,-2.57848900647973)(5.52,-2.55958250488054)(5.53,-2.52112194500554)(5.54,-2.46406762133685)(5.55,-2.38946819306695)(5.56,-2.2985038518075)(5.57,-2.19246108421745)(5.58,-2.07273906650249)(5.59,-1.94083005562451)(5.6,-1.79831669108754)(5.61,-1.6468433393309)(5.62,-1.48813282695454)(5.63,-1.32392906941905)(5.64,-1.1560207619805)(5.65,-0.986220488213755)(5.66,-0.816324633191247)(5.67,-0.648137307024643)(5.68,-0.483433072443394)(5.69,-0.323953447328043)(5.7,-0.171392844676838)(5.71,-0.0273847508868877)(5.72,0.106506690825202)(5.73,0.228801928324967)(5.74,0.338109504558569)(5.75,0.433146794594044)(5.76,0.512734843020552)(5.77,0.575820113497867)(5.78,0.621474038775573)(5.79,0.648896832357957)(5.8,0.65743474438628)(5.81,0.646571486986507)(5.82,0.61593849398913)(5.83,0.565320583215701)(5.84,0.494651571386061)(5.85,0.404019845902142)(5.86,0.293672312829595)(5.87,0.164001021407968)(5.88,0.0155627559957647)(5.89,-0.150947920942047)(5.9,-0.334677945974999)(5.91,-0.534639806627441)(5.92,-0.749704367684184)(5.93,-0.978603216980263)(5.94,-1.21996138242816)(5.95,-1.47223276535958)(5.96,-1.73382898710593)(5.97,-2.0029147636414)(5.98,-2.27765376742948)(5.99,-2.55587496845368)(6,-2.83517083221982) 
};

\addplot [
color=red,
dashed,
line width=1.0pt
]
coordinates{
 (0,3.56853312836059)(0.01,3.76977208012353)(0.02,3.89179784586829)(0.03,3.94016771950718)(0.04,3.95980814724261)(0.05,3.96961637779295)(0.06,3.97531372209009)(0.07,3.9789552065215)(0.08,3.98142651609247)(0.09,3.98316421576158)(0.1,3.98440489404198)(0.11,3.98528516333906)(0.12,3.98588658608309)(0.13,3.98625744170371)(0.14,3.98642364126957)(0.15,3.98639384501193)(0.16,3.98616083400905)(0.17,3.98569958064737)(0.18,3.98496117345604)(0.19,3.98385970214372)(0.2,3.98224502873585)(0.21,3.97984298023221)(0.22,3.976109912629)(0.23,3.96982188185159)(0.24,3.95764032287669)(0.25,3.92748301743351)(0.26,3.82760141386315)(0.27,3.58944122346359)(0.28,3.27220029369898)(0.29,2.91947099148002)(0.3,2.54290372259659)(0.31,2.14711888146853)(0.32,1.73503206004528)(0.33,1.30906964274521)(0.34,0.871504707379234)(0.35,0.424583337629709)(0.36,-0.0294320171997258)(0.37,-0.488258926828607)(0.38,-0.949490470558962)(0.39,-1.41079305238463)(0.4,-1.86958385587183)(0.41,-2.32310191503198)(0.42,-2.76797467532158)(0.43,-3.19866969803441)(0.44,-3.59839677943758)(0.45,-3.87641164632719)(0.46,-3.95277441044482)(0.47,-3.97211399377589)(0.48,-3.9800953792982)(0.49,-3.98435893646938)(0.5,-3.98697805391295)(0.51,-3.9887289837429)(0.52,-3.98996504327478)(0.53,-3.99086921763543)(0.54,-3.99154526298864)(0.55,-3.9920559877829)(0.56,-3.99244123724024)(0.57,-3.99272707973249)(0.58,-3.99293080352966)(0.59,-3.99306375906201)(0.6,-3.99313300847305)(0.61,-3.99314226534167)(0.62,-3.99309236722754)(0.63,-3.99298139498894)(0.64,-3.99280446844542)(0.65,-3.99255317388162)(0.66,-3.99221449321538)(0.67,-3.99176895978539)(0.68,-3.99118751277765)(0.69,-3.99042598722477)(0.7,-3.98941501033)(0.71,-3.98804027809374)(0.72,-3.98610064337263)(0.73,-3.98320881614116)(0.74,-3.97851781277718)(0.75,-3.96978728480495)(0.76,-3.94898696369142)(0.77,-3.87300534815167)(0.78,-3.62536111962799)(0.79,-3.27329983622353)(0.8,-2.89599461033791)(0.81,-2.51024246097781)(0.82,-2.12168577091083)(0.83,-1.7336395017872)(0.84,-1.34870717620872)(0.85,-0.969187851243355)(0.86,-0.597202881233504)(0.87,-0.234749070142706)(0.88,0.116290560859089)(0.89,0.454155499761723)(0.9,0.777180004730529)(0.91,1.08383943316513)(0.92,1.37271895980002)(0.93,1.64254046427575)(0.94,1.89217432121816)(0.95,2.12063341193389)(0.96,2.32706195032777)(0.97,2.51082965889214)(0.98,2.67141289899435)(0.99,2.80853613726118)(1,2.92208824423103)(1.01,3.01219484215047)(1.02,3.07914987626905)(1.03,3.12345753755892)(1.04,3.14570244352216)(1.05,3.14657369359797)(1.06,3.12678938464987)(1.07,3.08710809079257)(1.08,3.02837669088025)(1.09,2.95154686113637)(1.1,2.85771204342227)(1.11,2.7481150487877)(1.12,2.62413976404365)(1.13,2.48728098019088)(1.14,2.33915342144411)(1.15,2.18142907757133)(1.16,2.01585623228632)(1.17,1.8442198955754)(1.18,1.66833700296179)(1.19,1.49003698277658)(1.2,1.31115143969189)(1.21,1.13349838069709)(1.22,0.958871280557697)(1.23,0.789025783148397)(1.24,0.625667736891175)(1.25,0.470441516421612)(1.26,0.324918915531104)(1.27,0.190586966865729)(1.28,0.0688411843036124)(1.29,-0.0390304586091213)(1.3,-0.131847461501472)(1.31,-0.208554219864287)(1.32,-0.268220331416426)(1.33,-0.310048507188366)(1.34,-0.333384475876839)(1.35,-0.337713845340831)(1.36,-0.322672263321515)(1.37,-0.288045270742109)(1.38,-0.233770635858005)(1.39,-0.159938751351269)(1.4,-0.0667977110406535)(1.41,0.0452591889659366)(1.42,0.175676613120978)(1.43,0.323762275302526)(1.44,0.488667516898878)(1.45,0.669411171223631)(1.46,0.864872897404087)(1.47,1.07379674310019)(1.48,1.29481930443605)(1.49,1.52641984218154)(1.5,1.76702550211533)(1.51,2.01484772845549)(1.52,2.26809729003107)(1.53,2.52468421512374)(1.54,2.78240778076758)(1.55,3.03855871604913)(1.56,3.28939288079497)(1.57,3.52807905518332)(1.58,3.73703158978049)(1.59,3.87533801257205)(1.6,3.93365901113924)(1.61,3.95669214296662)(1.62,3.96780581352209)(1.63,3.97411347849878)(1.64,3.97808411871898)(1.65,3.9807513086963)(1.66,3.98261369885979)(1.67,3.98393730124849)(1.68,3.98487401610818)(1.69,3.98551382636886)(1.7,3.98590980976664)(1.71,3.98609054257225)(1.72,3.98606592517949)(1.73,3.98582882439035)(1.74,3.98535310142994)(1.75,3.98458709204689)(1.76,3.98343946477207)(1.77,3.98174946755963)(1.78,3.97922106730189)(1.79,3.97526044585106)(1.8,3.96850880618231)(1.81,3.95517114507634)(1.82,3.92112418953135)(1.83,3.80861275543517)(1.84,3.56072902579834)(1.85,3.2414681413322)(1.86,2.88845223665124)(1.87,2.51193967873517)(1.88,2.11624223174281)(1.89,1.70417494037895)(1.9,1.27812431760203)(1.91,0.840344651106538)(1.92,0.393071017878194)(1.93,-0.0614515432185506)(1.94,-0.520902052126319)(1.95,-0.982947664266122)(1.96,-1.44519072203272)(1.97,-1.90502845589594)(1.98,-2.35971396374123)(1.99,-2.80573550112787)(2,-3.23710599476912)(2.01,-3.63409771594679)(2.02,-3.89166445742461)(2.03,-3.95626972085773)(2.04,-3.97349216864431)(2.05,-3.98085061463511)(2.06,-3.98485370543765)(2.07,-3.98733986685468)(2.08,-3.98901407739562)(2.09,-3.99020223255608)(2.1,-3.99107495810226)(2.11,-3.99172978055129)(2.12,-3.99222610991888)(2.13,-3.9926018310804)(2.14,-3.99288184695306)(2.15,-3.99308275460451)(2.16,-3.99321552069793)(2.17,-3.99328704507891)(2.18,-3.99330105994389)(2.19,-3.99325859425622)(2.2,-3.99315811102811)(2.21,-3.99299535205635)(2.22,-3.9927628556395)(2.23,-3.99244904353514)(2.24,-3.99203665041318)(2.25,-3.99150006853874)(2.26,-3.99080075632999)(2.27,-3.98987896458506)(2.28,-3.98863793350935)(2.29,-3.98691126998287)(2.3,-3.98438849898226)(2.31,-3.98042037158484)(2.32,-3.97340268468551)(2.33,-3.95818510412143)(2.34,-3.90993603595416)(2.35,-3.72062887147795)(2.36,-3.38533279982783)(2.37,-3.01175869379213)(2.38,-2.627147348254)(2.39,-2.23884537209096)(2.4,-1.85060198991471)(2.41,-1.46515953630395)(2.42,-1.08487513725486)(2.43,-0.711903924993536)(2.44,-0.348262503864676)(2.45,0.00415735110491004)(2.46,0.343578937260455)(2.47,0.668328193470347)(2.48,0.976870288697807)(2.49,1.26778284420175)(2.5,1.53977863399115)(2.51,1.79171373612712)(2.52,2.02261285693532)(2.53,2.23158600861228)(2.54,2.41800023551282)(2.55,2.58133803342155)(2.56,2.72128600269125)(2.57,2.83771090579958)(2.58,2.93070668513258)(2.59,3.00053298676719)(2.6,3.0476284341184)(2.61,3.07255596626009)(2.62,3.07597409748234)(2.63,3.05860102431309)(2.64,3.02121947216765)(2.65,2.96469653478988)(2.66,2.89000325763717)(2.67,2.79824981462972)(2.68,2.69067905896305)(2.69,2.56866448806349)(2.7,2.4337065540613)(2.71,2.28738891795648)(2.72,2.13139299273238)(2.73,1.96744499258929)(2.74,1.79733139890911)(2.75,1.62285798920026)(2.76,1.44585406637984)(2.77,1.26814673388842)(2.78,1.09155470061092)(2.79,0.917871556698475)(2.8,0.748854913781424)(2.81,0.586213982347618)(2.82,0.431596981201796)(2.83,0.286580815181874)(2.84,0.152659454910401)(2.85,0.0312330196658277)(2.86,-0.0763986076194295)(2.87,-0.169051890909735)(2.88,-0.245657065173333)(2.89,-0.305277072575938)(2.9,-0.347104800672836)(2.91,-0.370472370492746)(2.92,-0.374857368746771)(2.93,-0.359883090776164)(2.94,-0.325322548151428)(2.95,-0.27110335659986)(2.96,-0.197303787437494)(2.97,-0.104157126249056)(2.98,0.00794819127199006)(2.99,0.138477124567516)(3,0.286723685426516)(3.01,-3.9683261287621)(3.02,-3.97181117489699)(3.03,-3.97393083079651)(3.04,-3.97510911954804)(3.05,-3.97554312949746)(3.06,-3.97529392562418)(3.07,-3.97431066185422)(3.08,-3.97240727703676)(3.09,-3.96917080339482)(3.1,-3.96370939283604)(3.11,-3.95390768869083)(3.12,-3.9338425292765)(3.13,-3.88403471406298)(3.14,-3.75773077868349)(3.15,-3.54208362260676)(3.16,-3.27352931521287)(3.17,-2.97164937108085)(3.18,-2.64372329211401)(3.19,-2.29344459501088)(3.2,-1.92342576286841)(3.21,-1.53593455372995)(3.22,-1.13313609797322)(3.23,-0.717190416302752)(3.24,-0.290289844679975)(3.25,0.145347144254757)(3.26,0.587354937529879)(3.27,1.03349253447534)(3.28,1.48122408639098)(3.29,1.92805171756394)(3.3,2.37118089421364)(3.31,2.80710621475923)(3.32,3.2300118877676)(3.33,3.62153634725558)(3.34,3.88468316254533)(3.35,3.95436517062647)(3.36,3.9726672886186)(3.37,3.98038333725688)(3.38,3.98455108316232)(3.39,3.98712919443297)(3.4,3.98886162738896)(3.41,3.99009014655812)(3.42,3.99099292020137)(3.43,3.99167150551054)(3.44,3.9921876729724)(3.45,3.9925808151381)(3.46,3.99287688200948)(3.47,3.99309326923481)(3.48,3.99324161952407)(3.49,3.9933294708539)(3.5,3.99336122541419)(3.51,3.99333868181279)(3.52,3.9932612548032)(3.53,3.99312593037071)(3.54,3.99292694956042)(3.55,3.99265515584702)(3.56,3.99229685850379)(3.57,3.99183193126401)(3.58,3.99123060078649)(3.59,3.99044785459713)(3.6,3.98941322014215)(3.61,3.98801082144354)(3.62,3.98603715983991)(3.63,3.98310117146191)(3.64,3.97835003621112)(3.65,3.96953855083544)(3.66,3.94869880150403)(3.67,3.87406005325575)(3.68,3.63398840396657)(3.69,3.2918377105877)(3.7,2.92575446082401)(3.71,2.55274679162196)(3.72,2.17848438522829)(3.73,1.80624308155533)(3.74,1.43856498232181)(3.75,1.07766624589043)(3.76,0.725575375584433)(3.77,0.384178951615298)(3.78,0.0552446886822814)(3.79,-0.259581446938258)(3.8,-0.558784472703361)(3.81,-0.840967581192128)(3.82,-1.10487587726269)(3.83,-1.34939631892618)(3.84,-1.57355805003081)(3.85,-1.77654566788039)(3.86,-1.95771909517378)(3.87,-2.11657192059833)(3.88,-2.25280851115724)(3.89,-2.36624028216918)(3.9,-2.45692881680474)(3.91,-2.52504761654521)(3.92,-2.57090204431927)(3.93,-2.59500405551079)(3.94,-2.59793040698393)(3.95,-2.58042289483275)(3.96,-2.54329718204106)(3.97,-2.48752761635901)(3.98,-2.41417043770604)(3.99,-2.3244217409292)(4,-2.21957668134459)(4.01,-2.10104310872739)(4.02,-1.97032467715352)(4.03,-1.82900780021157)(4.04,-1.67875003661421)(4.05,-1.52127281011849)(4.06,-1.35832871236113)(4.07,-1.19171682563865)(4.08,-1.02323851817605)(4.09,-0.854705546260733)(4.1,-0.687918337310194)(4.11,-0.524651539943346)(4.12,-0.366647636215948)(4.13,-0.215597468893915)(4.14,-0.0731329976747045)(4.15,0.0591861812130546)(4.16,0.179882418808105)(4.17,0.287574272132994)(4.18,0.380978561234086)(4.19,0.458928961491648)(4.2,0.520376694759341)(4.21,0.564397713031802)(4.22,0.5902081539875)(4.23,0.597156166918941)(4.24,0.584737628889665)(4.25,0.552594125700553)(4.26,0.500519897301436)(4.27,0.428457226933623)(4.28,0.33651053858067)(4.29,0.224928275091101)(4.3,0.0941233007449934)(4.31,-0.0553463960188052)(4.32,-0.222769258427149)(4.33,-0.40728525832406)(4.34,-0.607899197217276)(4.35,-0.823460665946959)(4.36,-1.05271835679831)(4.37,-1.29423800691145)(4.38,-1.54654399253727)(4.39,-1.80791664882432)(4.4,-2.07666628364921)(4.41,-2.35074030673729)(4.42,-2.6280799460741)(4.43,-2.90611439739685)(4.44,-3.18144733084212)(4.45,-3.44810778122091)(4.46,-3.69052118663126)(4.47,-3.86228305844865)(4.48,-3.93344115130557)(4.49,-3.95875589541949)(4.5,-3.97020839267479)(4.51,-3.97650481657765)(4.52,-3.98041076379981)(4.53,-3.98302554020946)(4.54,-3.98486347387905)(4.55,-3.98619444244732)(4.56,-3.98717204651903)(4.57,-3.98788874752299)(4.58,-3.98840205575357)(4.59,-3.98874789102271)(4.6,-3.98894766797914)(4.61,-3.98901199776156)(4.62,-3.98894230884595)(4.63,-3.9887308785682)(4.64,-3.98835925986941)(4.65,-3.98779451923993)(4.66,-3.98698179190365)(4.67,-3.98582976788414)(4.68,-3.98418090150778)(4.69,-3.98174506849415)(4.7,-3.97793249071929)(4.71,-3.9713554064545)(4.72,-3.95792703241846)(4.73,-3.920715466694)(4.74,-3.78241758227419)(4.75,-3.48865366893298)(4.76,-3.12932286838691)(4.77,-2.7405297030772)(4.78,-2.331363992552)(4.79,-1.90589580815271)(4.8,-1.46696025897242)(4.81,-1.01703509392048)(4.82,-0.558483955347998)(4.83,-0.0936474929230403)(4.84,0.375164038638615)(4.85,0.845411326683367)(4.86,1.31487816560348)(4.87,1.78087908108871)(4.88,2.24077056541483)(4.89,2.69132309623339)(4.9,3.12759089734937)(4.91,3.53637334473032)(4.92,3.84712336648399)(4.93,3.94668310598619)(4.94,3.96997992530672)(4.95,3.97901206277205)(4.96,3.98368543222494)(4.97,3.98650312251321)(4.98,3.98836390389259)(4.99,3.98966615587465)(5,3.99061247951624)(5.01,3.99131623763394)(5.02,3.99184537220903)(5.03,3.99224264788988)(5.04,3.99253586874116)(5.05,3.99274337895843)(5.06,3.99287716778773)(5.07,3.99294466163418)(5.08,3.99294974017285)(5.09,3.99289324825027)(5.1,3.99277312731668)(5.11,3.99258420347346)(5.12,3.99231758365101)(5.13,3.99195952539706)(5.14,3.99148948723365)(5.15,3.99087680050074)(5.16,3.9900748347166)(5.17,3.98901028790617)(5.18,3.98756223895152)(5.19,3.98551760683238)(5.2,3.98246536821309)(5.21,3.97750477373665)(5.22,3.96824903077656)(5.23,3.94616858606679)(5.24,3.86738383042669)(5.25,3.6269400711034)(5.26,3.29140897873612)(5.27,2.93262517958019)(5.28,2.56665988840322)(5.29,2.19905868342188)(5.3,1.83308862215383)(5.31,1.47129201087247)(5.32,1.11589174024812)(5.33,0.768927026195697)(5.34,0.432297048544918)(5.35,0.107775057597465)(5.36,-0.202975318072866)(5.37,-0.498427500570645)(5.38,-0.77717267524122)(5.39,-1.03794194520135)(5.4,-1.27960954196911)(5.41,-1.50118592804992)(5.42,-1.70185151376491)(5.43,-1.88093449214622)(5.44,-2.03793088048079)(5.45,-2.17249628102501)(5.46,-2.28446284847649)(5.47,-2.37385017908888)(5.48,-2.44077892080709)(5.49,-2.48558586785614)(5.5,-2.50871348180619)(5.51,-2.51075120815912)(5.52,-2.49239747403564)(5.53,-2.45448308233466)(5.54,-2.39795879844874)(5.55,-2.32388274494358)(5.56,-2.23344375001894)(5.57,-2.12793049965429)(5.58,-2.00874515948324)(5.59,-1.8773853053595)(5.6,-1.73542571010989)(5.61,-1.5845244129485)(5.62,-1.42638717933147)(5.63,-1.2627662923814)(5.64,-1.09545389495605)(5.65,-0.926244872437121)(5.66,-0.756951242414601)(5.67,-0.589368914345873)(5.68,-0.425272118862172)(5.69,-0.266403624852888)(5.7,-0.114454870929924)(5.71,0.0289398423072515)(5.72,0.162218064963469)(5.73,0.283898346905596)(5.74,0.392594632769731)(5.75,0.487018245413195)(5.76,0.565996856252533)(5.77,0.62847510661519)(5.78,0.673521063356228)(5.79,0.70034350382887)(5.8,0.708281575560738)(5.81,0.69682213803453)(5.82,0.665597024875878)(5.83,0.614390668221194)(5.84,0.543134436515845)(5.85,0.45192230691664)(5.86,0.340991977287515)(5.87,0.210744870293334)(5.88,0.0617289649196939)(5.89,-0.105355546716676)(5.9,-0.289658625119358)(5.91,-0.490195772804808)(5.92,-0.705827247640711)(5.93,-0.935310927656176)(5.94,-1.17722715141822)(5.95,-1.43010696748922)(5.96,-1.69225742201832)(5.97,-1.96199880726537)(5.98,-2.23732741564465)(5.99,-2.51626139368441)(6,-2.79637422610687) 
};

\end{axis}
\end{tikzpicture}